\documentclass{amsart}
\usepackage{array}
\usepackage{tabularx}
\usepackage{amsmath,amssymb}
\usepackage{mathrsfs}
\usepackage{amsthm}
\usepackage{comment}
\usepackage[all]{xy}

\usepackage[
bookmarks=false,
colorlinks=true,
debug=true,
naturalnames=true,
pdfnewwindow=true,
citecolor=blue,
linkcolor=blue,
urlcolor = blue]{hyperref}

\theoremstyle{definition}
\newtheorem{Def}{Definition}[section]
\newtheorem{Thm}[Def]{Theorem}
\newtheorem{Lem}[Def]{Lemma}
\newtheorem{Prop}[Def]{Proposition}
\newtheorem{Cor}[Def]{Corollary}
\newtheorem{Rem}[Def]{Remark}

\newtheorem{theo}{Theorem}

\newcommand{\C}{\mathbb{C}}
\newcommand{\Cc}{\mathcal{C}}
\newcommand{\Irr}{\mathop{\mathsf{Irr}}\nolimits}
\newcommand{\g}{\mathfrak{g}}
\newcommand{\h}{\mathfrak{h}}
\newcommand{\Z}{\mathbb{Z}}
\newcommand{\kk}{\Bbbk}
\newcommand{\modfg}{\text{-}\mathrm{mod}_{\mathrm{fg}}}
\newcommand{\modfd}{\text{-}\mathrm{mod}_{\mathrm{fd}}}
\newcommand{\std}{\Delta}
\newcommand{\pcstd}{\bar{\nabla}}
\newcommand{\pstd}{\bar{\Delta}}
\newcommand{\Ext}{\mathop{\mathrm{Ext}}\nolimits}
\newcommand{\End}{\mathop{\mathrm{End}}\nolimits}
\newcommand{\Hom}{\mathop{\mathrm{Hom}}\nolimits}

\newcommand{\Ann}{\mathop{\mathrm{Ann}}\nolimits}
\newcommand{\rad}{\mathop{\text{rad}}}

\newcommand{\Ker}{\mathop{\text{Ker}}}
\newcommand{\Ima}{\mathop{\text{Im}}}

\newcommand{\Q}{\mathbb{Q}}
\newcommand{\Stab}{\mathop{\mathrm{Stab}}\nolimits}
\newcommand{\Specm}{\mathop{\mathrm{Specm}}\nolimits}
\newcommand{\KP}{\mathrm{KP}}
\newcommand{\M}{\mathfrak{M}}
\newcommand{\G}{\mathbb{G}}
\newcommand{\hlam}{\hat{\lambda}}
\newcommand{\hnu}{\hat{\nu}}
\newcommand{\cP}{\mathsf{P}}
\newcommand{\cQ}{\mathsf{Q}}
\newcommand{\lP}{\mathcal{P}}
\newcommand{\lQ}{\mathcal{Q}}
\newcommand{\K}{\mathcal{K}}
\newcommand{\T}{\mathbb{T}}
\newcommand{\V}{\mathcal{V}}
\newcommand{\W}{\mathcal{W}}
\newcommand{\hK}{\widehat{\mathcal{K}}}
\newcommand{\N}{\mathfrak{N}}
\newcommand{\LL}{\mathfrak{L}}
\newcommand{\hmu}{\hat{\mu}}
\newcommand{\Mregg}{\M_{0}^{\bullet \, \mathrm{reg}}}
\newcommand{\Mreg}{\M_{0}^{\mathrm{reg}}}
\newcommand{\SG}{\mathfrak{S}}
\newcommand{\pt}{\mathsf{p}}
\newcommand{\R}{\mathcal{R}}
\newcommand{\hW}{\widehat{W}}
\newcommand{\hR}{\widehat{\R}}
\newcommand{\rr}{\mathfrak{r}}
\newcommand{\tilU}{\widetilde{U}_{q}}
\newcommand{\bW}{\mathbb{W}}
\newcommand{\hCc}{\widehat{\Cc}}
\newcommand{\hS}{\widehat{\mathbb{S}}}
\newcommand{\hH}{\widehat{H}}
\newcommand{\Mm}{\mathcal{M}}
\newcommand{\hMm}{\widehat{\mathcal{M}}}
\newcommand{\HH}{\mathbb{H}}
\newcommand{\tZ}{\widetilde{\mathfrak{Z}}}
\newcommand{\Aa}{\mathcal{A}}

\numberwithin{equation}{section}

\newcommand{\arxiv}[1]{\href{http://arxiv.org/abs/#1}{\texttt{arXiv:#1}}}

\title[Affine highest weight categories and 
Schur-Weyl duality]%
{Affine highest weight categories 
and quantum affine Schur-Weyl duality
of Dynkin quiver types}
\author[R.~Fujita]{Ryo Fujita}
\address{Institut de Math\'{e}matiques de Jussieu-Paris Rive Gauche, IMJ-PRG, Universit\'{e} de Paris, B\^{a}timent Sophie Germain, F-75013, Paris, France}
\email{ryo.fujita@imj-prg.fr}

\begin{document}

\maketitle

\begin{abstract}
For a Dynkin quiver $Q$ (of type $\mathrm{ADE}$),
we consider 
a central completion of
the convolution algebra of
the equivariant $K$-group
of a certain Steinberg type graded quiver variety.
We observe that 
it is affine quasi-hereditary 
and prove that its category of  finite-dimensional modules
is identified with a block of
Hernandez-Leclerc\rq{}s monoidal category $\Cc_{Q}$
of modules over the quantum loop algebra $U_{q}(L\g)$
via Nakajima\rq{}s homomorphism.
As an application, we show that Kang-Kashiwara-Kim\rq{}s
generalized quantum affine Schur-Weyl duality functor 
gives an equivalence
between the category of finite-dimensional modules  
over the quiver Hecke algebra 
associated with $Q$
and Hernandez-Leclerc\rq{}s category $\Cc_{Q}$,
assuming the simpleness of some poles of 
normalized $R$-matrices 
for type $\mathrm{E}$.
\end{abstract}


\section*{Introduction}  
		
\subsection{}	
For a Dynkin quiver $Q$ (of type $\mathrm{ADE}$),
one can associate the following two interesting monoidal categories
$\Cc_{Q}$ and $\Mm_{Q}$. 		
			
The first one $\Cc_{Q}$ is
a certain monoidal subcategory 
of the category of finite-dimensional modules
over the quantum loop algebra $U_{q}(L\g)$,
where $\g$ is the complex simple Lie algebra
whose Dynkin diagram is the underlying graph of 
the quiver $Q$. 
This category was introduced by Hernandez-Leclerc~\cite{HL15}.
The definition of the category $\Cc_{Q}$ involves 
the Auslander-Reiten quiver of the path algebra of $Q$.
The complexified Grothendieck ring of $\Cc_{Q}$
is known to be isomorphic to
the coordinate algebra $\C[N]$ of
the unipotent group $N$ 
associated with the positive part of $\g$, under which
 the classes of simple modules 
correspond to the dual canonical basis elements.
Note that we have a decomposition 
$\Cc_{Q} = \bigoplus_{\beta \in \cQ^{+}} \Cc_{Q, \beta}$,
which corresponds to 
the weight decomposition 
$\C[N] = \bigoplus_{\beta \in \cQ^{+}} \C[N]_{\beta}$.

The second one $\Mm_{Q}$ is 
the direct sum of the
categories $\Mm_{Q, \beta} =
H_{Q}(\beta) \modfd^{0}$ 
of finite-dimensional modules 
over the quiver Hecke algebra $H_{Q}(\beta)$
on which the center acts nilpotently.
The monoidal structure of $\Mm_{Q}$ is given by an analog of
parabolic inductions. 
The quiver Hecke algebra 
(which is also known as the Khovanov-Lauda-Rouquier algebra)
was introduced by Khovanov-Lauda~\cite{KL09}
and by Rouquier~\cite{Rouquier08}
as an algebraic object which
generalizes the affine Hecke algebra of type $\mathrm{A}$
in the sense that 
it gives a categorification of the dual of the integral form $U_{q}(\g)^{+}_{\Z}$ of the positive half 
of the quantized enveloping algebra
$U_{q}(\g)$.			 
More precisely, the quiver Hecke algebra 
$H_{Q}(\beta)$ is equipped with a $\Z$-grading
and hence the direct sum of the Grothendieck groups 
of the categories of finite-dimensional graded modules
over $H_{Q}(\beta)$ for various $\beta$
becomes a $\Z[q^{\pm 1}]$-algebra,
where $q$ corresponds to the grading shift. 
It is isomorphic to the dual of $U_{q}(\g)^{+}_{\Z}$
with the classes of self-dual simple modules
corresponding to the dual canonical basis elements. 
Our category $\Mm_{Q}$ is obtained 
by forgetting the gradings, which corresponds     
to specializing $q$ to $1$ at the level of
the Grothendieck ring. 
Therefore the complexified Grothendieck ring 
of the monoidal category $\Mm_{Q}$ is
also isomorphic to $\C[N]$. 

Thus, we encounter a natural question, 
originally asked by Hernandez-Leclerc~\cite{HL15}, 
whether there is any functorial relationship between
these two monoidal categories $\Cc_{Q}$ and $\Mm_{Q}$.    
Kang-Kashiwara-Kim~\cite{KKK15} gave an elegant answer to this question
by constructing the 
{\em generalized quantum affine Schur-Weyl duality functor}
$\mathcal{F}_{Q} \colon \Mm_{Q} \to \Cc_{Q}$
for any quiver $Q$ of type $\mathrm{AD}$.
The functor $\mathcal{F}_{Q}$
is a direct sum $\bigoplus_{\beta \in \cQ^{+}} \mathcal{F}_{Q, \beta}$
of functors $\mathcal{F}_{Q, \beta}\colon \Mm_{Q, \beta} \to \Cc_{Q, \beta}$, 
where $\mathcal{F}_{Q, \beta}$ is given by
a certain $(U_{q}(L \g), H_{Q}(\beta))$-bimodule
constructed by using the normalized $R$-matrices
for $\ell$-fundamental modules of $U_{q}(L\g)$.
Here the $\ell$-fundamental modules are quantum loop analogs of the fundamental modules of $\g$,
and the normalized $R$-matrices are certain intertwining operators between tensor product modules.
A normalized $R$-matrix can be seen as a matrix-valued rational function, 
whose singularity is strongly related to the structure of tensor product modules. 
Kang-Kashiwara-Kim's construction also works for type $\mathrm{E}$
if we assume the simpleness of some specific poles of
normalized $R$-matrices.\footnote{After the initial submission of this paper, it was also proved that this assumption is always satisfied for type $\mathrm{E}$ by Oh-Scrimshaw~\cite{OS19} and by the author~\cite{Fujita20} independently.}
Moreover, under this assumption,
it was proved in \cite{KKK15} that
the functor $\mathcal{F}_{Q}$ is an exact monoidal functor which  
induces an isomorphism between the Grothendieck rings
for any quiver $Q$ of type $\mathrm{ADE}$.

When the quiver $Q$ is 
an equioriented quiver of type $\mathrm{A}$,
Kang-Kashiwara-Kim\rq{}s construction coincides with
that of the usual quantum affine Schur-Weyl duality 
between $U_{q}(L\mathfrak{sl}_{n})$ and 
the affine Hecke algebra of type $\mathrm{A}$.
In particular, the functor 
$\mathcal{F}_{Q}$ gives an equivalence of categories 
in this special case thanks to Chari-Pressley~\cite{CP96}.
Thus it is natural to expect that 
the functor $\mathcal{F}_{Q}$ also gives an equivalence 
of categories
for general $Q$ of type $\mathrm{ADE}$. 
The goal of this paper is to verify this expectation.   

\subsection{}
We briefly explain our strategy.
An isomorphism between Grothendieck rings 
does not imply an actual equivalence of categories
because by passing to Grothendieck rings we forget some homological information     
such as extensions among modules.
Note that, in our case, both categories $\Cc_Q$ and $\Mm_Q$ are far from semi-simple.
Therefore, in order to verify that 
the functor $\mathcal{F}_{Q}$ gives an equivalence,
one should ask whether it respects 
the homological properties.      
However, 
the categories 
$\Cc_{Q}$ and $\Mm_{Q}$ 
are not suitable for a homological study
because they have no projective modules.\footnote{
In fact, for any simple module $L$ in $\Cc_Q$ (or $\Mm_Q$), 
we can construct a module of arbitrary length with a simple head $\simeq L$
as a finite-dimensional quotient of the affinization of $L$ (see Remark \ref{Rem:aff}),
which implies that we do not have a projective cover of $L$ in the category $\Cc_Q$ (nor $\Mm_Q$).}
Thus we want to discuss some suitable \lq\lq{}completions\rq\rq{}
$\hCc_{Q}$ and $\hMm_{Q}$
of categories $\Cc_{Q}$ and $\Mm_{Q}$ respectively
to get enough projective modules.

For the category $\Mm_{Q}$, 
the desired completion is easily obtained.
Namely, we just take the central completion $\hH_{Q}(\beta)$ 
of the quiver Hecke algebra along the trivial central character
and define the category $\hMm_{Q, \beta}$ as 
the category of finitely generated $\hH_{Q}(\beta)$-modules.
The homological properties of the category 
$\hMm_{Q, \beta}$ are well-understood,
due to Kato~\cite{Kato14}
and Brundan-Kleshchev-McNamara~\cite{BKM14}.
More precisely, the category $\hMm_{Q, \beta}$ has a structure
of {\em affine highest weight category}, which is
equivalent to saying that the algebra $\hH_{Q}(\beta)$
is an {\em affine quasi-hereditary algebra}.  
The notion of affine highest weight category
was axiomatized by Kleshchev~\cite{Kleshchev15}
as a generalization of the notion of
highest weight category introduced by  
Cline-Parshall-Scott~\cite{CPS88}.
In particular, an affine highest weight category has
{\em standard modules}, which
filter projective modules.  

On the other hand, it is not obvious how to define a suitable completion 
$\hCc_{Q, \beta}$ because a priori we have  
no information about the center of the category $\Cc_{Q, \beta}$.
To remedy this situation, we rely on
Nakajima\rq{}s construction~\cite{Nakajima01}
of $U_{q}(L\g)$-modules using convolution 
algebras of equivariant $K$-groups of quiver varieties.
For us, an advantage  
of this kind of convolution construction is that
the representation ring of the group of equivariance 
automatically gives a large central subalgebra of
the convolution algebra. 
For a Dynkin quiver $Q$
and an element $\beta \in \cQ^{+}$, we have the corresponding graded quiver varieties 
$\M_{\beta}^{\bullet}, \M_{0, \beta}^{\bullet}$
with actions of a linear algebraic group
$\G(\beta)$ determined by $\beta$.
We remark that 
our $\M^{\bullet}_{0, \beta}$ is the same graded quiver variety as  
Hernandez-Leclerc studied in \cite{HL15}, where
it was proved that $\M^{\bullet}_{0, \beta}$
is isomorphic to the space of representations of
the quiver $Q$ of dimension vector $\beta$.
From the canonical projective $\G(\beta)$-equivariant morphism
$\M_{\beta}^{\bullet} \to \M_{0, \beta}^{\bullet}$,
we form
the corresponding Steinberg type variety 
$Z^{\bullet}_{\beta} = \M_{\beta}^{\bullet} 
\times_{\M_{0, \beta}^{\bullet}} \M_{\beta}^{\bullet}$.
Thanks to the construction by Nakajima~\cite{Nakajima01}, we obtain
an algebra homomorphism from
$U_{q}(L \g)$ 
to a suitable central completion 
$\hK^{\G(\beta)}(Z^{\bullet}_{\beta})$ of
the $\G(\beta)$-equivariant $K$-group 
with the convolution product.
Then we define our \lq\lq{}completion\rq\rq{} 
$\hCc_{Q, \beta}$ of 
Hernandez-Leclerc\rq{}s category $\Cc_{Q, \beta}$
as the category of finitely generated 
$\hK^{\G(\beta)}(Z^{\bullet}_{\beta})$-modules.
We prove that this category 
$\hCc_{Q, \beta}$ actually plays a role of 
a completion of 
$\Cc_{Q, \beta}$
and that it has a structure of affine highest weight category.

\begin{theo}[{= Theorem~\ref{Thm:main}}]
\label{theo:main}
The followings hold.
\begin{enumerate}
\item \label{theo:main:affhw}
The convolution algebra $\hK^{\G(\beta)}(Z^{\bullet}_{\beta})$
is an affine quasi-hereditary algebra and hence
the category $\hCc_{Q, \beta}$ has a structure of 
affine highest weight category.
\item \label{theo:main:main}
Via Nakajima\rq{}s homomorphism, 
standard modules in $\hCc_{Q, \beta}$ are
identified with the deformed local Weyl modules 
(see Definition \ref{Def:deflocWeyl}) and
the full subcategory of all the finite-dimensional
modules in $\hCc_{Q, \beta}$ is identified 
with the category $\Cc_{Q, \beta}$.
\end{enumerate}
\end{theo}

Kang-Kashiwara-Kim\rq{}s functor $\mathcal{F}_{Q, \beta}$
is extended to an exact functor 
$\hMm_{Q, \beta} \to \hCc_{Q, \beta}$
between the two affine highest weight categories.
There is a simple sufficient condition
(= Theorem~\ref{Thm:criterion}) 
obtained in \cite{Fujita18} 
for an exact functor 
between two affine highest weight categories
to give an equivalence of categories.
We prove that the functor $\mathcal{F}_{Q, \beta}$
satisfies this sufficient condition.
Restricting the equivalence $\mathcal{F}_{Q, \beta}$ 
to the subcategory 
$\Mm_{Q, \beta}$ of finite-dimensional modules
and summing up over $\beta \in \cQ^{+}$, 
we obtain the following result. 

\begin{theo}[= Corollary \ref{Cor:equiv}]
\label{theo:equiv}
Under a certain assumption about simpleness
of poles of normalized $R$-matrices for
type $\mathrm{E}$,
Kang-Kashiwara-Kim\rq{}s
generalized quantum affine Schur-Weyl duality 
functor $\mathcal{F}_{Q} \colon \Mm_{Q} \to \Cc_{Q}$
gives an equivalence of monoidal categories.
\end{theo}

\subsection{Remark}
Theorem~\ref{theo:main} (\ref{theo:main:affhw})
can be obtained as 
a consequence of the theory of geometric extension
algebras developed by Kato~\cite{Kato17} and 
by McNamara~\cite{McNamara17},
although in this paper
we give an alternative proof 
which does not use geometric extension algebras. 
In fact, 
our convolution algebra $\hK^{\G(\beta)}(Z^{\bullet}_{\beta})$
is  isomorphic to the completion of 
the geometric extension algebra associated with 
$\M^{\bullet}_{\beta} \to \M_{0, \beta}^{\bullet}$,
which satisfies some
conditions presented in \cite{Kato17, McNamara17}.
Moreover, we can see that this geometric extension algebra is
Morita equivalent to 
the quiver Hecke algebra $H_{Q}(\beta)$
thanks to the result of Varagnolo-Vasserot~\cite{VV11}. 
From this point of view, we already know an abstract equivalence 
of categories $\hMm_{Q, \beta} \simeq \hCc_{Q, \beta}$. 
See Remark~\ref{Rem:ext} for more details.
However, one should note that
it is still not obvious that
the equivalence is realized concretely by  
Kang-Kashiwara-Kim\rq{}s functor $\mathcal{F}_{Q, \beta}$.  
In addition, to prove the other statement (\ref{theo:main:main})
in Theorem~\ref{theo:main},
it is necessary to
compare the quantum loop algebra $U_{q}(L\g)$
and the convolution algebra 
$\hK^{\G(\beta)}(Z^{\bullet}_{\beta})$
via Nakajima\rq{}s homomorphism,
which is neither injective nor surjective in general.

\subsection*{Organization}
This paper is organized as follows.
In Section~\ref{Sec:affhw}, we recall 
the definition and some properties 
of affine highest weight categories
and affine quasi-hereditary algebras.
Section~\ref{Sec:HL} is concerned with
the representation theory of quantum loop algebras
$U_{q}(L\g)$ of type $\mathrm{ADE}.$
In Section~\ref{Sec:quivvar}, we recall and prove
some geometric properties of quiver varieties, 
which are needed in the sequel.
Section~\ref{Sec:completion} is the main part of this paper.
After recalling Nakajima\rq{}s construction in Subsection~\ref{Ssec:Nakajima},
we prove Theorem~\ref{theo:main} 
in Subsection~\ref{Ssec:main}. 
We study the Kang-Kashiwara-Kim functor
in Section~\ref{Sec:KKK}.
Theorem~\ref{theo:equiv} is proved in Subsection~\ref{Ssec:final}.

\subsection*{Acknowledgments}
The author is deeply grateful to 
Syu Kato and Ryosuke Kodera 
for many fruitful discussions and encouragements.
He also thanks Masaki Kashiwara,
Hiraku Nakajima and Katsuyuki Naoi for helpful discussions.
A part of this work was done during the author's
visit at UC Riverside in March 2017.
He thanks Vyjayanthi Chari for hospitality 
and stimulating discussions during his stay.
He also thanks the anonymous referee 
for many valuable comments. 

The work of the author was supported in part by
the Kyoto Top Global University program.
It was also supported by Grant-in-Aid for JSPS 
Research Fellow (No.~18J10669) 
and by JSPS Overseas Research Fellowships during the revision.
 
									
									

\subsection*{Convention}	

For an algebra $A$, 
the category of finitely generated left $A$-modules is denoted by $A \modfg$.
If $\kk$ is a field and $A$ is a $\kk$-algebra,
the category
of finite-dimensional left $A$-modules is denoted by $A \modfd$.
For a two-sided ideal $\mathfrak{a} \subset A$
and a left $A$-module $M$,
the quotient $M/\mathfrak{a}M$ is denoted by $M/\mathfrak{a}$
for simplicity. 
Working over a field $\kk$, the symbol $\otimes$
stands for $\otimes_{\kk}$ if there is no other clarification.
For $i=1,2$, let $R_{i}$ be a complete local commutative $\kk$-algebra
with maximal ideal $\rr_{i} \subset R_{i}$ satisfying $R_{i}/\rr_{i} \cong \kk$.
For any $R_{i}$-module $M_{i} \; (i=1,2)$, 
the symbol $M_{1}\hat{\otimes}M_{2}$ denotes the completion
of the $(R_{1} \otimes R_{2})$-module $M_{1} \otimes M_{2}$ with respect to
the maximal ideal $\rr_{1} \otimes R_{2} + R_{1} \otimes \rr_{2}$.
Note that $M_{1}\hat{\otimes}M_{2}$ is a module over 
the complete local algebra $R_{1} \hat{\otimes} R_{2}$. 

			
\section{Affine highest weight categories}
\label{Sec:affhw}

In this section, we recall the definitions 
and some properties of 
(topologically complete) 
affine highest weight categories and
affine quasi-hereditary algebras following Kleshchev~\cite{Kleshchev15}. 

Let $A$ be a left Noetherian algebra
over an algebraically closed field $\kk$
and $J \subset A$ the Jacobson radical of $A$.
Throughout this section,
we assume 
that $\dim (A/J) < \infty$ and
$A$ is complete with respect to the $J$-adic topology, 
i.e.,~$\displaystyle
\varprojlim A/J^{n} \cong A$.
Let $\Cc := A \modfg$
be the $\kk$-linear abelian category 
of all finitely generated left $A$-modules.
Our assumption guarantees that
any simple module of $\Cc$ is finite-dimensional
and the number of isomorphism classes of simple modules
in $\Cc$ is finite. 
We parametrize the set $\Irr \Cc$ 
of simple isomorphism classes in $\Cc$ 
by a finite set $\Pi$ as 
$\Irr \Cc = \{L(\pi) \in \Cc \mid \pi \in \Pi \}.$
For each $\pi \in \Pi$,
we fix a projective cover $P(\pi)$ of the simple module $L(\pi)$.

\begin{Def}
\label{Def:affhered}
A two-sided ideal $I \subset A$ is said to be {\em affine heredity} if
the following three conditions are satisfied:

\begin{enumerate}

\item
\label{Def:affhered:idem}
We have $\Hom_{\Cc}(I, A/I) = 0$;

\item
\label{Def:affhered:proj}
As a left $A$-module, we have $I \cong P(\pi)^{\oplus m}$ 
for some $\pi \in \Pi$ and $m \in \Z_{>0}$;

\item
\label{Def:affhered:endo}
The endomorphism $\kk$-algebra 
$\End_{A}(P(\pi))$ is isomorphic to 
a ring of formal power series 
$\kk [\![ z_{1}, \ldots, z_{n} ]\!]$ 
for some $n \in \Z_{\ge 0}$, and
$P(\pi)$ is free of finite rank over
$\End_{A}(P(\pi))$.

\end{enumerate}
\end{Def}

\begin{Def}
We say that the algebra 
$A$ is {\em affine quasi-hereditary} if 
there is a chain of ideals:
\begin{equation}
\label{Eq:affheredchain}
0 = I_{l} \subsetneq I_{l-1} \subsetneq \cdots
\subsetneq I_{1} 
\subsetneq I_{0} = A  
\end{equation}
such that,  
for each $i \in \{1,2, \ldots, l\}$, the ideal 
$I_{i-1}/I_{i}$ is an affine heredity ideal of
the algebra $A/I_{i}$.
We refer to such a chain (\ref{Eq:affheredchain}) as
an {\em affine heredity chain}.
\end{Def}

Let $\le$ be a partial order of $\Pi$.

\begin{Def}
\label{Def:affhw}

The category $\Cc = A \modfg$ 
is called an \emph{affine highest weight category}
for the poset $(\Pi, \le)$ if, for each $\pi \in \Pi$, there
exists an indecomposable module $\std(\pi)$ which 
is a nonzero quotient of $P(\pi)$ (i.e.,~$P(\pi) \twoheadrightarrow \Delta(\pi) \twoheadrightarrow L(\pi)$)
satisfying the following three conditions:

\begin{enumerate}
\item 
The endomorphism $\kk$-algebra 
$B_{\pi} := \End_{\Cc}(\std(\pi))$ 
is isomorphic to a ring of formal power series
$\kk [\![ z_{1}, \ldots , z_{n_{\pi}} ]\!]$
for some $n_{\pi} \in \Z_{\ge 0}$, 
and $\std(\pi)$ is free of finite rank over $B_{\pi}$;
\item 
Define $\pstd(\pi) := \std(\pi)/\rad B_{\pi}$,
where $\rad B_{\pi}$ denotes the maximal ideal of $B_{\pi}$. 
Then each composition factor of 
the kernel 
of the natural quotient map 
$\pstd(\pi) \twoheadrightarrow L(\pi)$
is isomorphic to $L(\sigma)$ for some $\sigma < \pi$;
\item
The kernel 
of natural quotient map 
$P(\pi) \twoheadrightarrow \std(\pi)$
is filtered by various $\std(\sigma)$\rq{}s with $\sigma > \pi$.
\end{enumerate} 
We refer to the module $\std(\pi)$ (resp.~$\pstd(\pi)$)
as the {\em standard module} 
(resp.~{\em proper standard module})
associated with $\pi \in \Pi$.
\end{Def}

The next theorem is an analog of a famous result by Cline-Parshall-Scott~\cite{CPS88}. 

\begin{Thm}[{\cite[Theorem 6.7]{Kleshchev15}}]
\label{Thm:CPS}
The followings are mutually equivalent.
\begin{enumerate}
\item The algebra $A$ is an affine quasi-hereditary algebra;
\item There is a partial order $\le$ of $\Pi$ such that the category $\Cc = A \text{-mod}_{\text{fg}}$ is  
an affine highest weight category for $(\Pi, \le)$.
\end{enumerate}
\end{Thm}

\begin{Rem}
\label{Rem:ordering}
Let $A$ be an affine quasi-hereditary algebra 
with $\Irr \Cc = \{ L(\pi) \in \pi \in \Pi \}.$ 
Then the partial order $\le$ of $\Pi$ in Theorem~\ref{Thm:CPS} can be chosen as follows.
The standard modules of the affine highest weight 
category $\Cc$ are obtained as 
an indecomposable direct summand of
subquotients $I_{i-1}/I_{i}$ of 
an affine heredity chain (\ref{Eq:affheredchain}). 
Thus an affine heredity chain (\ref{Eq:affheredchain}) 
gives a total ordering $\{ \pi_{1}, \pi_{2}, \ldots , \pi_{l}\}$
of the parameter set $\Pi$ by 
$I_{i-1}/I_{i} \cong \std(\pi_{i})^{\oplus m_{i}}$.
Using this notation, we
define a partial order $\le$ on the set $\Pi$ so that the 
following condition is satisfied:
\begin{itemize}
\item[$(*)$] For $\sigma, \tau \in \Pi$, 
we have $\sigma < \tau$ if and only if
for any affine heredity chain we have
$\sigma = \pi_{i}, \tau = \pi_{j}$
for some $i,j$ with $1 \le i < j \le l$.
\end{itemize}
Then we can prove that the category $\Cc$ is
an affine highest weight category
for this partial order $\le$ on $\Pi$.  
\end{Rem}

The following theorem
 is the Ext-version of BGG type reciprocity. 

\begin{Thm}[{\cite[Lemmas 7.2 and 7.4]{Kleshchev15}}]
\label{Thm:affhw}
Let $\Cc$ be an affine highest weight category for a poset
$(\Pi, \le)$. Then, 
for each $\pi \in \Pi$, 
there exists an indecomposable module $\pcstd(\pi) \in \Cc$,
uniquely up to isomorphism, 
characterized by the following $\Ext$-orthogonality:
 $$
\Ext_{\Cc}^{i}(\std(\sigma), \pcstd(\pi)) =
					\begin{cases}
					\kk
					& i=0,\sigma = \pi; \\
					0 & \text{else}.
					\end{cases}
$$ 
\end{Thm}

We refer to the module $\pcstd(\pi)$ as 
the {\em proper costandard module}
associated with $\pi \in \Pi$.
The following criterion is proved
by using a theory of tilting modules in
affine highest weight categories.

\begin{Thm}[{\cite[Theorem 3.9]{Fujita18}}]
\label{Thm:criterion}
For $i=1,2$, let $\Cc_{i} = A_{i} \modfg$ 
be an affine highest weight category
for a poset $(\Pi_{i}, \le_{i})$.
Assume that we have an exact functor 
$F \colon \Cc_{1} \to \Cc_{2}$
and the following conditions are satisfied:
\begin{enumerate}
\item The algebra $A_{i}$ is a finitely generated module
over its center ($i=1,2$);
\item
There exists a bijection 
$f\colon \Pi_{1} \xrightarrow{\simeq} \Pi_{2}$ 
preserving the partial orders
and we have the following isomorphisms 
for each $\pi \in \Pi_{1}$:
$$
F(\std(\pi)) \cong \std(f(\pi)), \qquad
F(\pcstd(\pi)) \cong \pcstd(f(\pi)). 
$$
\end{enumerate}
Then the functor $F$ gives an equivalence of categories
$F \colon \Cc_{1} \simeq \Cc_{2}$.    
\end{Thm}


\section{Representations of quantum loop algebras}
\label{Sec:HL}

In this section, we recall and prove some basic facts
on the representation theory of quantum loop algebras
$U_{q}(L \g)$ of type $\mathrm{ADE}$.


\subsection{Quivers and root systems}
\label{Ssec:quivers}

From now on, we fix a Dynkin quiver $Q = (I, \Omega)$
(of type $\mathrm{ADE}$)
with its set of vertices $I=\{1,2, \ldots, n\}$ and
its set of arrows $\Omega$. 
We write $i \sim j$ if $i , j \in I$ are adjacent in $Q$.
We also fix a function  
$\xi \colon I \to \Z; i \mapsto \xi_{i}$ such that
$\xi_{j} = \xi_{i} -1$ if $i \to j \in \Omega$.
Given a Dynkin quiver $Q$,
such a function $\xi$ is determined uniquely up to adding a constant 
and called a {\em height function} on $Q$.

Let $\g$ be the complex simple Lie algebra 
whose Dynkin diagram is the underlying graph of the quiver $Q$.
The Cartan matrix $A =  (a_{ij})_{i,j \in I}$ of $\g$ is given by
$$
a_{ij} = 
\begin{cases}
2 & \text{if $i=j$;} \\
-1 &\text{if $i \sim j$;} \\
0 &\text{otherwise}.
\end{cases}
$$
Let $\cP^{\vee} = \bigoplus_{i \in I} \Z h_{i}$ 
be the coroot lattice of $\g$.
The fundamental weights $\{ \varpi_{i} \mid i \in I\}$ 
form a basis of  {the weight lattice} 
$\cP := \Hom_{\Z}(\cP^{\vee}, \Z)$
which is dual to $\{h_{i} \mid i \in I \}.$
Let $\cP^{+} := \sum_{i \in } \Z_{\ge 0}\varpi_{i}$
be the set of {dominant weights}.  
The simple roots $\{ \alpha_{i} \mid i \in I \}$ are defined by
$\alpha_{i} := \sum_{j \in I} a_{ij} \varpi_{j}$. 
We define the root lattice by 
$\cQ := \bigoplus_{i \in I} \Z \alpha_{i} \subset \cP$
and put $\cQ^{+} := \sum_{i \in I} \Z_{\ge 0} \alpha_{i}$.
We define a partial order $\le$ 
called the dominance order on
$\cP$ by the condition that for $\lambda, \mu \in \cP$,
we have $\lambda \le \mu$
if and only if $\mu - \lambda \in \cQ^{+}$. 

The Weyl group $W$ of $\g$ is a group of 
linear transformations of $\cP$ generated by 
the set $\{s_{i} \mid i \in I\}$ of simple reflections, which are
defined by
$s_{i}(\lambda) := \lambda - \lambda(h_{i}) \alpha_{i}$
for $\lambda \in \cP$.
For an element $w \in W$, its length $l(w)$ is the smallest number
$l \in \Z_{\ge 0}$ 
such that there is an expression $w = s_{i_{1}} s_{i_{2}} \cdots s_{i_{l}}$.
We say that an expression 
$w = s_{i_{1}} s_{i_{2}} \cdots s_{i_{l}}$ is reduced
if $l=l(w)$. 
Let $w_{0}$ denote the unique longest element of $W$.    
Let $\mathsf{R} := W \{ \alpha_{i} \mid i \in I\}$ 
be the set of roots, which decomposes
as $\mathsf{R} = \mathsf{R}^{+} \sqcup ( - \mathsf{R}^{+})$,
where $\mathsf{R}^{+} := \mathsf{R} \cap \cQ^{+}$ denotes
the set of positive roots. 

For each $i \in I$, 
we denote by $s_{i}Q$ the quiver obtained from $Q$ by
changing the orientations of all arrows 
incident to $i$. 
A vertex $i \in I$ is called a source (resp.~sink)
if there is no arrow
of the form $j \to i$ (resp.~$i \to j$) in $\Omega$.
A reduced expression
$w = s_{i_{1}} s_{i_{2}} \cdots s_{i_{l}}$ is said to be 
adapted to $Q$ if 
the vertex $i_{k}$ is a source of the quiver
$s_{i_{k-1}}\cdots s_{i_{1}}Q$ for every $k \in \{ 1,2, \ldots l \}$.
For any Dynkin quiver $Q$, we can choose a total ordering
$I = \{i_{1}, i_{2}, \ldots, i_{n}\}$ of the vertex set $I$ such that 
we have $a < b$ whenever $i_{a} \to i_{b} \in \Omega$.  
Then the reduced expression 
$s_{i_{1}} s_{i_{2}} \cdots s_{i_{n}}$
is adapted to $Q$.   
We define the corresponding Coxeter element $\tau \in W$
to be the product
$\tau := s_{i_{1}} s_{i_{2}} \cdots s_{i_{n}}$.
The element $\tau$ does not 
depend on the choice of such a total ordering
$I = \{i_{1}, i_{2}, \ldots, i_{n}\}$
(i.e.,~depends only on the orientation of the quiver $Q$).

Let 
$\widehat{I} := \{ (i, p) \in I \times \Z 
\mid p - \xi_{i} \in 2 \Z \}$,
where $\xi \colon I \to \Z$ is a fixed height function on $Q$. 
Following Hernandez-Leclerc~\cite[Section 2.2]{HL15},
we define the bijection\footnote{Our bijection $\phi$ is the inverse of the bijection $\varphi$ in \cite{HL15}.} 
$\phi \colon \mathsf{R}^{+} \times \Z \to \widehat{I}$ by
the following rule:

\begin{itemize}
\item[(a)]
For each $i \in I$, we put 
$\gamma_{i} := \sum_{j} \alpha_{j}$ where 
$j$ runs over all the vertices $j \in I$ such that there is an oriented path
in $Q$ from $j$ to $i$.
Then we define $\phi (\gamma_{i}, 0) := (i, \xi_{i})$;
\item[(b)]
Inductively, if $\phi(\alpha, k) = (i, p)$ 
for $(\alpha, k) \in \mathsf{R}^{+} \times \Z$, then we define
\begin{align*}
\phi(\tau^{\pm1}(\alpha), k) &:= (i, p\mp 2) 
&& \text{if $\tau^{\pm 1}(\alpha) \in \mathsf{R}^{+}$}, \\
\phi(-\tau^{\pm1}(\alpha), k \mp 1) &:= (i, p\mp 2) 
&& \text{if $\tau^{\pm 1}(\alpha) \in - \mathsf{R}^{+}$}. 
\end{align*}
\end{itemize}

Let $\widehat{Q}$ be the infinite quiver 
whose set of vertices is $\widehat{I}$ and 
whose set of arrows consists of all the arrows 
$(i, p) \to (j, p+1)$ for
$(i,p) \in \widehat{I}$ and $j \sim i$.
Note that the quiver $\widehat{Q}$ does not depends on 
the orientation of $Q$.
This quiver $\widehat{Q}$ is called the 
{\em repetition quiver}. 
Recall that by
Gabriel\rq{}s theorem,
taking dimension vector gives a bijection
between the set of isomorphism classes of  
indecomposable modules of the path algebra 
$\C Q$ and the set $\mathsf{R}^{+}$ of positive roots.
For a positive root $\alpha \in \mathsf{R}^{+}$,
let $M(\alpha)$ denote the indecomposable 
$\C Q$-module whose dimension vector is $\alpha$.
It is known that
the full subquiver $\Gamma_{Q}$ of $\widehat{Q}$
whose vertex set is 
the subset $\widehat{I}_{Q} := \phi(\mathsf{R}^{+} \times \{0\})$
is isomorphic to the {\em Auslander-Reiten quiver} 
of $\C Q $. 
The indecomposable module $M(\alpha)$ corresponds
to the vertex $\phi(\alpha) := \phi(\alpha, 0)$.
Moreover,
the action of the Coxeter element $\tau$ on $\mathsf{R}$ corresponds 
to the Auslander-Reiten translation.


\subsection{Quantum loop algebras}
\label{Ssec:qloop}

Let $q$ be an indeterminate.
From now on, let
$\kk$ denote
the algebraic closure of $\Q(q)$ in the ambient field 
$\bigcup_{m \in \Z_{\ge 1}} \overline{\Q}(\!(q^{1/m})\!)$.

\begin{Def}
\label{Def:qloop}
The quantum loop algebra 
$U_{q} \equiv U_{q}(L\g)$ 
associated with $\g$
is a $\kk$-algebra given by the generators:
$$
\{ e_{i,r} , f_{i,r} \mid i \in I, r \in \Z \}\cup 
\{ q^{h} \mid h \in \cP^{\vee}\} \cup
\{ h_{i, m} \mid i \in I, m \in \Z\setminus \{0\}\}$$
and the following relations:
$$
q^{0} = 1, \quad
q^{h} q^{h^{\prime}} = q^{h+h^{\prime}}, \quad
[q^{h}, h_{i,m}] = [h_{i,m}, h_{j, l}] = 0, 
$$
$$
q^{h} e_{i,r} q^{-h} = q^{\alpha_{i}(h)} e_{i,r}, 
\quad
q^{h} f_{i,r} q^{-h} = q^{-\alpha_{i}(h)} f_{i,r},
$$
$$
(z-q^{\pm a_{ij}}w)\psi_{i}^{\varepsilon}(z) x_{j}^{\pm}(w)
= (q^{\pm a_{ij}}z - w)x_{j}^{\pm}(w) \psi_{i}^{\varepsilon}(z), 
$$
$$
[x_{i}^{+}(z), x_{j}^{-}(w)] = 
\frac{\delta_{ij}}{q - q^{-1}} 
\left( \delta \left( \frac{w}{z} \right) \psi_{i}^{+}(w)
- \delta \left( \frac{z}{w} \right) \psi_{i}^{-}(z) \right), 
$$
$$
(z - q^{\pm a_{ij}}w)x_{i}^{\pm}(z) x_{j}^{\pm}(w) 
=
(q^{\pm a_{ij}}z-w)
x_{j}^{\pm}(w) x_{i}^{\pm}(z),   
$$
\vspace{-6mm}
\begin{multline}
\{ x_{i}^{\pm}(z_{1}) x_{j}^{\pm}(z_{2}) x_{j}^{\pm}(w)
-(q + q^{-1})x_{i}^{\pm}(z_{1}) x_{j}^{\pm}(w)x_{i}^{\pm}(z_{2})
\\
+ x_{j}^{\pm}(w) x_{i}^{\pm}(z_{1}) x_{i}^{\pm}(z_{2}) 
\} 
+
\{ z_{1} \leftrightarrow z_{2}\}
= 0
\qquad \text{if $i \sim j$},
\nonumber
\end{multline}
where $\varepsilon \in \{+, - \}$
and $\delta(z),  \psi_{i}^{\pm}(z), x_{i}^{\pm}(z)$ 
are the formal series defined as follows:
$$
\delta(z) := \sum_{r=-\infty}^{\infty}z^{r},
\quad
\psi_{i}^{\pm}(z) := q^{\pm h_{i}}
\exp \left( \pm (q-q^{-1}) 
\sum_{m=1}^{\infty}h_{i, \pm m} z^{\mp m} \right),
$$
$$
x_{i}^{+}(z) := \sum_{r=-\infty}^{\infty}e_{i,r}z^{-r},
\quad x_{i}^{-}(z) := \sum_{r=-\infty}^{\infty}f_{i,r}z^{-r}.
$$
In the last relation, the second term $\{ z_{1} \leftrightarrow z_{2}\}$
means the exchange of $z_{1}$ with $z_{2}$ in the first term.
\end{Def}

\begin{Rem}
\label{Rem:Beck}
By \cite{Beck94},
we have a $\kk$-algebra isomorphism
$U_{q}(L\g) \cong 
U_{q}^{\prime}(\widehat{\g}) / \langle q^{c} -1 \rangle$, 
where the RHS is 
a quotient of the quantum affine algebra
$U_{q}^{\prime}(\widehat{\g})$ 
(without the degree operator),
which is presented by the Chevalley type generators
$\{ e_{i}, f_{i} \mid i \in I \cup \{ 0 \} \} \cup 
\{ q^{h} \mid h \in 
\cP^{\vee} \oplus \Z c \}$,
by the ideal generated by a central element $q^{c}-1$.
Although this isomorphism depends on 
a function $o \colon I \to \{\pm 1\}$ such that
$o(i) = - o(j)$ if $i \sim j$,
the choice does not affect the results of this paper.
Via this isomorphism,
the quantum loop algebra $U_{q}(L\g)$  
inherits a structure of Hopf algebra from 
the quantum affine algebra $U_{q}^{\prime}(\widehat{\g})$.
In terms of the Chevalley type generators,
the coproduct
$\Delta$ is given 
by:
$$
\Delta(e_{i}) = e_{i} \otimes q^{-h_{i}} + 1 \otimes e_{i}, \quad
\Delta(f_{i}) = f_{i} \otimes 1 + q^{h_{i}} \otimes f_{i}, \quad
\Delta(q^{h}) = q^{h} \otimes q^{h}  
$$
for $i \in I \cup \{0\}, h \in \cP^{\vee}$.
By \cite[Proposition 7.1]{Damiani98} (see also \cite[Remark 3.4]{Hernandez10}), for each $i \in I$ and $r \in \Z_{>0}$,
we have
\begin{equation}
\label{Eq:delta_looph}
\Delta(h_{i, \pm r}) - (h_{i, \pm r}\otimes 1
+ 1 \otimes h_{i, \pm r}) \in
 \bigoplus_{\gamma \in \cQ^{+} \setminus \{ 0 \}}
(U_{q})_{- \gamma} \otimes
(U_{q})_{ \gamma}, 
\end{equation} 
where
$(U_{q})_{\gamma}
:=\{x \in U_{q} \mid q^{h} x q^{-h} = q^{\gamma(h)}x\;
(\forall h \in \cP^{\vee}) \}
$.
The antipode $S$ is given by 
$$
S(e_{i}) = - e_{i}q^{h_{i}}, \quad
S(f_{i}) = - q^{-h_{i}}f_{i}, \quad
S(q^{h}) = q^{-h}, 
$$
which we use to define the dual modules.
\end{Rem}  

A $U_{q}$-module $M$ is said to be 
of type $\mathbf{1}$ if it has 
a decomposition
$
M = \bigoplus_{\lambda \in \cP}M_{\lambda},
$
where
$
M_{\lambda} := \{v \in M \mid 
q^{h}v = q^{\lambda(h)}v \ (\forall h \in \cP^{\vee})\}.
$
A non-zero subspace $M_{\lambda}$ is called a weight spaces of $M$.
Let $\Cc_{\g}$ denote the category of 
finite-dimensional $U_{q}$-modules of type $\mathbf{1}$. 
This is an abelian $\kk$-linear monoidal category.

We also use the modified quantum loop algebra
$\tilU(L\g)$ defined by
$$
\tilU \equiv
\tilU(L\g)
:= \bigoplus_{\lambda \in \cP} U_{q}a_{\lambda},
\quad
U_{q}a_{\lambda} 
:= \left.
U_{q} \middle/
\sum_{h \in \cP^{\vee}} U_{q}(q^{h} - q^{\lambda(h)}),\right. 
$$
where $a_{\lambda}$ stands for 
the image of $1$ in the quotient.
This is a non-unital $\kk$-algebra, whose multiplication is given by
$$
a_{\lambda} a_{\mu} = \delta_{\lambda \mu} a_{\lambda},
\quad
a_{\lambda} x = x a_{\lambda - \gamma},
$$
where $x \in (U_{q})_{\gamma}, \gamma \in \cQ$. 
By definition, considering a 
$\tilU$-module is 
the same as considering a $U_{q}$-module 
of type $\mathbf{1}$.  

We use the following notation.
Let $\lP := \bigoplus_{(i , a) \in I \times \kk^{\times}}
\Z \varpi_{i, a}  
$  
be the set of {\em $\ell$-weights}, which is 
a free abelian group with a basis 
$ \{\varpi_{i,a} \mid i \in I, a \in \kk^{\times} \}$.
We call a basis element $\varpi_{i,a}$ 
a {\em fundamental $\ell$-weight}.
An element in the submonoid  
$\lP^{+} 
:= \sum
\Z_{\ge 0} \varpi_{i,a} 
$
is said to be {\em $\ell$-dominant}.
We define a $\Z$-linear map
$\mathsf{cl}\colon \lP \to \cP$
by $\varpi_{i,a} \mapsto \varpi_{i}$.  
Following~\cite{FR99},
we define the {\em $\ell$-root}
$\alpha_{i,a}$ for each $(i, a) \in I \times \kk^{\times}$ by 
$$
\alpha_{i,a} := \varpi_{i,aq} + \varpi_{i, aq^{-1}} 
-\sum_{j \sim i} \varpi_{j,a}.
$$
We define the {\em $\ell$-root lattice} by
$\lQ := \bigoplus_{(i,a) \in I \times \kk^{\times}} \Z \alpha_{i,a}
\subset \lP$ and set
$\lQ^{+} := \sum \Z_{\ge 0}\alpha_{i,a}.$
Note that 
$\mathsf{cl} \colon \lP \to \cP$ induces a map
$\mathsf{cl} \colon \lQ \to \cQ$ since 
$\mathsf{cl} (\alpha_{i,a}) = \alpha_{i}.$
We define a partial order $\le$ on
$\lP$ called the {\em Nakajima partial ordering}
by the condition
that for $\hlam, \hmu \in \lP$, we have $\hlam \le \hat{\mu}$
if and only if $\hat{\mu} - \hlam \in \lQ^{+}$. 

Let $U_{q}(L\h)$ denote the commutative $\kk$-subalgebra
of $U_{q}(L\g)$ generated by 
$\{ q^{h} \mid h \in \cP^{\vee}\} \cup
\{ h_{i, r} \mid i \in I, r \in \Z \setminus \{ 0 \} \}$.
A module $M \in \Cc_{\g}$ decomposes into a direct sum
of generalized eigenspaces for $U_{q}(L\h)$ as
$M = \bigoplus M_{\Psi},$
where $\Psi = (\Psi^+_{i}(z), \Psi_i^-(z))_{i \in I} 
\in \kk[\![ z^{-1}]\!]^{I} \times \kk[\![z]\!]^I$ and
$M_{\Psi}$ is the subspace 
on which all the coefficients of the series 
$\psi^{\pm}_{i}(z) - \Psi_{i}^{\pm}(z)\, \mathrm{id}_M$
act nilpotently.
By \cite{FR99}, it is known that 
if $M_{\Psi} \neq 0$,
there is a unique $\ell$-weight
$\hlam = \sum l_{i,a}
\varpi_{i,a} \in \lP$ such that we have
\begin{equation}
\label{Eq:Psi}
\Psi_{i}^{\pm}(z) = q^{\mathsf{cl}(\hlam)(h_{i})}\left(
\prod_{a \in \kk^{\times}} 
\left(
\frac{1-aq^{-2}z^{-1}}{1-az^{-1}}\right)^{l_{i,a}}
\right)^{\pm},
\end{equation}
where $(-)^{\pm}$ denotes 
the formal expansion at $z=\infty$ and $0$ respectively.
In this case, we write $M_{\hlam} = M_{\Psi}$
and call it the {\em $\ell$-weight space} 
of $\ell$-weight $\hlam$.

We say a module $M \in \Cc_{\g}$ is an 
{\em $\ell$-highest weight module} 
of $\ell$-highest weight $\hlam \in \lP$
if there exists a generating vector $v_{0} \in M$
satisfying
$$
x_{i}^{+}(z) \cdot v_{0} = 0, \quad
\psi_{i}^{\pm}(z) \cdot v_{0} =
q^{\mathsf{cl}(\hlam)(h_{i})}\left(
\prod_{a \in \kk^{\times}} 
\left(
\frac{1-aq^{-2}z^{-1}}{1-az^{-1}}\right)^{l_{i,a}}
\right)^{\pm} v_{0}
$$
in $M[\![z, z^{-1}]\!]$ for any $i \in I$.
Compare the latter equation with (\ref{Eq:Psi}).
In this case, it is known by \cite{CP, CP95}, that the $\ell$-highest weight $\hlam$
automatically becomes $\ell$-dominant, i.e.,~$\hlam \in \lP^{+}$
and we have $M_{\hlam} = \kk \cdot v_{0}.$
Moreover, for any $\hlam \in \lP^+$,
there is a simple $\ell$-highest weight module $L(\hlam) \in \Cc_\g$ uniquely up to isomorphism
and any simple module in $\Cc_{\g}$ is of this form.
By \cite{Nakajima01, FM01}, it is known that for an element $\hmu \in \lP$,  
we have $L(\hlam)_{\hmu} \neq 0$ only if
$\hmu \le \hlam$.
When $\hlam = \varpi_{i,a}$ for some $(i,a) \in I \times \kk^\times$, the simple module $L(\varpi_{i,a})$ 
is called an \emph{$\ell$-fundamental module}. 

Recall that
for two simple modules $M_{1}, M_{2} \in \Cc_{\g}$,
we say that $M_{1}$ and  $M_{2}$ are \emph{linked}
if there is no splitting $\Cc_{\g} \cong
\Cc_{1} \oplus \Cc_{2}$  such that
$M_{1} \in \Cc_{1}$ and $M_{2} \in \Cc_{2}$.

\begin{Thm}[Chari-Moura~\cite{CM05}]
\label{Thm:block}
For any $\ell$-dominant $\ell$-weight
$\hlam , \hmu \in \lP^{+}$,
the simple modules $L(\hlam)$
and $L(\hmu)$ are linked if and only if
$\hlam - \hmu \in \lQ$.
\end{Thm} 

For each $M \in \Cc_{\g}$, we define its left dual module
$M^{*}$ (resp.~right dual module ${}^{*}M$)
as the dual space $\Hom_{\kk}(M, \kk)$ equipped
with the left $U_{q}$-action
obtained by twisting the natural right action with the antipode $S$ (resp.~$S^{-1}$).
For any $M_{1}, M_{2} \in \Cc_{\g}$, we have
$$
(M_{1} \otimes M_{2})^{*} \cong M_{2}^{*} \otimes M_{1}^{*},
\quad
{}^{*}(M_{1} \otimes M_{2})
\cong {}^{*}M_{2} \otimes {}^{*}M_{1}.
$$
We define the following $\Z$-linear maps on $\lP$:
\begin{align*}
(-)^{*} \colon & \lP \to \lP, \quad \varpi_{i, a} \mapsto 
\varpi_{i,a}^{*}
:=\varpi_{i^{*}, aq^{-h}}, \\
{}^{*}(-) \colon & \lP \to \lP, \quad \varpi_{i, a} \mapsto 
{}^{*}\varpi_{i,a} := \varpi_{i^{*}, aq^{h}},
\end{align*}    
where $i \mapsto i^{*}$ denotes the involution on $I$ defined by
$\alpha_{i^{*}} := - w_{0} \alpha_{i}$ and $h$ is 
the Coxeter number (the order of a Coxeter element of $W$).
Then we have 
\begin{equation}
\nonumber
L(\varpi_{i,a})^{*} \cong L(\varpi_{i,a}^{*}), \quad
{}^{*}L(\varpi_{i,a}) \cong L({}^{*}\varpi_{i,a}).
\end{equation}
See~\cite[Corollary 6.10]{FM01}.

\subsection{Weyl modules}
\label{Ssec:Weyl} 

In this subsection, we recall the global and
local Weyl modules of $U_{q}$ introduced by Chari-Pressley~\cite{CP01}. 
Also we define the 
deformed local Weyl modules,
which will play a role of 
standard modules of an affine highest weight category later.

\begin{Def}
A $U_{q}$-module $M$ of type $\mathbf{1}$
is said to be
{\em $\ell$-integrable}
if the following property is satisfied:
For each $v \in M$, there exists an integer $n_{0} \ge 1$
such that we have
$
e_{i, r_{1}} e_{i, r_{2}}\cdots e_{i, r_{N}} v
=
f_{i, r_{1}} f_{i, r_{2}}\cdots f_{i,r_{N}} v
= 0
$
for any $N \ge n_{0}$ and
any $i \in I$, $r_{1} , \ldots, r_{N} \in \Z$.
\end{Def}

\begin{Rem}
In this paper,
we do not impose 
that $\dim M_{\lambda} < \infty$
for $\ell$-integrability.
Note that
any finite-dimensional modules of type $\mathbf{1}$,
i.e.,~any objects of the category $\mathcal{C}_{\g}$
are automatically $\ell$-integrable. 
\end{Rem}

First we define the global Weyl modules. 

\begin{Def}
Let $\lambda \in \cP^{+}$ be a dominant weight.
We define the
{\em global Weyl module} 
associated with $\lambda$
to be the left $U_{q}$-module $\bW(\lambda)$  generated by a 
cyclic vector $w_{\lambda}$ 
satisfying the following defining relations:
$$
e_{i,r} w_{\lambda} = 0, \quad
q^{h} w_{\lambda} = q^{\lambda(h)} w_{\lambda},\quad 
(f_{i,r})^{\lambda(h_{i})+1} w_{\lambda} = 0,
$$ 
where $i \in I, r \in \Z$ and $h \in \cP^{\vee}$.
\end{Def}

In what follows, $\SG_d$ denotes the symmetric group of degree $d \in \Z_{\ge 1}$.
For a dominant weight 
$\lambda = \sum_{i \in I} l_{i}\varpi_{i} \in \cP^{+}$,
we define the following $\kk$-algebra of
partially symmetric Laurent polynomials:
\begin{equation}
\label{Eq:Rlambda}
\R(\lambda) :=
\bigotimes_{i \in I} 
(\kk[z_{i}^{\pm 1}]^{\otimes l_{i}})^{\SG_{l_{i}}} 
=
\bigotimes_{i \in I}
\kk[z_{i, 1}^{\pm 1}, \ldots , 
z_{i, l_{i}}^{\pm 1}]^{\SG_{l_{i}}}.
\end{equation}

\begin{Thm}[Chari-Pressley~\cite{CP01}, Nakajima]
\label{Thm:globalWeyl}
Write $\lambda = \sum_{i \in I} l_{i} \varpi_{i} \in \cP^{+}.$
\begin{enumerate}
\item
\label{Thm:globalWeyl:univ}
The global Weyl module $\bW(\lambda)$ is 
$\ell$-integrable and has the following universal property: 
If  $M$ is an $\ell$-integrable $U_{q}$-module
with a vector $v \in M_{\lambda}$
of weight $\lambda$ satisfying
$x_{i}^{+}(z) \cdot v=0$ for any $i \in I$,
then there is a unique $U_{q}$-homomorphism
$\bW(\lambda) \to M$ such that
$w_{\lambda} \mapsto v$; 
\item
\label{Thm:globalWeyl:End}
There is a unique isomorphism  
$\End_{U_{q}} (\bW(\lambda))
\cong 
\R(\lambda)
$ such that we have
$$\psi_{i}^{\pm}(z) w_{\lambda}
= q^{l_{i}}
\prod_{k=1}^{l_{i}}
\left( 
\frac{1 - q^{-2}z_{i, k}z^{-1}}{1-z_{i,k}z^{-1}}
\right)^{\pm} w_{\lambda}
$$
for each $i \in I$;
\item
\label{Thm:globalWeyl:free}
$\bW(\lambda)$ 
is free over $\R(\lambda)$ 
of finite rank.
\end{enumerate} 
\end{Thm}

\begin{proof}
See \cite[Section 4]{CP01}.
The assertion 
(\ref{Thm:globalWeyl:free}) is proved by 
the geometric realization due to Nakajima.
For details, see Theorem~\ref{Thm:Kcentral} and Theorem~\ref{Thm:Kfibers}
(\ref{Thm:Kfibers:K0}) below.
\end{proof}

\begin{Rem}
The global Weyl module
$\bW(\lambda)$
is known to be isomorphic to 
the level $0$ extremal weight module
$V^{\mathrm{max}}(\lambda)$
of extremal weight $\lambda$,
in the sense of Kashiwara~\cite{Kashiwara94}.
See \cite[Proposition 4.5]{CP01}
and \cite[Remark 2.15]{Nakajima04}.
\end{Rem}

Next we discuss the local Weyl modules.
We identify a point 
of the quotient space 
$(\kk^{\times})^{N}/ \SG_{N}$
with a $(\Z_{\ge 0})$-linear combination
of the formal symbols $\{[a] \mid a \in \kk^{\times} \}$
whose coefficients sum up to $N$.  
Note that we have
$\Specm \R(\lambda)
\cong \prod_{i \in I}\left( (\kk^{\times} )^{l_{i}} 
/ \SG_{l_{i}} \right).
$
Let $\hlam = 
\sum_{(i, a) \in I \times \kk^{\times}}
l_{i,a} \varpi_{i,a}
\in \lP^{+}$
be an $\ell$-dominant $\ell$-weight
and put
$\lambda := \mathsf{cl}(\hlam) 
\in \cP^{+}.$
Let
$\rr_{\lambda, \hlam}$ denote
the maximal ideal of $\R(\lambda)$
corresponding to the
point
$$\left(\sum_{a \in \kk^{\times}}
l_{i, a} [a] \right)_{i \in I} \in 
\prod_{i \in I} \left( ( \kk^{\times} )^{l_{i}} 
/ \SG_{l_{i}} \right).$$
\begin{Def}
We define the 
{\em local Weyl module} $W(\hlam)$
associated with $\hlam \in \lP^{+}$
by $W(\hlam) := \bW(\lambda)/\rr_{\lambda, \hlam}$.
The image of the cyclic vector 
$w_{\lambda} \in \bW(\lambda)$ is denoted by $w_{\hlam} \in W(\hlam)$.
\end{Def}

\begin{Thm}[Chari-Pressley~\cite{CP01}]
Let $\hlam \in \lP^{+}$ be an $\ell$-dominant $\ell$-weight. 
\begin{enumerate}
\item 
The local Weyl module $W(\hlam)$ is 
a finite-dimensional 
$\ell$-highest weight module of $\ell$-highest weight
$\hlam$ with $W(\hlam)_{\hlam} = \kk \cdot w_{\hlam}.$
Moreover it has the following universal property:
If $M \in \Cc_{\g}$ is an $\ell$-highest weight module
of $\ell$-highest weight $\hlam$ with 
$M_{\hlam} = \kk \cdot v_{0}$,
then there is a unique surjective $U_{q}$-homomorphism
$W(\hlam) \to M$ such that $w_{\hlam} \mapsto v_{0}$;
\item
$W(\hlam)$ has a simple head isomorphic to
$L(\hlam).$
\end{enumerate}
\end{Thm}
\begin{proof}
The assertions follow immediately from Theorem~\ref{Thm:globalWeyl}.
\end{proof}

\begin{Rem} \label{Rem:locWeyl}
The local Weyl module $W(\hlam)$ is known to be isomorphic to
the standard module in the sense of Nakajima~\cite[Section 13]{Nakajima01}
and Varagnolo-Vasserot~\cite{VV02}.
See Theorem~\ref{Thm:AK} below and \cite[Corollary 7.16]{VV02}.  
\end{Rem}

Finally, we introduce
the deformed local Weyl modules.
Let $\hlam = \sum l_{i,a} \varpi_{i,a} \in \lP^{+}$ be 
an $\ell$-dominant $\ell$-weight and 
set $\lambda := \mathsf{cl}(\hlam)$. 
We define  
$$
\hR(\lambda, \hlam) := 
\varprojlim_{N}
\R(\lambda)
/ \rr_{\lambda, \hlam}^{N}.
$$
\begin{Def}
\label{Def:deflocWeyl}
We define the {\em deformed local Weyl module}  $\hW(\hlam)$
associated with $\hlam \in \lP^{+}$ by
$$
\hW(\hlam) := \bW(\lambda) \otimes_{\R(\lambda)}
\hR(\lambda, \hlam)
\cong \varprojlim_{N} \bW(\lambda) / \rr_{\lambda, \hlam}^{N}.
$$
We set $\widehat{w}_{\hlam} := w_{\lambda} \otimes 1 
\in \hW(\hlam).$
\end{Def}
We also use the following algebra:
\begin{equation}
\label{Eq:Rhlam}
\R(\hlam) := \bigotimes_{i \in I} \bigotimes_{a \in \kk^{\times}}
\left( \kk[z_{i}^{\pm1}]^{\otimes l_{i,a}} \right)
^{\SG_{l_{i,a}}}.
\end{equation}
Note that 
the algebra $\R(\lambda)$ is
a subalgebra of $\R(\hlam)$.
Let $\rr_{\hlam}$ denote a maximal ideal of
$\R(\hlam)$ corresponding to the point
$$\left( l_{i,a} [a] \right)_{(i,a) \in I \times \kk^{\times}}
\in \prod_{(i, a) \in I \times \kk^{\times}} 
\left( (\kk^{\times})^{l_{i,a}} / \SG_{l_{i,a}} \right) = \Specm \R(\hlam).$$
Then we have $\rr_{\lambda, \hlam} = \R(\lambda) \cap 
\rr_{\hlam}$ and there is a natural isomorphism 
\begin{equation}
\label{Eq:etale}
\hR(\hlam) := 
\varprojlim_{N}
\R(\hlam)
/ \rr_{\hlam}^{N}
\cong \hR(\lambda, \hlam).
\end{equation}  
Hereafter we identify $\hR(\lambda, \hlam)$
with $\hR(\hlam)$ via the isomorphism \eqref{Eq:etale}.
For simplicity, we will use the same symbol $\rr_{\hlam}$ to denote the maximal ideal of the local ring $\hR(\hlam)$.  


\begin{Prop}
\label{Prop:defWeyl}
The deformed local Weyl module $\hW(\hlam)$ 
satisfies the following properties:
\begin{enumerate}
\item  
\label{Prop:defWeyl:univ}
For each $M \in \Cc_{\g}$, 
taking the image of $\widehat{w}_{\hlam}$ gives a
natural isomorphism:
$$
\Hom_{U_{q}}(\hW(\hlam) , M) \cong \{ v \in M_{\hlam}
\mid e_{i,r} v= 0 \; \text{for any $i \in I, r \in \Z$} \};
$$
\item 
\label{Prop:defWeyl:End}
$\End_{U_{q}}(\hW(\hlam)) = \hR(\hlam)$ 
and
$\hW(\hlam)$ is free over $\hR(\hlam)$ of finite rank;
\item 
$\hW(\hlam) / \rr_{\hlam} \cong W(\hlam)$.
\end{enumerate}
\end{Prop}
\begin{proof}
The assertions follow immediately from Theorem~\ref{Thm:globalWeyl}.
\end{proof}

\begin{Rem} \label{Rem:aff}
Let $z$ be an indeterminate.
For any $U_q$-module $M$, we can equip the $\kk[z^{\pm 1}]$-module
$M[z^{\pm 1}] := M \otimes \kk[z^{\pm 1}]$ with a structure of $U_q$-module by setting
\begin{equation} \label{Eq:aff}
\begin{aligned}
e_{i,r} \cdot (v\otimes \varphi) &= e_{i,r}v \otimes z^r \varphi, 
& f_{i,r} \cdot (v\otimes \varphi) &= f_{i,r}v \otimes z^r \varphi, \\
q^h \cdot (v\otimes \varphi) &= q^h v \otimes \varphi, 
& h_{i,l} \cdot (v\otimes \varphi) &= h_{i,l}v \otimes z^l \varphi,
\end{aligned}
\end{equation}
for any $v \in M, \varphi \in \kk[z^{\pm 1}]$ and the defining generators $e_{i,r}, f_{i,r}, q^h, h_{i,l}$ of $U_q$. 
The resulting $U_q$-module $M[z^{\pm 1}]$ is called the \emph{affinization of $M$} 
(cf.~\cite[Section 4.2]{Kashiwara02}).

For each $i \in I$, we have an isomorphism $\bW( \varpi_i ) \cong L(\varpi_{i,1})[z^{\pm 1}]$ of 
$U_q$-modules, under which the generating vector $w_{\varpi_i}$ corresponds to the vector 
$v_0 \otimes 1$ with $v_0 \in L(\varpi_{i,1})$ being a fixed $\ell$-highest weight vector.
This isomorphism yields the isomorphism $\End_{U_q}(\bW(\varpi_i)) \cong \kk[z^{\pm 1}]$
in Theorem~\ref{Thm:globalWeyl}~\eqref{Thm:globalWeyl:End}.
In particular, the deformed local Weyl module $\hW(\varpi_{i,a})$ is isomorphic to 
the $\kk[\![z-a]\!]$-module
$L(\varpi_{i,1}) \otimes \kk[\![ z-a ]\!]$ equipped with the action of $U_q$
by the same formulas as \eqref{Eq:aff}.
\end{Rem}


\subsection{$R$-matrices and factorization of deformed local Weyl modules}
\label{Ssec:Rmatrix}

In this subsection, we recall some facts about 
$R$-matrices between $\ell$-fundamental modules 
following \cite{AK97, KKK18, Kashiwara02}
and
describe a factorization 
of deformed local Weyl modules.  

For any pair $(i_1, i_2) \in I^{2}$,
let us write
$\End_{U_{q}}(\bW(\varpi_{i_j})) = \kk[ z_{j}^{\pm 1}]$
for $j=1, 2$
as in Theorem~\ref{Thm:globalWeyl} 
(\ref{Thm:globalWeyl:End}).
Then there is a unique homomorphism 
of $(U_{q}, \kk[z_{1}^{\pm1}, z_{2}^{\pm1}])$-bimodules,
called the {\em normalized $R$-matrix} 
$$
R^{\mathrm{norm}}_{i_1, i_2}
\colon \bW(\varpi_{i_1}) \otimes \bW(\varpi_{i_2})
\to 
\kk(z_{2}/z_{1}) \otimes_{\kk[(z_{2}/z_{1})^{\pm 1}]}
\left( \bW(\varpi_{i_2}) \otimes \bW(\varpi_{i_1}) \right),
$$
such that $R^{\mathrm{norm}}_{i_1, i_2} 
(w_{\varpi_{i_1}}\otimes w_{\varpi_{i_2}}) = w_{\varpi_{i_2}} \otimes
w_{\varpi_{i_1}}$.
The {\em denominator} of 
$R^{\mathrm{norm}}_{i_1, i_2}$ is defined to be the monic polynomial 
$d_{i_1, i_2}(u) \in \kk[u]$ with the smallest degree 
among polynomials satisfying
$$
\Ima R^{\mathrm{norm}}_{i_1, i_2} 
\subset 
d_{i_1, i_2}(z_{2}/z_{1})^{-1} \otimes  
\left( \bW(\varpi_{i_2}) \otimes \bW(\varpi_{i_1}) \right).
$$
By \cite[Proposition 9.3]{Kashiwara02}, we have
\begin{equation} \label{Eq:denom}
d_{i_1, i_2}(a) = 0 \quad \Longrightarrow 
\quad a \in \textstyle \bigcup_{m \in \Z_{>0}} q^{1/m}\overline{\Q}[\![q^{1/m}]\!].
\end{equation}

\begin{Thm}[\cite{CM05, Kashiwara02, VV02}]  
\label{Thm:AK}
Let $\hlam = \sum_{j=i}^{l} \varpi_{i_{j}, a_{j}} \in \lP^{+}$
be an $\ell$-dominant $\ell$-weight.
Then the following three conditions are mutually equivalent: 
\begin{enumerate}
\item 
\label{Thm:AK:cyclic}
The tensor product module
$L(\varpi_{i_{1}, a_{1}}) \otimes
L(\varpi_{i_{2}, a_{2}}) \otimes \cdots
\otimes L(\varpi_{i_{l}, a_{l}})$
is generated by the tensor product 
of $\ell$-highest weight vectors;
\item 
\label{Thm:AK:localWeyl}
$W(\hlam) \cong 
L(\varpi_{i_{1}, a_{1}}) \otimes
L(\varpi_{i_{2}, a_{2}}) \otimes \cdots
\otimes L(\varpi_{i_{l}, a_{l}});$
\item 
\label{Thm:AK:denom}
$d_{i_{j}, i_{k}}(a_{k}/a_{j}) \neq 0$ for any $1 \le j < k \le l$.
\end{enumerate}
We can always make these conditions satisfied by reordering $\{(i_j, a_j)\}_j$ suitably. 
\end{Thm}
\begin{proof}
The equivalence of (\ref{Thm:AK:cyclic}) and
(\ref{Thm:AK:localWeyl}) was proved by
Chari-Moura~\cite[Theorem 6.4]{CM05}
using the results from geometry 
due to Nakajima~\cite{Nakajima01}.
The equivalence of (\ref{Thm:AK:cyclic})
and (\ref{Thm:AK:denom}) was proved 
by Kashiwara~\cite[Proposition 9.4]{Kashiwara02}. 
The last assertion follows from \eqref{Eq:denom}.
See also \cite[Corollary 7.15]{VV02}.
\end{proof}
  
\begin{Def}
Let $\hlam \in \lP^{+}$ be an $\ell$-dominant $\ell$-weight.
We define the {\em dual local Weyl module} 
associated with $\hlam$ by 
$W^{\vee}(\hlam) := W({}^{*}\hlam)^{*}$.
\end{Def}

\begin{Prop}
\label{Prop:dualWeyl}
Let $\hlam = \sum_{j=i}^{l} \varpi_{i_{j}, a_{j}} \in \lP^{+}$
be an $\ell$-dominant $\ell$-weight.
Assume 
$W(\hlam) \cong 
L(\varpi_{i_{1}, a_{1}}) \otimes
L(\varpi_{i_{2}, a_{2}}) \otimes \cdots
\otimes L(\varpi_{i_{l}, a_{l}}).$
Then we have
$$
W^{\vee}(\hlam) \cong 
L(\varpi_{i_{l}, a_{l}}) \otimes
L(\varpi_{i_{l-1}, a_{l-1}}) \otimes \cdots
\otimes L(\varpi_{i_{1}, a_{1}}).
$$
\end{Prop}
\begin{proof}
Use the equivalence of (\ref{Thm:AK:localWeyl})
and
(\ref{Thm:AK:denom})
in Theorem~\ref{Thm:AK}
and the fact $d_{i_1, i_2}(u) = d_{i_{1}^{*}, i_{2}^{*}}(u)$ 
(see \cite[Appendix A]{AK97}).
\end{proof}

For any two $\ell$-dominant $\ell$-weights 
$\hlam = \sum l_{i, a}\varpi_{i, a}, 
\hlam^{\prime} = \sum l^{\prime}_{i, a} \varpi_{i,a} \in \lP^{+}$,
we have the following injective homomorphism:
\begin{multline*}
\R(\hlam + \hlam^{\prime})
= \bigotimes_{i \in I} \bigotimes_{a \in \kk^{\times}}
\left( \kk[z_{i}^{\pm1}]^{\otimes l_{i,a} + l^{\prime}_{i,a}} \right)
^{\SG_{(l_{i,a} + l^{\prime}_{i,a})}}\\
\hookrightarrow
\bigotimes_{i \in I} \bigotimes_{a \in \kk^{\times}}
\left( \kk[z_{i}^{\pm1}]^{\otimes l_{i,a}} 
\otimes \kk[z_{i}^{\pm 1}]^{\otimes l^{\prime}_{i,a}}\right)
^{\SG_{l_{i,a}} \times \SG_{l^{\prime}_{i ,a}}}
= \R(\hlam) \otimes \R(\hlam^{\prime}).
\end{multline*}
We have similar injective homomorphisms
$$
\R(\hlam_{1} + \cdots + \hlam_{l})
\hookrightarrow
\R(\hlam_{1}) \otimes \cdots \otimes \R(\hlam_{l})
$$
for any $\hlam_{1}, \ldots, \hlam_{l} \in \lP^{+}.$
They induce the following injective homomorphisms
for the completions:
$$
\hR(\hlam_{1} + \cdots + \hlam_{l})
\hookrightarrow
\hR(\hlam_{1}) \hat{\otimes} \cdots \hat{\otimes} \R(\hlam_{l}).
$$
We refer to these kinds of injective homomorphisms as the \emph{standard inclusions}. 

Let $\hlam := \sum_{j=1}^{l} m_{j} \varpi_{i_j, a_{j}}
\in \lP^{+}$
be an $\ell$-dominant $\ell$-weight with
$(i_{j}, a_{j}) \neq (i_{k}, a_{k})$ for $j \neq k$.
We see that the standard inclusion is an isomorphism 
$\hR(\hlam) \cong
\hR(m_{1} \varpi_{i_1, a_1}) \hat{\otimes}
\cdots \hat{\otimes} \hR(m_{l} \varpi_{i_l, a_l})$ in this case.
Thus the completed tensor product
$
\hW(m_{1} \varpi_{i_1, a_1}) \hat{\otimes}
\cdots \hat{\otimes} \hW(m_{l} \varpi_{i_l, a_l})
$
is regarded as a $(U_{q}, \hR(\hlam))$-bimodule.

\begin{Prop}
\label{Prop:hW_tensor}
With the above notation, we further assume 
that $d_{i_j, i_k}(a_k / a_j) \neq 0$
for any $1 \le j <  k \le l$. 
Then we have an isomorphism of 
$(U_{q}, \hR(\hlam))$-bimodules:
$$
\hW(\hlam) \cong \hW(m_{1} \varpi_{i_1, a_1}) \hat{\otimes}
\cdots \hat{\otimes} \hW(m_{l} \varpi_{i_l, a_l}).
$$
\end{Prop}

\begin{proof}
By the universal property of the 
global Weyl module $\bW(\lambda)$ 
(Theorem~\ref{Thm:globalWeyl} (\ref{Thm:globalWeyl:univ})),
there is the $U_{q}$-homomorphism
\begin{equation}
\label{Eq:bW_hom}
\bW(\lambda) \to \bW(m_{1} \varpi_{i_{1}})
\otimes \cdots
\otimes \bW(m_{l} \varpi_{i_{l}}); 
\, w_{\lambda} \mapsto w_{m_{1} \varpi_{i_{1}}}
\otimes \cdots \otimes w_{m_{l} \varpi_{i_{l}}}.
\end{equation}
By (\ref{Eq:delta_looph}), we have
$$\psi_{i}^{\pm}(z)  
(w_{m_{1} \varpi_{i_{1}}}
\otimes \cdots \otimes w_{m_{l} \varpi_{i_{l}}})
= 
(\psi_{i}^{\pm}(z)  w_{m_{1} \varpi_{i_{1}}})
\otimes \cdots \otimes 
(\psi_{i}^{\pm}(z)  w_{m_{l} \varpi_{i_{l}}})
$$
for any $i \in I$.
Therefore,
by Theorem~\ref{Thm:globalWeyl} 
(\ref{Thm:globalWeyl:End}), 
the homomorphism (\ref{Eq:bW_hom}) is actually 
a homomorphism of $(U_{q}, \R(\hlam))$-bimodules,
where the RHS of (\ref{Eq:bW_hom}) is regarded
as an $\R(\hlam)$-module via the standard inclusion.  
Completing with respect to the maximal ideal
$\rr_{\lambda, \hlam} \subset \R(\lambda)$,
we get
\begin{equation}
\label{Eq:hW_hom}
\hW(\hlam) \to 
\hW(m_{1} \varpi_{i_1, a_1}) \hat{\otimes}
\cdots \hat{\otimes} \hW(m_{l} \varpi_{i_l, a_l}).
\end{equation}
Note that both sides of (\ref{Eq:hW_hom})
are free over $\hR(\hlam)$.
By Theorem~\ref{Thm:AK},
taking the quotients by $\rr_{\hlam}$ 
induces an isomorphism 
$$
W(\hlam) \xrightarrow{\cong} 
W(m_{1}\varpi_{i_1, a_1}) \otimes 
\cdots \otimes W(m_{l}\varpi_{i_l, a_l})
\cong 
L(\varpi_{i_1, a_1})^{\otimes m_{1}} \otimes
\cdots \otimes L(\varpi_{i_l, a_l})^{\otimes m_l}.
$$
Therefore we see that (\ref{Eq:hW_hom})
is an isomorphism by Nakayama\rq{}s lemma.
\end{proof}

Note that 
if $d_{i_1, i_2}(a_2/a_1) \neq 0$,
the $R$-matrix 
$R^{\mathrm{norm}}_{i_1, i_2}$ 
induces a homomorphism 
of $(U_{q}, \kk[\![ z_{1} -a_{1}, z_{2}-a_{2}]\!])$-bimodules
$$R^{\mathrm{norm}}_{i_1, i_2}
\colon \hW(\varpi_{i_1, a_1}) \hat{\otimes}
\hW(\varpi_{i_2, a_2})
\to 
\hW(\varpi_{i_2, a_2}) \hat{\otimes}
\hW(\varpi_{i_1, a_1}).
$$

Let $NH_{m}$ denote
the {\em nil-Hecke algebra} of degree $m \in \Z_{>0}$. 
For its basic properties, see \cite[Examples 2.2.3]{KL09} and references therein.
It is a $\kk$-algebra presented by the generators
$\{x_{1}, \ldots, x_{m} \} \cup \{ \tau_{1}, \ldots, \tau_{m-1} \}$
satisfying the following relations:
$$
x_{k} x_{l} = x_{l} x_{k}, \quad \tau_{k}^{2} = 0,
\quad 
\tau_{k} \tau_{k+1} \tau_{k} = \tau_{k+1} \tau_{k} \tau_{k+1},
\quad \tau_{k} \tau_{l} = \tau_{l} \tau_{k} \quad \text{if $|k-l|>1$}, 
$$
$$
\tau_{k} x_{k+1} - x_{k} \tau_{k}
= x_{k+1} \tau_{k} - \tau_{k} x_{k} = 1,
\quad
\tau_{k} x_{l} = x_{l} \tau_{k} \quad \text{if $l \neq k, k+1$}.
$$
The third and forth relations 
above are called the braid relations.
Let $\sigma_{i} \in \SG_{m}$ 
denote the transposition of $i$ and $i+1$ $(1 \le i < m)$.
For each $\sigma \in \SG_{m}$, we define
$\tau_{\sigma} := \tau_{i_1} \cdots \tau_{i_l} \in NH_{m}$,
where $\sigma = \sigma_{i_1} \cdots \sigma_{i_l}$ 
is a reduced expression of $\sigma$.
Thanks to the braid relations, 
the element $\tau_{\sigma}$ does not depend on
the choice of a reduced expression.
It is known that the algebra
$NH_{m}$ is free as a left (or a right) 
$\kk[x_{1}, \ldots, x_{m}]$-module with
a free basis $\{ \tau_{\sigma} \mid \sigma \in \SG_{m}\}$
and the center of $NH_{m}$ coincides
with the subalgebra $\mathbb{S}_{m}
:= \kk[x_{1}, \ldots, x_{m}]^{\SG_{m}}$
of symmetric polynomials.
Moreover, it is known that $NH_{m}$ is isomorphic 
to the matrix algebra of 
rank $m!$ over its center $\mathbb{S}_{m}$.
A primitive idempotent $e_{m}$ is given by
$e_{m} := \tau_{\sigma_{0}} x_2 x_{3}^{2} \cdots x_{m}^{m-1}$,
where $\sigma_{0} \in \SG_{m}$ is the longest element.
The nil-Hecke algebra $NH_{m}$ is a graded algebra 
by setting $\mathrm{deg} \, x_{k} = 2,  \mathrm{deg} \, \tau_{k} = -2.$
Since the grading is bounded from below, 
the completion $\widehat{NH}_{m}$ with respect to 
the grading naturally becomes an algebra. 
The completed nil-Hecke algebra 
$\widehat{NH}_{m}$ is isomorphic to the matrix algebra
over its center $\hS_{m} := \kk [\![ x_{1}, \ldots, x_{m}]\!]^{\SG_{m}}$
of rank $m!$.

Let $\hlam := m \varpi_{i,a} \in \lP^{+}$
for a pair $(i,a) \in I \times \kk^{\times}$ and $m \in \Z_{>0}$.
Consider the completed tensor product
$\hW(\varpi_{i, a})^{\hat{\otimes} m}$,
which is regarded as a $(U_{q}, \hR(\hlam))$-bimodule,
via the standard inclusion
$\hR(\hlam) \hookrightarrow \hR(\varpi_{i,a})^{\hat{\otimes} m} = \kk[\![z_1 -a, \ldots, z_m -a ]\!]$.
Since we know that $d_{i, i}(1) \neq 0$ (recall \eqref{Eq:denom}),
the 
$(U_{q}, \hR(\hlam))$-bimodule automorphism
$$r_{k} := 1^{\hat{\otimes} k-1} 
\hat{\otimes} \sigma_k R^{\mathrm{norm}}_{i,i}
\hat{\otimes} 1^{\hat{\otimes} m-k-1} \in
\End (\hW(\varpi_{i, a})^{\hat{\otimes} m})
$$
is well-defined for $1 \le k < m$, where $\sigma_k$ 
denotes the permutation of $z_k$ and $z_{k+1}$.
Then the formulas 
$$
x_{k} \mapsto z_{k} -a, \quad
\tau_{k} \mapsto (z_{k} - z_{k+1})^{-1}(r_{k} -1)
$$
define a right action of $\widehat{NH}_{m}$,
which makes $\hW(\varpi_{i, a})^{\hat{\otimes} m}$ into
a $(U_{q}, \widehat{NH}_{m})$-bimodule.

\begin{Rem}
The nil-Hecke algebra $HN_{m}$ 
coincides with the quiver Hecke algebra of type $\mathrm{A}_1$. 
The above construction of the $\widehat{NH}_{m}$-action is a special case of the construction 
of Kang-Kashiwara-Kim~\cite{KKK18}, which we will review later in Section~\ref{Ssec:KKK}.
\end{Rem}

\begin{Prop}
\label{Prop:nilHecke}
Let $\hlam = m \varpi_{i,a} \in \lP^{+}$. 
As $(U_{q}, \hR(\hlam))$-bimodules, we have
$$
\hW(\varpi_{i,a})^{\hat{\otimes} m} \cong 
\hW(\hlam)^{\oplus m!}.
$$
\end{Prop}
\begin{proof}
Since $\widehat{NH}_{m} \cong
\left(\widehat{NH}_{m} e_{m} \right)^{\oplus m!}$
as a left $\widehat{NH}_{m}$-module,
we have
$$
\hW(\varpi_{i,a})^{\hat{\otimes} m}
=
\hW(\varpi_{i,a})^{\hat{\otimes} m} \otimes_{\widehat{NH}_{m}}
\widehat{NH}_{m}
\cong 
\left( \hW(\varpi_{i,a})^{\hat{\otimes} m} e_{m} \right)^{\oplus m!}.
$$
A summand
$ \hW(\varpi_{i,a})^{\hat{\otimes} m} e_{m}$
is a $(U_{q}, \hS_{m})$-bimodule.
Note that we have $\hS_{m} = \hR(\hlam)$ inside 
$\End(\hW(\varpi_{i,a})^{\hat{\otimes} m} e_{m})$.
Since $\hW(\varpi_{i,a})^{\hat{\otimes} m}$ is free over
$\hR(\varpi_{i,a})^{\hat{\otimes} m}$ of rank $(\dim L(\varpi_{i,a}))^{m}$,
the summand $ \hW(\varpi_{i,a})^{\hat{\otimes} m} e_{m}$ 
is free over $\hR(\hlam)$ of rank
$(\dim L(\varpi_{i,a}))^{m}$.
Using the universal property, we have the $(U_{q}, \hR(\hlam))$-homomorphism
$$\hW(\hlam) \to  \hW(\varpi_{i,a})^{\hat{\otimes} m} e_{m}
; \quad \widehat{w}_{\hlam} \mapsto 
(\widehat{w}_{\varpi_{i,a}})^{\hat{\otimes} m} e_{m}.
$$
Taking the quotients by $\rr_{\hlam}$, we get a non-zero
$U_{q}$-homomorphism
$W(\hlam) \to \hW(\varpi_{i,a})^{\hat{\otimes} m} e_{m} / \rr_{\hlam}$. 
This should be an isomorphism because 
$W(\hlam)\cong L(\varpi_{i,a})^{\otimes m} \cong W^\vee(\hlam)$
is simple (recall Theorem~\ref{Thm:AK} and Proposition~\ref{Prop:dualWeyl}) and $\dim W(\hlam) 
= \dim (\hW(\varpi_{i,a})^{\hat{\otimes} m} e_{m} / \rr_{\hlam})$.
Thus we complete the proof by Nakayama\rq{}s lemma.
\end{proof}

\begin{Cor}
\label{Cor:hW_tensor}
Let $\hlam := \sum_{j=1}^{l} m_{j} \varpi_{i_j, a_{j}}
\in \lP^{+}$
be an $\ell$-dominant $\ell$-weight with
$(i_{j}, a_{j}) \neq (i_{k}, a_{k})$ for $j \neq k$.
Assume that $d_{i_j, i_k}(a_k / a_j) \neq 0$
for any $1 \le j <  k \le l$. 
Then there is an isomorphism of $(U_{q}, \hR(\hlam))$-bimodules
$$
\hW(\varpi_{i_1, a_1})^{\hat{\otimes} m_{1}}
\hat{\otimes} \cdots \hat{\otimes}
\hW(\varpi_{i_l, a_l})^{\hat{\otimes}m_l}
\cong
\hW(\hlam)^{\oplus m_{1}! \cdots m_{l}!},
$$
where $\hR(\hlam)$ acts on the LHS via the
standard inclusion. 
\end{Cor}
\begin{proof}
It follows from Propositions
\ref{Prop:hW_tensor} and \ref{Prop:nilHecke} above.
\end{proof}

\subsection{Affine cellular structure}
\label{Ssec:affcell}

In this subsection,
we recall
the affine cellular algebra structure 
(in the sense of Koenig-Xi~\cite{KX12}) of
the modified quantum loop algebra
$\tilU$,
following 
\cite{BN04, Nakajima15}.
Note that the notion of affine cellular algebra 
is closely related to the notion of affine quasi-hereditary algebra 
as explained in \cite[Section 9]{Kleshchev15}.
What we discuss in this subsection will be a key ingredient 
in the proof of our main theorem in Section~\ref{Ssec:main}. 

Let 
$\lambda = \sum_{i \in I} l_{i} \varpi_{i} \in \cP^{+}$.
By Theorem~\ref{Thm:globalWeyl} 
(\ref{Thm:globalWeyl:End}),
the global Weyl module $\bW(\lambda)$
is regarded as a ($\tilU, \R(\lambda)$)-bimodule.
We obtain a ($\R(\lambda), \tilU$)-bimodule
$\bW(\lambda)^{\sharp}$ 
from $\bW(\lambda)$
by twisting the actions
of $\tilU$ and $\R(\lambda)$
with the anti-involution 
$\sharp$ on $\tilU \otimes \R(\lambda)$ given by
$$
\sharp (e_{i}) = f_{i}, \quad
\sharp (f_{i}) = e_{i}, \quad
\sharp (q^{h}) = q^{h}, \quad
\sharp (a_{\lambda}) = a_{\lambda}, \quad
\sharp (z_{j, k}) = z^{-1}_{j,k},
$$
where $e_{i}, f_{i} \, (i \in I \cup\{0\})$
are the Chevalley generators (see Remark~\ref{Rem:Beck}),
$h \in \cP^{\vee}$, $\lambda \in \cP$ and
$z_{j,k} \, (j \in I, 1 \le k \le l_{j})$
are as in (\ref{Eq:Rlambda}).
 
Fix a dominant weight $\lambda \in \cP^{+}$.
Let $U_{\le \lambda}$ be the following quotient of 
the modified quantum loop algebra $\tilU$ :
\begin{equation}
\label{Eq:Ule}
U_{\le \lambda} := \left. \tilU \middle/ \bigcap_{\mu \le \lambda}
\Ann_{\tilU} \bW(\mu), \right.
\end{equation}
where $\Ann_{\tilU} M$
denotes the annihilator of a $\tilU$-module $M$.  
We fix a total ordering
$\{ \lambda_{1} , \lambda_{2}, \ldots, \lambda_{l} \}$ 
of the finite set $\cP^{+}_{\le \lambda} := 
\{\mu \in \cP^{+} \mid \mu \le \lambda \}$
such that we have $\lambda_{l} = \lambda$
and $i<j$ whenever $\lambda_{i} < \lambda_{j}$. 
For each $i \in \{1, 2, \ldots , l-1\}$,
we define a two-sided ideal $I_{i}$ 
of $U_{\le \lambda}$ by
\begin{equation}
\label{Eq:cellularideals}
I_{i} := \bigcap_{j \le i}
\Ann_{U_{\le \lambda}} \bW(\lambda_{j}).
\end{equation}
We also define $I_{0} := U_{\le \lambda}$ and $I_{l} := 0$.
By definition, we have $I_{i} \subset I_{i-1}$ for 
each $i \in \{1, \ldots, l\}$. 
         
\begin{Thm}
[Beck-Nakajima~\cite{BN04, Nakajima15}]
\label{Thm:cellular}
For each $i \in \{1, \ldots, l\}$,
there is an isomorphism of 
$(U_{\le \lambda}, U_{\le \lambda})$-bimodules
$$
I_{i-1} / I_{i} \cong \bW(\lambda_{i}) \otimes_{\R(\lambda_{i})}
\bW(\lambda_{i})^{\sharp}.
$$
Under this isomorphism,
the image of the element $a_{\lambda_{i}} \in I_{i-1}$
corresponds to the generating vector
$w_{\lambda_{i}} \otimes w_{\lambda_{i}} \in 
\bW(\lambda_{i}) \otimes_{\R(\lambda_{i})}
\bW(\lambda_{i})^{\sharp}.$
\end{Thm}      
\begin{proof}
See \cite[Section A(ii), A(iii)]{Nakajima15}.
\end{proof}


\subsection{Hernandez-Leclerc category $\Cc_{Q}$}
\label{Ssec:catCQ}

In this subsection, we define the Hernandez-Leclerc category 
$\Cc_{Q}$ following \cite{HL15}.

Henceforth, we only consider $\ell$-weights $\varpi_{i,a}$
with $a = q^{p}$ for some $p \in \Z$.
To simplify the notation, 
we write 
$\varpi_{i,p}$ and
$\alpha_{i,p}$ for $(i,p) \in I \times \Z$
instead of $\varpi_{i, q^{p}}$ and $\alpha_{i,q^{p}}$ respectively.
Recall that we defined the subsets 
$\widehat{I}_{Q} \subset \widehat{I}
\subset I \times \Z$ in Subsection \ref{Ssec:quivers}.
Now, we consider the following sublattices:
$$\lP \quad \supset \quad \lP_{\Z} := 
\bigoplus_{(i,p) \in \widehat{I}} \Z \varpi_{i,p}
\quad \supset \quad 
\lP_{Q} := \bigoplus_{(i,p) \in \widehat{I}_{Q}}
\Z \varpi_{i,p},$$
$$
\lQ
\quad \supset \quad
\lQ_{\Z} := \bigoplus_{(i,p) \in \widehat{J}} \Z \alpha_{i,p}
\quad \supset \quad
\lQ_{Q} 
:= \bigoplus_{(i,p) \in \widehat{J}_{Q}} \Z \alpha_{i,p},
$$
where $\widehat{J} := (I \times \Z) \setminus \widehat{I}$
and 
$\widehat{J}_{Q}:= \{ (i,p) \in I \times \Z \mid
(i, p-1), (i , p+1) \in \widehat{I}_{Q} \}$.
By definition, we have 
$\lQ_{\Z} =\lP_{\Z} \cap \lQ$.
We set
$\lP_{\Z}^{+}
:= \lP_{\Z} \cap \lP^{+},
\lP_{Q}^{+}
:= \lP_{Q} \cap \lP^{+},
\lQ_{\Z}^{+}
:= \lQ_{\Z} \cap \lQ^{+},
\lQ_{Q}^{+}
:= \lQ_{Q} \cap \lQ^{+}.
$

\begin{Lem}
\label{Lem:lrootQ}
The followings hold.
\begin{enumerate}
\item \label{Lem:lrootQ1}
$\lQ_{Q} = \lP_{Q} \cap \lQ.$
\item \label{Lem:lrootQ2}
Assume that $\hlam - \hnu \in \lP_{\Z}^{+}$
for $\hlam \in \lP^{+}_{Q}$, $\hnu \in \lQ^{+}_{\Z}$.
Then we have $\hnu \in \lQ^{+}_{Q}$.
\end{enumerate}
\end{Lem}

\begin{proof}
As mentioned in the last paragraph
of Subsection \ref{Ssec:quivers},  
the full subquiver $\Gamma_{Q}$ 
with its vertex set $\widehat{I}_{Q}$ 
inside $\widehat{Q}$ is isomorphic to the 
Auslander-Reiten quiver of the path algebra $\C Q$.
In particular, the following properties are satisfied:
\begin{itemize}
\item[(i)] 
If both $(i, p_{1})$ and $(i, p_{2})$ belong to $\widehat{I}_{Q}$
with $p_{1} < p_{2}$,
then $(i, p)$ also belongs to $\widehat{I}_{Q}$ for any $p$ with 
$p_{1} < p < p_{2}$
and $p - \xi_{i} \in 2\Z$.
\item[(ii)]
If both $(i, p-1)$ and $(i, p+1)$ belong to $\widehat{I}_{Q}$,
then $(j, p)$ also 
belongs to $\widehat{I}_{Q}$ for any $j$ with $i \sim j$. 
\end{itemize}
From these properties, we obtain the assertions. 
\end{proof}

The following combinatorial lemma will be used in Subsection \ref{Ssec:final}.  

\begin{Lem}
\label{Lem:path}
Suppose that $Q$ is not of type $\mathrm{A}_{1}$.\footnote{The type $\mathrm{A}_1$ is exceptional here
because its repetition quiver is disconnected.}
\begin{enumerate}
\item \label{Lem:path:repet}
Let $(i,p), (j,r) \in \widehat{I}$
be two vertices of the repetition quiver $\widehat{Q}$.
If $d_{i,j}(q^{r-p})=0$,
there is an oriented path in $\widehat{Q}$
from $(i,p)$ to $(j,r)$. 
\item \label{Lem:path:AR}
Let $(i,p), (j,r) \in \widehat{I}_{Q}$
be two vertices of the Auslander-Reiten quiver $\Gamma_{Q}$.
If $d_{i,j}(q^{r-p})=0$,
there is an oriented path in $\Gamma_{Q}$
from $(i,p)$ to $(j,r)$. 
\end{enumerate}
\end{Lem}

\begin{proof}
We prove the assertion (\ref{Lem:path:repet}) first.
By \eqref{Eq:denom},
we have $r > p$.
By Theorem~\ref{Thm:AK},
we see that the module $L(\varpi_{i,p}) \otimes L(\varpi_{j,r})$
is not simple.
Therefore there is a non-zero element 
$\hnu \in \lQ^{+}_{\Z}$ such that
\begin{equation}
\label{Eq:cond}
\varpi_{i, p} + \varpi_{j,r} - \hnu \in \lP^{+}_{\Z},
\end{equation}
which imposes $r - p \ge 2$.
We write
$\hnu = \sum_{(k,s) \in X} n_{k,s} \alpha_{k,s}$
with $X := \{ (k, s) \in \widehat{J} \mid n_{k,s} >0 \}.$
Then from (\ref{Eq:cond}), we can easily see the followings:
\begin{itemize}
\item[(a)] $p< s < r$ holds whenever $(k, s) \in X$;
\item[(b)] $(i,p+1) \in X$;
\item[(c)] $k=j$ holds when $(k, r-1) \in X$.
\end{itemize}
Set $\hnu_{0} := \hnu - \alpha_{i,p+1}$. 
Then we have $\hnu_{0} \in \lQ^{+}_{\Z}$ by the 
property (b). Rewrite (\ref{Eq:cond}) as
\begin{equation}
\label{Eq:cond0}
\sum_{i_{0} \sim i} 
\varpi_{i_{0}, p+1}
- \varpi_{i, p+2} 
+ \varpi_{j, r}
- \hnu_{0} 
\in \lP^{+}_{\Z}.
\end{equation} 
If $p+2 = r$, 
we have $(i, p+2) = (j,r)$ by (c) and
find a path
$(i, p) \to (i_{0}, p+1) \to (j,r)$ in $\widehat{Q}$
for an $i_{0}$ with $i_{0} \sim i$.
If $p+2 < r$, (\ref{Eq:cond0}) implies that 
there is some
$(k_{1}, p+2) \in X$ with $k_{1} \sim i$.
We set $\hnu_{1} := \hnu_{0} - \alpha_{k_{1}, p+2} \in \lQ^{+}_{\Z}$
and rewrite (\ref{Eq:cond0}) as
$$
\sum_{i_{0} \sim i, i_{0} \neq k_{1}} 
\varpi_{i_{0}, p+1}
+
\sum_{i_{1} \sim k_{1}, i_1 \neq i} 
\varpi_{i_{1}, p+2}
- \varpi_{k_{1}, p+3}
+ \varpi_{j, r}
- \hnu_{1} 
\in \lP^{+}_{\Z}.
$$
If $p+3 = r$, we have $(k_1, p+3) = (j,r)$
and find a path 
$(i, p) \to (k_{1}, p+1) \to (i_{1}, p+2) \to (j,r)$
for an $i_{1} \in I$ with $i_{1} \sim k_{1}$.
If $p+3 < r$, we find another $(k_2, p+3) \in X$ with $k_2 \sim k_1$ and $k_2 \neq i$, and repeat a similar argument.
After repeating a similar argument finitely many times,
we get the assertion (\ref{Lem:path:repet}).
A proof of the assertion (\ref{Lem:path:AR}) 
can be completely similar, thanks to 
Lemma~\ref{Lem:lrootQ} (\ref{Lem:lrootQ2}).   
\end{proof}

\begin{Def}[Hernandez-Leclerc~\cite{HL10, HL15}]
We define the category $\Cc_{Q}$
(resp.~$\Cc_{\Z}$)
to be the full subcategory 
of the category $\Cc_{\g}$ consisting of modules
whose composition factors are isomorphic to 
$L(\hlam)$ for some $\hlam \in \lP^{+}_{Q}$
(resp.~$\hlam \in \lP^{+}_{\Z}$).
\end{Def} 

The categories $\Cc_{Q}$ and $\Cc_{\Z}$ are 
proved to be monoidal subcategories
of $\Cc_{\g}$ by a similar reason as in
the proof of Lemma~\ref{Lem:lrootQ}.
See \cite[Lemma 5.8]{HL15} 
and \cite[Proposition 5.8]{HL10} respectively. 

Using the bijection
$\phi \colon \mathsf{R}^{+} \times \{ 0 \} \to \widehat{I}_{Q}$,
we can write as 
$\lP_{Q} = \bigoplus_{\alpha \in \mathsf{R}^{+}} 
\Z \varpi_{\phi(\alpha)}$, where
$\phi(\alpha) := \phi(\alpha, 0).$
We define a $\Z$-linear map 
$\mathsf{deg} \colon \lP_{Q} \to \cQ$ by
$\mathsf{deg}\, \varpi_{\phi(\alpha)} := \alpha$ for 
$\alpha \in \mathsf{R}^{+}$.
For each $\beta \in \cQ^{+}$, we define the
finite subset 
$$\lP^{+}_{Q, \beta} :=
\{ \hlam \in \lP^{+}_{Q} \mid
\mathsf{deg} (\hlam) = \beta
\}
$$
of $\ell$-dominant $\ell$-weights of degree $\beta$. 
Let $\Cc_{Q, \beta}$ be the full subcategory of $\Cc_{Q}$
consisting of modules whose composition
factors are isomorphic to
$L(\hlam)$ for some $\hlam \in \lP^{+}_{Q,\beta}$.

\begin{Lem}
\label{Lem:block}
The block decomposition of the category $\Cc_\g$ 
in Theorem~\ref{Thm:block} induces a direct decomposition
$$
\Cc_{Q} \cong \bigoplus_{\beta \in \cQ^{+}}
\Cc_{Q, \beta}.
$$
Moreover, we have $\Cc_{Q, \beta_{1}} \otimes 
\Cc_{Q, \beta_{2}} \subset 
\Cc_{Q, \beta_{1}+\beta_{2}}$ 
for $\beta_{1}, \beta_{2} \in \cQ^{+}$.
\end{Lem}    

\begin{proof}
Let $(i,p) \in \widehat{J}_{Q}$.
Then the indecomposable module 
$M(\phi^{-1}(i,p+1))$ is non-projective and 
its Auslander-Reiten translation is $M(\phi^{-1}(i, p-1))$,
where we regard $\phi^{-1} \colon \widehat{I}_{Q} \to \mathsf{R}^{+}.$
By the Auslander-Reiten theory, there is an almost split sequence
$$
0 \to M(\phi^{-1}(i,p-1)) \to 
\bigoplus_{j \sim i} M(\phi^{-1}(j,p))
\to M(\phi^{-1}(i , p+1)) \to 0. 
$$ 
Because the dimension vector function
$\underline{\dim} (-)$ is additive,
we have
\[
\mathsf{deg}\, \alpha_{i,p}
=
\phi^{-1}(i,p-1)
+ \phi^{-1}(i,p+1)
- \sum_{j \sim i} \phi^{-1}(j,p) = 0.
\]
Therefore we have $\mathsf{deg} \, \hnu = 0$ for any
$\hnu \in \lQ_{Q}$.
Combining this observation with
Theorem~\ref{Thm:block} and Lemma~\ref{Lem:lrootQ}
(\ref{Lem:lrootQ1}),
we obtain the former assertion.

The latter assertion follows from the fact
$$\dim(M_{1} \otimes M_{2})_{\hlam} = 
\sum_{\hlam_{1} + \hlam_{2} = \hlam} \dim (M_{1})_{\hlam_{1}}
\cdot  \dim(M_{2})_{\hlam_{2}},$$
which holds for any $M_1, M_2 \in \Cc_\g$. 
This is due to \cite[Theorem 3]{FR99}.
\end{proof}

\begin{Rem}
The decomposition $\Cc_{Q} = \bigoplus_{\beta \in \cQ^{+}} \Cc_{Q, \beta}$ in Lemma~\ref{Lem:block}
turns out to be a block decomposition, i.e.,~$L(\hlam_{1})$ and $L(\hlam_{2})$
are linked in $\Cc_{Q, \beta}$ for any $\hlam_{1}, \hlam_{2} \in \lP^+_{Q, \beta}$.
Indeed, the composition multiplicity of the simple module
$L(\hlam)$ in the local Weyl module $W(\hlam_{\beta})$ is non-zero 
for each $\hlam \in \lP^+_{Q, \beta}$ (see Section~\ref{Ssec:identify} for the notation).
This follows from the geometric fact $\M^\bullet_{0} (\hlam_{\beta} - \hlam, \hlam_{\beta}) \neq \varnothing$
(see Lemma~\ref{Lem:strata} below).
\end{Rem}


\section{Quiver varieties}
\label{Sec:quivvar}

In this section, we review the definitions and some properties
of the (graded) quiver varieties associated with a Dynkin 
quiver $Q$. 
Basic references are \cite{Nakajima94, Nakajima98, Nakajima01}.
We keep the notation in Section \ref{Sec:HL}.


\subsection{Quiver varieties of Dynkin types}
\label{Ssec:usual}

Fix 
an element 
$\nu = \sum_{i \in I} n_{i} \alpha_{i} \in \cQ^{+}$
and a dominant weight
$\lambda = \sum_{i \in I} l_{i} \varpi_{i}
\in \cP^{+}$.
Consider $I$-graded $\C$-vector spaces
$
V^{\nu} = \bigoplus_{i \in I} V^{\nu}_{i},
W^{\lambda} = \bigoplus_{i \in I} W^{\lambda}_{i}
$
such that $\dim V^{\nu}_{i} = n_{i}, 
\dim W^{\lambda}_{i} = l_{i}$ for each $i \in I$.
We form the following space of linear maps:
$$
\mathbf{N}(V^{\nu}, W^{\lambda})
:= 
\left(
\bigoplus_{i \to j \in \Omega} \Hom (V^{\nu}_{i}, V^{\nu}_{j})
\right)
\oplus
\left(
\bigoplus_{i \in I} \Hom (W^{\lambda}_{i}, V^{\nu}_{i})
\right),
$$
which is considered as the space of
framed representations of the quiver $Q$
of dimension vector $(\nu, \lambda)$.
On the space $\mathbf{N}(V^{\nu}, W^{\lambda})$, 
the group
$G(\nu) := 
\prod_{i \in I} GL(V^{\nu}_{i})$ 
acts by conjugation.
Let 
$$\mathbf{M}(V^{\nu}, W^{\lambda}) 
:= T^{*}\mathbf{N}(V^{\nu}, W^{\lambda}) 
= \mathbf{N}(V^{\nu}, W^{\lambda})\oplus 
\mathbf{N}(V^{\nu}, W^{\lambda})^{*}$$ 
be
the cotangent space of
$\mathbf{N}(V^{\nu}, W^{\lambda})$, 
which is naturally a 
symplectic vector space. 
We identify the space  
$\mathbf{M}(V^{\nu}, W^{\lambda})$ 
with the direct sum
$$
\left(
\bigoplus_{(i,j); i \sim j} \Hom (V^{\nu}_{j}, V^{\nu}_{i}) \right)
\oplus
\left( \bigoplus_{i \in I} \Hom (W^{\lambda}_{i}, V^{\nu}_{i}) \right)
\oplus
\left(
\bigoplus_{i \in I} \Hom (V^{\nu}_{i}, W^{\lambda}_{i}) \right).
$$
According to this direct sum expression,
we write an element of $\mathbf{M}(V^{\nu}, W^{\lambda})$ 
as a triple
$(B,a,b)$ of linear maps $B = \bigoplus B_{ij}$,
$a = \bigoplus a_{i}$ and $b = \bigoplus b_{i}$.
Let $\mu = \bigoplus_{i \in I} \mu_{i} 
\colon \mathbf{M}(V^{\nu},W^{\lambda}) 
\to \bigoplus_{i \in I} \mathfrak{gl}(V^{\nu}_{i})$
be the moment map with respect to the $G(\nu)$-action.
Explicitly, it is given by the formula
$$
\mu_{i}(B, a, b)
= a_{i} b_{i}
+\sum_{j \sim i} \varepsilon(i,j) B_{ij}B_{ji},
$$
where $\varepsilon(i,j) := 1$ (resp.~$-1$)
if we have $j \to i$ (resp.~$i \to j$) in $\Omega$.
A point $(B, a, b) \in \mu^{-1}(0)$ is said to be stable
if there exists no non-zero $I$-graded subspace 
$V^{\prime} \subset V^{\nu}$ such that 
$B(V^{\prime}) \subset V^{\prime}$ and 
$V^{\prime} \subset \Ker b$.
Let $\mu^{-1}(0)^{\mathrm{st}}$ be 
the set of stable points, on which $G(\nu)$ acts freely. 
Then we consider a set-theoretic quotient
$\M(\nu, \lambda) := \mu^{-1}(0)^{\mathrm{st}} / G(\nu)$.
It is known that this quotient has a structure of
a non-singular quasi-projective variety which 
is isomorphic to a quotient in the geometric invariant theory.
On the other hand, we also consider the affine algebro-geometric
quotient
$\M_{0}(\nu, \lambda) := \mu^{-1}(0)/\!/G(\nu) 
= \mathrm{Spec}\, \C[\mu^{-1}(0)]^{G(\nu)}$,
together with the canonical projective morphism
$\pi \colon \M(\nu, \lambda) 
\to \M_{0}(\nu, \lambda).$
We refer to these varieties $\M(\nu, \lambda), 
\M_{0}(\nu, \lambda)$ as {\em quiver varieties}.

On the linear space $\mathbf{M}(V^{\nu}, W^{\lambda})$, 
the group
$
G(\lambda) := 
\prod_{i \in I} GL(W^{\lambda}_{i})
$ 
acts by conjugation
and $\C^{\times}$ acts
as the scalar multiplication.
The combined action of the group
$\G(\lambda) := G(\lambda) \times \C^{\times}$
on $\mathbf{M}(V^{\nu},W^{\lambda})$ commutes with 
the action of the group $G(\nu)$. Thus we
have the induced $\G(\lambda)$-actions on the quotients
$\M(\nu, \lambda), \M_{0}(\nu, \lambda)$,
which make the canonical morphism $\pi$ into
a $\G(\lambda)$-equivariant morphism.  

For $\nu, \nu^{\prime} \in \cQ^{+}$ with
$\nu \le \nu^{\prime}$, 
we fix a direct sum decomposition
$V^{\nu^{\prime}} = V^{\nu} \oplus V^{\nu^{\prime} - \nu}$.
Extending by $0$ on $V^{\nu^{\prime}-\nu}$,
we have an injective linear map
$\mathbf{M}(V^{\nu}, W^{\lambda}) \hookrightarrow
\mathbf{M}(V^{\nu^{\prime}}, W^{\lambda})$. 
This induces a natural closed embedding
$\M_{0} (\nu , \lambda)
\hookrightarrow 
\M_{0}(\nu^{\prime},\lambda),
$
which does not depend on the choice of 
decomposition 
$V^{\nu^{\prime}} = V^{\nu} \oplus V^{\nu^{\prime} - \nu}$  .
Via this natural embedding, we regard  
$\M_{0}(\nu, \lambda)$ as a
closed subvariety of $\M_{0}(\nu^{\prime}, \lambda)$. 
We consider the union of them and 
obtain the following combined morphism: 
$$
\pi \colon \M(\lambda) 
:= \bigsqcup_{\nu} \M(\nu, \lambda)
\to 
\M_{0}(\lambda) 
:= \bigcup_{\nu} \M_{0}(\nu, \lambda).
$$  
For each $x \in \M_{0}(\lambda)$,
let $\M(\lambda)_{x} := \pi^{-1}(x)$
denote the fiber of $x$.
The fiber 
$\LL(\lambda) := \pi^{-1}(0)$
of the origin $0 \in \M_{0}(\lambda)$ 
is called the {\em central fiber}.
We also set
$\M(\nu, \lambda)_{x} := \M(\lambda)_{x} \cap
\M(\nu, \lambda)$ and
$\LL(\nu, \lambda) := \LL(\lambda) \cap \M(\nu, \lambda).$

Recall that 
the geometric points of
$\M_{0}(\nu, \lambda)$       
correspond to the closed $G(\nu)$-orbits in $\mu^{-1}(0)$.
Let $\M_{0}^{\mathrm{reg}}(\nu, \lambda)$
be the
subset of $\M_{0}(\nu, \lambda)$ 
consisting of the closed $G(\nu)$-orbits
containing the elements  
$\mathbf{x} = (B, a, b) \in \mu^{-1}(0)$ with
trivial stabilizers (i.e.,~$\Stab_{G(\nu)} \mathbf{x} = \{1\}$).
This is a (possibly empty) non-singular open subset 
of $\M_{0}(\nu, \lambda)$, 
on which the morphism $\pi$
becomes an isomorphism of varieties
$\pi^{-1}(\Mreg(\nu, \lambda)) \xrightarrow{\cong} 
\Mreg(\nu, \lambda).$
It is known that 
$\M_{0}^{\mathrm{reg}}(\nu, \lambda) \neq 
\varnothing$
if and only if $\lambda - \nu$
is a dominant weight 
appearing in the finite-dimensional irreducible
$\g$-module of highest weight $\lambda$. 
They form a stratification:  
\begin{equation}
\label{Eq:stratif1}
\M_{0}(\lambda) = 
\bigsqcup_{\nu \in \cQ^{+}; \lambda - \nu \in \cP^{+}} 
\M_{0}^{\mathrm{reg}}(\nu, \lambda).
\end{equation}
The closure inclusion
$\M_{0}^{\mathrm{reg}}(\nu, \lambda) \subset
\overline{\M_{0}^{\mathrm{reg}}(\nu^{\prime}, \lambda)}$
implies $\nu \le \nu^{\prime}$.


\subsection{Graded quiver varieties}
\label{Ssec:graded}

Fix an element 
$\hnu = \sum_{(i,p) \in \widehat{J}} 
n_{i,p} \alpha_{i,p} \in \lQ^{+}_{\Z}$
and 
an $\ell$-dominant $\ell$-weight
$\hlam = \sum_{(i,p) \in \widehat{I}} 
l_{i,p} \varpi_{i,p}
\in \lP^{+}_{\Z}$.
Consider a $\widehat{J}$-graded $\C$-vector space
$
V^{\hnu} = \bigoplus_{(i,p) \in \widehat{J}} 
V^{\hnu}_{i}(p)
$
with $\dim V^{\hnu}_{i}(p) = n_{i,p}$ for $(i,p) \in \widehat{J}$,
and an $\widehat{I}$-graded $\C$-vector space
$
W^{\hlam} = \bigoplus_{(i,p) \in \widehat{I}} W^{\hlam}_{i}(p)
$
with $\dim W^{\hlam}_{i}(p) = l_{i,p}$ for $(i,p) \in \widehat{I}$.
We consider the following space of linear maps: 
\begin{multline}
\mathbf{M}^{\bullet}(V^{\hnu}, W^{\hlam})
:=
\left(
\bigoplus_{(i,p) \in \widehat{J}, j \in I ; i \sim j} 
\Hom (V^{\hnu}_{i}(p), V^{\hnu}_{j}(p-1)) \right) 
\\
\oplus
\left( \bigoplus_{(i,p) \in \widehat{I}} 
\Hom (W^{\hlam}_{i}(p), V^{\hnu}_{i}(p-1)) \right)
\oplus
\left(
\bigoplus_{(i,p) \in \widehat{J}} 
\Hom (V^{\hnu}_{i}(p), W^{\hlam}_{i}(p-1)) \right).
\nonumber
\end{multline}
According to this direct sum expression,
we write an element of $\mathbf{M}^{\bullet}(V^{\hnu}, W^{\hlam})$ 
as a triple $(B,a,b)$ of linear maps $B = \bigoplus B_{ji}(p)$,
$a = \bigoplus a_{i}(p)$ and $b = \bigoplus b_{i}(p)$.
Let $\mu^{\bullet} = \bigoplus_{(i,p) \in \widehat{J}} \mu_{i,p} 
\colon \mathbf{M}^{\bullet}
(V^{\hnu},W^{\hlam}) \to \bigoplus_{(i, p) \in \widehat{J}} 
\Hom(V^{\hnu}_{i}(p), V^{\hnu}_{i}(p-2))$
be the map defined by the formula
$$
\mu_{i,p}^{\bullet}(B, a, b)
= a_{i}(p-1) b_{i}(p)
+\sum_{j \sim i} \varepsilon(i,j) B_{ij}(p-1)B_{ji}(p),
$$
where $\varepsilon(i,j)$ is the same as 
in Subsection \ref{Ssec:usual}. 
The map $\mu^{\bullet}$
is equivariant with respect to 
the conjugate action of the group
$G(\hnu) := 
\prod_{(i,p) \in \widehat{J}} GL(V^{\hnu}_{i}(p))$.
A point $(B, a, b) \in \mu^{\bullet \, -1}(0)$ is said to be stable
if there exists no non-zero $\widehat{J}$-graded subspace 
$V^{\prime} \subset V^{\hnu}$ such that 
$B(V^{\prime}) \subset V^{\prime}$ and 
$V^{\prime} \subset \Ker b$.
Let $\mu^{\bullet \, -1}(0)^{\mathrm{st}}$ be 
the set of stable points. 
Similarly as in Subsection \ref{Ssec:usual}, 
we consider two kinds of quotients
$\M^{\bullet}(\hnu, \hlam) 
:= \mu^{\bullet \, -1}(0)^{\mathrm{st}} / G(\hnu)$
and 
$\M_{0}^{\bullet}(\hnu, \hlam) 
:= \mu^{\bullet \,-1}(0)/\!/G(\hnu),
$
together with the canonical projective morphism
$\pi^{\bullet} \colon \M^{\bullet}(\hnu, \hlam) 
\to \M^{\bullet}_{0}(\hnu, \hlam).$
We refer to these varieties
$\M^{\bullet}(\hnu, \hlam), \M^{\bullet}_{0}(\hnu, \hlam)$ 
as {\em graded quiver varieties.}

On the space $\mathbf{M}^{\bullet}(V^{\hnu}, W^{\hlam})$, 
we have the conjugate action
of the group
$
G(\hlam) := 
\prod_{(i,p) \in \widehat{I}} GL(W^{\hlam}_{i}(p))
$ 
and 
the scalar action of 
$\C^{\times}$.
The combined action of the group
$\G(\hlam) 
:= G(\hlam) \times \C^{\times}$
on $\mathbf{M}(V^{\hnu},W^{\hlam})$
induces actions on the quotients
$\M(\hnu, \hlam), 
\M_{0}(\hnu, \hlam)$
which make the canonical morphism $\pi^{\bullet}$ into
a $\G(\hlam)$-equivariant morphism.  
As in Subsection \ref{Ssec:usual}, 
we can form the unions:
$$
\pi^{\bullet} \colon
\M^{\bullet}(\hlam) 
:= \bigsqcup_{\hnu} 
\M^{\bullet}(\hnu, \hlam)
\to
\M_{0}^{\bullet}(\hlam)
:= \bigcup_{\hnu}\M_{0}^{\bullet}
(\hnu, \hlam).
$$
Let $\M^{\bullet}(\hlam)_{x} := \pi^{\bullet \, -1}(x)$
denote the fiber of a point $x \in \M_{0}(\lambda)$.
We set
$\LL^{\bullet}(\hlam) := \pi^{\bullet \, -1}(0)$.

\subsection{Identification with fixed point subvarieties}
\label{Ssec:fixed}

Let $\hlam = \sum l_{i,p} \varpi_{i,p} \in \lP^{+}_{\Z}$
be an $\ell$-dominant $\ell$-weight.
In this subsection, we recall that the 
graded quiver varieties 
$\M^{\bullet}(\hlam), 
\M_{0}^{\bullet}(\hlam)$ can be realized
as fixed point subvarieties
of the usual quiver varieties 
$\M(\lambda), \M_{0}(\lambda)$
with $\lambda := \mathsf{cl}(\hlam)$
with respect to a certain torus action.

We have $\lambda = \sum_{i \in I} l_{i} \varpi_{i}$
with $l_{i} = \sum_{p \in 2\Z +\xi_{i}} l_{i,p}$
by the definition of $\mathsf{cl}$.
For each $i \in I$, 
we choose a direct sum decomposition 
$W^{\lambda}_{i} = \bigoplus_{p \in 2\Z + \xi_{i}} W^{\hlam}_{i}(p)$
such that $\dim W^{\hlam}_{i}(p) = l_{i,p}$.
Note that this choice specifies a group embedding
$G(\hlam) \hookrightarrow G(\lambda).$
Define a $1$-parameter subgroup
$f_{i} \colon \C^{\times} \to GL(W^{\lambda}_{i})$
by 
$f_{i} (t) |_{W^{\hlam}_{i}(p)} 
:= t^{p} \cdot \mathrm{id}_{W^{\hlam}_{i}(p)}$
for $t \in \C^{\times}$
and a $1$-dimensional subtorus $\T := (\prod_{i \in I}f_{i} \times \mathrm{id})
(\C^{\times})$ of $\G(\lambda) $.
Then we
consider the subvarieties 
$\M(\lambda)^{\T}, 
\M_{0}(\lambda)^{\T}$
consisting of $\T$-fixed points 
and 
the induced canonical morphism
$\pi^{\T} \colon 
\M(\lambda)^{\T} \to
\M_{0}(\lambda)^{\T}
$.
Since the centralizer of $\T$ in 
$\G(\lambda)$ is 
identical to 
the subgroup
$\G(\hlam) = G(\hlam) \times \C^{\times}
\subset \G(\lambda)$, 
we have the induced action of $\G(\hlam)$
on the $\T$-fixed point subvarieties $\M(\lambda)^{\T}, 
\M_{0}(\lambda)^{\T}$.
The morphism
$\pi^{\T}$ is $\G(\hlam)$-equivariant.

On the other hand,
for each $\hnu = \sum n_{i,p} \alpha_{i,p} \in \lQ^{+}_{\Z}$,
we fix a direct sum decomposition 
$V_{i}^{\nu} = \bigoplus_{p \in 2\Z + \xi_{i} + 1}V^{\hnu}_{i}(p)$
of $I$-graded vector space $V^{\nu}$ with 
$\nu := \mathsf{cl}(\hnu)$
such that $\dim V_{i}^{\hnu}(p) = n_{i,p}$, 
just as we have done for $W^{\lambda}$ in the last paragraph.
These direct sum decompositions induce
an embedding 
$\iota_{\hnu, \hlam}\colon
\mathbf{M}^{\bullet}(V^{\hnu}, W^{\hlam})
\hookrightarrow 
\mathbf{M}(V^{\nu}, W^{\lambda}).$
After taking the quotients,
this embedding $\iota_{\hnu, \hlam}$ yields the morphisms  
$\M^{\bullet}(\hnu, \hlam) \to \M(\nu, \lambda)^{\T}$ and
$\M^{\bullet}_{0}(\hnu, \hlam)
\to \M_{0}(\nu, \lambda)^{\T}.$

\begin{Lem}
\label{Lem:fixed}
The above morphisms induce
$\G(\hlam)$-equivariant
isomorphisms $\M^{\bullet}(\hlam)
\xrightarrow{\cong} \M(\lambda)^{\T},
\M_{0}^{\bullet}(\hlam)
\xrightarrow{\cong}
\M_{0}(\lambda)^{\T}
$
which make the following diagram commute:
$$
\xy
\xymatrix{
\M^{\bullet}(\hlam) \ar[r]^{\cong}
\ar[d]_{\pi^{\bullet}} 
& \M(\lambda)^{\T}
\ar[d]^{\pi^{\T}}\\
\M_{0}^{\bullet}(\hlam) \ar[r]^{\cong} & 
\M_{0}(\lambda)^{\T}.}
\endxy$$
In particular, we have a $\G(\hlam)$-equivariant isomorphism
$
\LL^{\bullet}(\hlam) \cong \LL(\lambda)^{\T}.
$ 
\end{Lem} 
\begin{proof}
See \cite[Section 4]{Nakajima01}.
\end{proof}

Hereafter,
we identify the graded quiver varieties 
$\M^{\bullet}_{0}(\hlam), \M^{\bullet}(\hlam)$
with the $\T$-fixed point subvarieties 
$\M_{0}(\lambda)^{\T}, \M(\lambda)^{\T}$
via the isomorphisms in Lemma~\ref{Lem:fixed}.
Then we have
\begin{equation}
\label{Eq:nuhatnu}
\M(\nu, \lambda)^{\T}
= \bigsqcup_{\hnu \in \lQ_{\Z}^{+} ; 
\mathsf{cl}(\hnu) = \nu}
\M^{\bullet}(\hnu, \hlam),
\quad
\M_{0}(\nu, \lambda)^{\T}
= \bigsqcup_{\hnu \in \lQ_{\Z}^{+} ; 
\mathsf{cl}(\hnu) = \nu}
\M_{0}^{\bullet}(\hnu, \hlam). 
\end{equation}

We define
$\M_{0}^{\bullet \, \mathrm{reg}}
(\hnu, \hlam)
:=
\M_{0}^{\bullet}
(\hnu, \hlam)
\cap 
\M_{0}^{\mathrm{reg}}
(\nu, \lambda).$
It is known that
$\M_{0}^{\bullet \, \mathrm{reg}}
(\hnu, \hlam) \neq \varnothing$
if and only if
$\hlam - \hnu$ is an $\ell$-dominant
$\ell$-weight appearing in the 
local Weyl module $W(\hlam)$.
By (\ref{Eq:stratif1}) and (\ref{Eq:nuhatnu}),  
we get a stratification:
\begin{equation}
\label{Eq:stratification_graded}
\M_{0}^{\bullet}
(\hlam)
= \bigsqcup_{\hnu \in \lQ_{\Z}^{+} ; 
\hlam
-\hnu \in \lP_{\Z}^{+}}
\M_{0}^{\bullet \, \mathrm{reg}}
(\hnu, \hlam).
\end{equation}
The closure inclusion
$
\M_{0}^{\bullet \, \mathrm{reg}}(\hnu_{1}, \hlam)
\subset
\overline{\M_{0}^{\bullet \, \mathrm{reg}}(\hnu_{2}, \hlam)}
$
implies $\hnu_{1} \le \hnu_{2}$.


\subsection{Structure of non-central fibers}
\label{Ssec:fibers}

In this subsection, we recall
the structure of (non-central)
fibers of the canonical morphisms $\pi$ and $\pi^{\bullet}$.
Our exposition is based on 
\cite[Section 6]{Nakajima94},
\cite[Section 3]{Nakajima01}
and \cite[Section 2.7]{Nakajima09}
with some more details
on the group actions. 

Let $(\nu, \lambda) \in \cQ^{+} \times \cP^{+}$
be a pair.
For any triple $\mathbf{x} =(B,a,b)
\in \mu^{-1}(0) \subset \mathbf{M}(V^{\nu}, W^{\lambda}), 
$
we consider the following 
two kinds of complexes of vector spaces:
\begin{align}
\label{Eq:three_pt}
C_{i}(\nu, \lambda)_{\mathbf{x}}\colon \quad &
V^{\nu}_{i} \xrightarrow{\sigma_{i}}
W^{\lambda}_{i} \oplus \bigoplus_{j \sim i} V^{\nu}_{j}
\xrightarrow{\tau_{i}}
V^{\nu}_{i}
\quad \text{for each $i \in I$},
\\
\intertext{where we define 
$\sigma_{i} := b_{i} \oplus \bigoplus_{j} B_{ji}$
and
$ 
\tau_{i} := a_{i} + \sum_{j} \varepsilon(i,j) B_{ij}$;}
\label{Eq:three_pt2}
\mathscr{C}(\nu, \lambda)_{\mathbf{x}}\colon \quad &
\bigoplus_{i \in I}\End(V^{\nu}_{i})
\xrightarrow{\iota}
\mathbf{M}(V^{\nu}, W^{\lambda})
\xrightarrow{\mathrm{d}\mu}
\bigoplus_{i \in I}\End(V^{\nu}_{i}),
\end{align}
where $\iota$ is given by
$\iota(\xi) = (B\xi - \xi B)\oplus (-\xi a) \oplus (b \xi)$
and $\mathrm{d}\mu$ is the differential of the
moment map $\mu = \bigoplus_{i} \mu_{i}$ 
at the point $\mathbf{x} = (B,a,b)$.
Note that 
the middle cohomology
$H^{0}(\mathscr{C}(\nu, \lambda)_{\mathbf{x}})$ 
of the complex (\ref{Eq:three_pt2})
is identical to 
the quotient space $(T_{\mathbf{x}} G(\nu) \mathbf{x})^{\perp}
/ T_{\mathbf{x}}G(\nu) \mathbf{x},
$
where $(T_{\mathbf{x}} G(\nu) \mathbf{x})^{\perp}$
is the symplectic perpendicular of 
the tangent space
$T_{\mathbf{x}} G(\nu) \mathbf{x}$
of the $G(\nu)$-orbit of $\mathbf{x}$.  
In particular, if $\mathbf{x}$ is stable,
the space $H^{0}(\mathscr{C}(\nu, \lambda)_{\mathbf{x}})$
is isomorphic to the tangent space 
$T_{x} \M(\nu, \lambda)$ of 
the point $x \in \M(\nu, \lambda)$ corresponding to $\mathbf{x}$.

Let $(\nu, \lambda) \in \cQ^{+} \times \cP^{+}$ be a pair
such that $\M_{0}^{\mathrm{reg}}(\nu, \lambda) \neq 
\varnothing$.
Recall that we have $\lambda - \nu \in \cP^{+}$ in this case.
We fix a point $x_{\nu} \in \M_{0}^{\mathrm{reg}}(\nu, \lambda)$
and its lift 
$\mathbf{x}_{\nu} \in \mu^{-1}(0) \subset 
\mathbf{M}(V^{\nu}, W^{\lambda})$ whose
$G(\nu)$-orbit is closed. 
Then in the complex $C_{i}(\nu, \lambda)_{\mathbf{x}_{\nu}}$,
the map
$\sigma_{i}$
is injective (\cite[Proposition 3.24]{Nakajima98})
and the map $\tau_{i}$
is surjective (\cite[Lemma 4.7]{Nakajima98}).
In particular, the dimension of the middle cohomology 
$H^{0}(C_{i}(\nu, \lambda)_{\mathbf{x}_{\nu}})
= \Ker \tau_{i} / \Ima \sigma_{i}$
is equal to 
$
(\lambda - \nu)(h_{i})$.
Therefore we can identify $W^{\lambda - \nu}_{i}
= H^{0}(C_{i}(\nu, \lambda)_{\mathbf{x}_{\nu}})$.

We pick 
an arbitrary element $\nu^{\prime}$
such that $\nu \le \nu^{\prime}$.
In order to construct the natural embedding
$\M_{0}(\nu, \lambda) \hookrightarrow
\M_{0}(\nu^{\prime}, \lambda),$
we fix a direct sum decomposition
$V^{\nu^{\prime}} = V^{\nu} \oplus V^{\nu^{\prime} - \nu}$.
Extending by $0$ on $V^{\nu^{\prime} - \nu}$,
we have an injective linear map
$\mathbf{M}(V^{\nu}, W^{\lambda})
\hookrightarrow
\mathbf{M}(V^{\nu^{\prime}}, W^{\lambda}),$
by which our fixed element $\mathbf{x}_{\nu} = (B,a,b)$ 
is regarded as an element of
$\mu^{-1}(0) \subset
\mathbf{M}(V^{\nu^{\prime}}, W^{\lambda}).$
Then we can calculate as
\begin{equation}
\label{Eq:middlecoh}
H^{0}(\mathscr{C}(\nu^{\prime}, \lambda)_{\mathbf{x}_{\nu}})
\cong
\mathbf{M}(V^{\nu^{\prime} - \nu}, W^{\lambda - \nu}) \oplus
H^{0}(\mathscr{C}(\nu, \lambda)_{\mathbf{x}_{\nu}}),
\end{equation}
where we have
$W^{\lambda - \nu}_{i}
= H^{0}(C_{i}(\nu, \lambda)_{\mathbf{x}_{\nu}})$.
We also see that the space 
$H^{0}(\mathscr{C}(\nu, \lambda)_{\mathbf{x}_{\nu}})$
is isomorphic to the tangent space 
$T:= T_{x_{\nu}} \M_{0}^{\mathrm{reg}}(\nu, \lambda)$.

The stabilizer $\Stab_{G(\nu^{\prime})} \mathbf{x}_{\nu}$
is naturally isomorphic to $G(\nu^{\prime} - \nu)$.
Under this isomorphism,
the action of $\Stab_{G(\nu^{\prime})} \mathbf{x}_{\nu}$
on the LHS of (\ref{Eq:middlecoh})
coincides with the action of $G(\nu^{\prime} - \nu)$ on 
the RHS of (\ref{Eq:middlecoh}),
which is the direct sum of the natural action on 
$\mathbf{M}(V^{\nu^{\prime} - \nu}, W^{\lambda - \nu})$ 
and the trivial action on 
$T \cong H^{0}(\mathscr{C}(\nu, \lambda)_{\mathbf{x}_{\nu}})$.

An appropriate Hamiltonian reduction 
with respect to 
the action of the group 
$\Stab_{G(\nu^{\prime})}\mathbf{x}_{\nu}
\cong G(\nu^{\prime} - \nu)$
on the RHS of (\ref{Eq:middlecoh}) yields the following
canonical map: 
$$\pi \times \mathrm{id}
\colon  \M(\nu^{\prime} - \nu, \lambda-\nu) \times T
\to
 \M_{0}(\nu^{\prime} - \nu, \lambda-\nu) \times T.
$$
According to the discussion in \cite[Section 3]{Nakajima01},
this gives a local description of 
$\pi \colon \M(\nu^{\prime} , \lambda) \to 
\M_{0}(\nu^{\prime}, \lambda)$
around the point 
$x_{\nu} \in \M_{0}^{\mathrm{reg}}(\nu , \lambda) 
\subset \M_{0}(\nu^{\prime}, \lambda).$ 
In precise, we have the following theorem.

\begin{Thm}[Nakajima~{\cite[Theorem 3.3.2]{Nakajima01}}]
\label{Thm:Local}
Let $x_{\nu} \in \M_{0}^{\mathrm{reg}}(\nu, \lambda)
\subset \M_{0}(\nu^{\prime}, \lambda).$
Then there exist neighborhoods $U, U_{S}, U_{T}$ 
of $x_{\nu} \in \M_{0}(\nu^{\prime}, \lambda), \,
0 \in \M_{0}(\nu^{\prime} - \nu, \lambda - \nu), \,
0 \in T:= T_{x_{\nu}} \M_{0}^{\mathrm{reg}}(\nu, \lambda)$ 
respectively and biholomorphic maps
$U \xrightarrow{\cong} U_{S} \times U_{T}; 
x_{\nu} \mapsto (0,0)$ and
$\pi^{-1}(U) \xrightarrow{\cong} \pi^{-1}(U_{S}) \times U_{T}$
such that the following diagram commutes:
$$
\xy
\xymatrix{
\M(\nu^{\prime}, \lambda) 
\ar@{}[r]|{\supset} &
\pi^{-1}(U)
\ar[r]^-{\cong}
\ar[d]_-{\pi} 
& \pi^{-1}(U_{S}) \times U_{T} 
\ar[d]^-{\pi \times \mathrm{id}} 
\ar@{}[r]|{\subset} 
&
\M(\nu^{\prime} - \nu, \lambda-\nu)
\times T \\
\M_{0}(\nu^{\prime}, \lambda) 
\ar@{}[r]|{\supset} &
U
\ar[r]^-{\cong}
& U_{S} \times U_{T} \ar@{}[r]|{\subset} 
&
\M_{0}(\nu^{\prime} - \nu, \lambda-\nu)
\times T.
}
\endxy$$ 
\end{Thm}

Now let us consider the action of the group 
$\Stab_{\G(\lambda)} x_{\nu}$ 
on the fiber $\M(\lambda)_{x_{\nu}}$. 
By the definition of $\M_{0}^{\mathrm{reg}}(\nu, \lambda)$,
we have $\Stab_{G(\nu)} \mathbf{x}_{\nu} = \{ 1 \}.$
Therefore the second projection
$G(\nu) \times \G(\lambda) \to \G(\lambda)$
restricts to an isomorphism 
$r \colon \Stab_{G(\nu) \times \G(\lambda)} \mathbf{x}_{\nu}
\xrightarrow{\cong}
\Stab_{\G(\lambda)} x_{\nu}$.
Via the fixed decomposition
$V^{\nu^{\prime}} = V^{\nu} \oplus V^{\nu^{\prime} - \nu}$,
we regard the group 
$\Stab_{G(\nu) \times \G(\lambda)} \mathbf{x}_{\nu}$
as a subgroup of 
$\Stab_{G(\nu^{\prime}) \times \G(\lambda)}\mathbf{x}_{\nu}$.
In fact, we have
\begin{equation}
\label{Eq:Stabdec}
\Stab_{G(\nu^{\prime}) \times \G(\lambda)}\mathbf{x}_{\nu}
\cong 
\Stab_{G(\nu) \times \G(\lambda)} \mathbf{x}_{\nu}
\times G(\nu^{\prime} -\nu).
\end{equation}
Thus the group 
$\Stab_{G(\nu) \times \G(\lambda)} \mathbf{x}_{\nu}$
acts on the vector space
$\mathbf{M}(\nu^{\prime},\lambda)$
and also on 
the middle cohomology (\ref{Eq:middlecoh}) of the complex 
$\mathscr{C}(\nu^{\prime}, \lambda)_{\mathbf{x}_{\nu}}$.
Note that this induced action 
preserves each summand
of the RHS of (\ref{Eq:middlecoh}).
In particular, we obtain an action of the group 
$\Stab_{G(\nu) \times \G(\lambda)} \mathbf{x}_{\nu}$
on the vector space 
$\mathbf{M}(V^{\nu^{\prime} - \nu}, W^{\lambda - \nu})$.
By the construction, 
we can easily see that 
this action factors through
the natural action of $\G(\lambda - \nu)$
on $\mathbf{M}(V^{\nu^{\prime} - \nu}, W^{\lambda - \nu})$.
The corresponding group homomorphism 
$\Stab_{G(\nu) \times \G(\lambda)} \mathbf{x}_{\nu}
\to \G(\lambda - \nu) = G(\lambda - \nu) \times \C^{\times}$
is the direct product of two homomorphisms
$\varphi \colon \Stab_{G(\nu) \times \G(\lambda)} \mathbf{x}_{\nu}
\to G(\lambda - \nu)$
and
$\rho \colon \Stab_{G(\nu) \times \G(\lambda)} \mathbf{x}_{\nu}
\to \C^{\times}$.
The homomorphism $\varphi$ is given as 
the induced action of the group
$ \Stab_{G(\nu) \times \G(\lambda)} \mathbf{x}_{\nu}$
on the middle cohomology of the complex
$C_{i}(\nu, \lambda)_{\mathbf{x}_{\nu}}$ under 
the identification 
$W^{\lambda - \nu}_{i}
= H^{0}(C_{i}(\nu, \lambda)_{\mathbf{x}_{\nu}})$.
The homomorphism $\rho$
is obtained by the projection,
$$
\rho \colon \Stab_{G(\nu) \times \G(\lambda)} \mathbf{x}_{\nu}
\hookrightarrow
G(\nu) \times \G(\lambda) =
G(\nu) \times G(\lambda) \times \C^{\times} 
\xrightarrow{\mathrm{pr}_{3}}
\C^{\times}. 
$$
The action of the group
$\Stab_{G(\nu) \times \G(\lambda)} \mathbf{x}_{\nu}$
on the space $\mathbf{M}(V^{\nu^{\prime} - \nu}, W^{\lambda - \nu})$
commutes with the action of group
$G(\nu^{\prime} - \nu)$.
After taking the Hamiltonian reductions  
with respect to the action of the group
$G(\nu^{\prime} - \nu)$,  
we obtain an action
of the group 
$\Stab_{\G(\lambda)} x_{\nu}$
on the central fiber 
$\LL(\nu^{\prime} - \nu, \lambda - \nu)$.
The above argument says that 
this action factors through the group homomorphism
\begin{equation}
\label{Eq:Stabhom}
\Stab_{\G(\lambda)} x_{\nu}
\xrightarrow[\cong]{r^{-1}}
\Stab_{G(\nu) \times \G(\lambda)} \mathbf{x}_{\nu}
\xrightarrow{\varphi \times \rho}
\G(\lambda-\nu).
\end{equation}
This homomorphism (\ref{Eq:Stabhom})
does not depend on $\nu^{\prime}$.

By Theorem~\ref{Thm:Local},
there is an isomorphism
$
\M(\nu^{\prime}, \lambda)_{x_{\nu}} 
\xrightarrow{\cong} \LL(\nu^{\prime} - \nu, \lambda - \nu).
$
As stated in \cite[Remark 3.3.3]{Nakajima01},
this isomorphism
can be equivariant with respect to the
actions of the group 
$\Stab_{\G(\lambda)} x_{\nu}$. 
Summing up over $\nu^{\prime}$, 
we obtain the following.

\begin{Lem}
\label{Lem:fiberisom}
Let $(\nu, \lambda) \in\cQ^{+} \times \cP^{+}$
be a pair such that $\M_{0}^{\mathrm{reg}}(\nu, \lambda) \neq
\varnothing$ and 
$\pi \colon \M(\lambda) \to \M_{0}(\lambda)$ be 
the canonical morphism.
Then for each point 
$x_{\nu} \in \M_{0}^{\mathrm{reg}}(\nu, \lambda)$,
there exists a 
$( \Stab_{\G(\lambda)} x_{\nu} )$-equivariant
isomorphism 
$$
\M(\lambda)_{x_{\nu}} \cong \LL(\lambda - \nu), 
$$ 
where the group $\Stab_{\G(\lambda)} x_{\nu}$ acts 
on the RHS $\LL(\lambda - \nu)$
via the group homomorphism 
$(\varphi \times \rho) \circ r^{-1}$ in
(\ref{Eq:Stabhom}).
\end{Lem}    

Next we consider the graded versions.
Let $(\hnu, \hlam) \in \lQ^{+}_{\Z} \times \lP^{+}_{\Z}$
be a pair.
For any triple 
$\mathbf{x} = (B,a,b) \in \mu^{\bullet\, -1}(0)
\subset \mathbf{M}^{\bullet}(V^{\hnu}, W^{\hlam})$  
and $(i , p) \in \widehat{I}$,
we can consider a complex of vector spaces
\begin{equation}
\nonumber
C_{i,p}(\hnu, \hlam)_{\mathbf{x}}\colon \;
V_{i}^{\hnu}(p+1)
\xrightarrow{\sigma_{i, p}}
W^{\hlam}_{i}(p) \oplus \bigoplus_{j \sim i} 
V^{\hnu}_{j}(p)
\xrightarrow{\tau_{i,p}}
V^{\hnu}_{i}(p-1),
\end{equation}
where
we define 
$\sigma_{i,p} := b_{i}(p+1) \oplus \bigoplus_{j} B_{ji}(p+1)$
and
$ 
\tau_{i,p} := a_{i}(p) + \sum_{j} \varepsilon(i,j) B_{ij}(p)$.

Now we assume that $\M_0^{\bullet\, \mathrm{reg}}(\hnu, \hlam) \neq 
\varnothing$.
In particular, we have $\hlam - \hnu \in \lP^{+}_{\Z}.$ 
We fix a point
$x_{\hnu} \in \M_{0}^{\bullet \, \mathrm{reg}}(\hnu, \hlam)$
and its lift 
$\mathbf{x}_{\hnu} \in 
\mu^{\bullet \, -1}(0) \subset
\mathbf{M}^{\bullet}(V^{\hnu}, W^{\hlam})$
whose $G(\hnu)$-orbit is closed.
By the same reason as in the non-graded case,
in the complex $C_{i,p}(\hnu, \hlam)_{\mathbf{x}_{\hnu}}$,
the map $\sigma_{i,p}$ is injective 
and the map $\tau_{i,p}$ is surjective.
Therefore the 
dimension vector of the 
$\widehat{I}$-graded vector space
$\bigoplus_{(i,p) \in \widehat{I}} 
H^{0}(C_{i,p}(\hnu, \hlam)_{\mathbf{x}_{\hnu}})$
is equal to $\hlam - \hnu$.
This allows us to identify
$W^{\hlam - \hnu}_{i}(p)$ with
$H^{0}(C_{i,p}(\hnu, \hlam)_{\mathbf{x}_{\hnu}})$
for each $(i,p) \in \widehat{I}$.
Similarly as in (\ref{Eq:Stabhom}),
we consider the following group homomorphism
\begin{equation}
\label{Eq:Stabhom2}
\Stab_{\G(\hlam)} x_{\hnu}
\xrightarrow[\cong]{\hat{r}^{-1}}
\Stab_{G(\hnu) \times \G(\hlam)} 
\mathbf{x}_{\hnu}
\xrightarrow{\hat{\varphi} \times \hat{\rho}}
\G(\hlam - \hnu)
\end{equation}
where
$\hat{r}$ is the isomorphism obtained 
as the restriction of the 
projection $G(\hnu) \times \G(\hlam) \to \G(\hlam)$,
$\hat{\varphi}$ is given as the induced action of
the group $\Stab_{G(\hnu) \times \G(\hlam)} 
\mathbf{x}_{\hnu}$
on $
W^{\hlam-\hnu}_{i}(p) = 
H^{0}(C_{i,p}(\hnu, \hlam)_{\mathbf{x}_{\hnu}})$ 
and $\hat{\rho}$ is the restriction of 
the projection $G(\hnu) \times G(\hlam) \times \C^{\times}
\to \C^{\times}$.   

\begin{Lem}
\label{Lem:fiberisom2}
Let $(\hnu, \hlam) \in\lQ^{+}_{\Z} \times \lP^{+}_{\Z}$
be a pair such that $\M_{0}^{\bullet \, 
\mathrm{reg}}(\hnu, \hlam) \neq
\varnothing$.
Then for each point 
$x_{\hnu} \in \M_{0}^{\bullet \, \mathrm{reg}}(\hnu, \hlam)$,
there exists a
$( \Stab_{\G(\hlam)} x_{\hnu} )$-equivariant
isomorphism 
$$
\M^{\bullet}(\hlam)_{x_{\hnu}} \cong \LL^\bullet(\hlam - \hnu), 
$$ 
where the group $\Stab_{\G(\hlam)} x_{\hnu}$ acts 
on the RHS $\LL^{\bullet}(\hlam - \hnu)$
via the group homomorphism 
$(\hat{\varphi} \times \hat{\rho}) \circ \hat{r}^{-1}$ in
(\ref{Eq:Stabhom2}).
\end{Lem}  

For a proof, we use the description in the non-graded case (Lemma~\ref{Lem:fiberisom})
and the identification with the $\T$-fixed point subvarieties (Lemma~\ref{Lem:fixed}).
\begin{proof}
We put $\nu:=\mathsf{cl}(\hnu), \lambda:=\mathsf{cl}(\hlam)$.
We make identifications of vector spaces: 
$W_{i}^{\lambda} = \bigoplus_{p \in 2\Z + \xi_{i}} W^{\hlam}_{i}(p),
V_{i}^{\nu} = \bigoplus_{p \in 2\Z + \xi_{i} + 1} V^{\hnu}_{i}(p),$
which specifies an embedding
$\iota \equiv \iota_{\hnu, \hlam}
\colon \mathbf{M}^{\bullet}(V^{\hnu}, W^{\hlam})
\hookrightarrow
\mathbf{M}(V^{\nu}, W^{\lambda})$.
Using these direct sum decompositions,
we define a group homomorphism
$f_{i} \colon \C^{\times} \to GL(W_{i}^{\lambda})$
(resp.~$g_{i} \colon \C^{\times} \to GL(V_{i}^{\nu})$)
for each $i \in I$ 
by $f_{i}(t) |_{W^{\hlam}_{i}(p)}
:= t^{p} \cdot \mathrm{id}_{W^{\hlam}_{i}(p)}$
(resp.~$g_{i}(t) |_{V_{i}^{\hnu}(p)}
:= t^{p} \cdot \mathrm{id}_{V_{i}^{\hnu}(p)}$).
Recall we have 
$\M^{\bullet}(\hlam) \cong \M(\lambda)^{\T}$
and
$\M_{0}^{\bullet}(\hlam) \cong \M_{0}(\lambda)^{\T}$
by Lemma~\ref{Lem:fixed},
where  
$\T := (\prod_{i \in I}f_{i} \times \mathrm{id}) (\C^{\times})$.
Under this identification, 
we also regard 
$x_{\hnu}$ is a point of 
$\M_{0}^{\mathrm{reg}}(\nu, \lambda).$
We can easily see that
the image $\iota(\mathbf{x}_{\hnu}) \in \mu^{-1}(0)
\subset \mathbf{M}(\nu, \lambda)$
has a closed $G(\nu)$-orbit corresponding to 
the point $x_{\hnu} \in \M_{0}^{\mathrm{reg}}(\nu, \lambda)$
and in particular  $\Stab_{G(\nu)} \iota(\mathbf{x}_{\hnu})
= \{ 1\}.$
Let
 $\widetilde{\T} := 
\left(\prod_{i \in I} g_{i} \times \prod_{i \in I}f_{i} 
\times \mathrm{id}\right)
(\C^{\times}).
$
This is a $1$-dimensional subtorus
of 
$\Stab_{G(\nu) \times \G(\lambda)} \iota(\mathbf{x}_{\hnu})$.
Under the isomorphism
$r \colon \Stab_{G(\nu) \times \G(\lambda)}  \iota(\mathbf{x}_{\hnu})
\xrightarrow{\cong}
\Stab_{\G(\lambda)} x_{\hnu}$,
the torus $\widetilde{\T}$ is isomorphic to $\T$.
In fact, we have
$r \circ 
(\prod_{i \in I} g_{i} \times \prod_{i \in I}f_{i} \times \mathrm{id})
= (\prod_{i \in I}f_{i} \times \mathrm{id}).
$

On the other hand, we have a decomposition
$C_{i}(\nu, \lambda)_{\iota(\mathbf{x}_{\hnu})}
\cong \bigoplus_{p \in 2\Z + \xi_{i}} 
C_{i,p}(\hnu, \hlam)_{\mathbf{x}_{\hnu}}$
of complexes and hence
$H^{0}(C_{i}(\nu, \lambda)_{\iota(\mathbf{x}_{\hnu})})
\cong 
\bigoplus_{p \in 2\Z + \xi_{i}} 
H^{0}(
C_{i,p}(\hnu, \hlam)_{\mathbf{x}_{\hnu}}),
$ 
which is identified with
$W^{\lambda - \nu}_{i} = \bigoplus_{p \in 2\Z + \xi_{i}} 
W_{i}^{\hlam - \hnu}(p).$
Let $h_{i} \colon \C^{\times} \to GL(W_{i}^{\lambda - \nu})$
be a group homomorphism
defined by
$h_{i}(t)|_{W^{\hlam - \hnu}_{i}(p)}
:= t^{p} \cdot \mathrm{id}_{W^{\hlam - \hnu}_{i}(p)} 
$ 
and set $\T^{\prime} := 
(\prod_{i \in I} h_{i} \times \mathrm{id})(\C^{\times})
\subset \G(\lambda - \nu).
$
Then we have
$$
(\varphi \times \rho) \circ r^{-1} \circ
\left(\prod_{i \in I}f_{i} \times \mathrm{id}\right) 
=
(\varphi \times \rho) \circ 
\left(\prod_{i \in I} g_{i} \times \prod_{i \in I}f_{i} \times \mathrm{id}
\right) 
= \prod_{i \in I}h_{i} \times \mathrm{id}.
$$
Therefore, under the isomorphism
in Lemma~\ref{Lem:fiberisom},
the action of the torus $\T$ on $\M(\lambda)_{x_{\hnu}}$
coincides with the action of the torus 
$\T^{\prime}$ on $\LL(\lambda - \nu)$.
Therefore, by Lemma~\ref{Lem:fixed}, we have 
\begin{equation}
\label{Eq:fiberisomTfixed}
\M^{\bullet}(\hlam)_{x_{\hnu}} 
= \M(\lambda)_{x_{\hnu}}^{\T} \cong 
\LL(\lambda - \nu)^{\T^{\prime}} \cong \LL^{\bullet}(\hlam - \hnu).
\end{equation}

It remains to show that this isomorphism 
(\ref{Eq:fiberisomTfixed}) is 
$(\Stab_{\G(\hlam)}x_{\hnu})$-equivariant. 
Note that the centralizer of the torus
$\T$ (resp.~$\widetilde{\T}$ , $\T^{\prime}$) 
in $\Stab_{\G(\lambda)} x_{\hnu}$
(resp.~$\Stab_{G(\nu) \times \G(\lambda)} 
\iota(\mathbf{x}_{\hnu})$, $\G(\lambda - \nu)$)
is the subgroup
$\Stab_{\G(\hlam)} x_{\hnu}$
(resp.~$\Stab_{G(\hnu) \times \G(\hlam)} 
\mathbf{x}_{\hnu}$, $\G(\hlam - \hnu)$). 
We have the following commutative diagram:
$$
\xy
\xymatrix{
\Stab_{\G(\lambda)} x_{\hnu}
&
\Stab_{G(\nu) \times \G(\lambda)} 
\iota(\mathbf{x}_{\hnu})
\ar[l]_-{r}^-{\cong}
\ar[r]^-{\varphi \times \rho}
&
\G(\lambda - \nu)
\\
\Stab_{\G(\hlam)} x_{\hnu}
\ar@{^{(}->}[u]
&
\Stab_{G(\hnu) \times \G(\hlam)} 
\mathbf{x}_{\hnu}
\ar[l]_-{\hat{r}}^-{\cong}
\ar[r]^-{\hat{\varphi} \times \hat{\rho}}
\ar@{^{(}->}[u]
&
\G(\hlam - \hnu).
\ar@{^{(}->}[u]
}
\endxy
$$
Because the isomorphism in Lemma~\ref{Lem:fiberisom}
is $( \Stab_{\G(\lambda)}x_{\hnu} )$-equivariant
via the homomorphism $(\varphi \times \rho) \circ r^{-1}$,
the induced isomorphism (\ref{Eq:fiberisomTfixed}) 
on the torus fixed parts is
$( \Stab_{\G(\hlam)}x_{\hnu} )$-equivariant
via the homomorphism $(\hat{\varphi} \times \hat{\rho}) \circ 
\hat{r}^{-1}$ from the above commutative diagram.
\end{proof}


\subsection{Graded quiver varieties for the category 
$\Cc_{Q, \beta}$}
\label{Ssec:identify}

Fix an element 
$\beta = \sum_{i \in I} d_{i} \alpha_{i} \in \cQ^{+}$.
Let 
$$
E_{\beta} : = \bigoplus_{i\to j \in \Omega}
\Hom(\C^{d_{i}}, \C^{d_{j}})
$$
be the space of representations of the quiver $Q$
of dimension vector $\beta$.
The group $G_{\beta} := \prod_{i \in I} \mathop{GL}_{d_{i}}
(\C)$ acts on $E_{\beta}$
by conjugation.
By Gabriel\rq{}s theorem,
the set of $G_{\beta}$-orbits
is in bijection with
the set 
$$\mathrm{KP}(\beta)
:= \left\{ \mathbf{m} = (m_{\alpha})_{\alpha \in \mathsf{R}^{+}} 
\in (\Z_{\ge 0})^{\mathsf{R}^{+}}
\; \middle| \; 
\sum_{\alpha \in \mathsf{R}^{+}} m_{\alpha} \alpha
 = \beta \right\}
$$ of Kostant partitions of $\beta$.
For each Kostant partition $\mathbf{m} \in \KP(\beta)$,
let $\mathbb{O}_{\mathbf{m}}$ denote
the corresponding $G_{\beta}$-orbit.  

Set
$\hlam_{\beta}
:= \sum_{i \in I}
d_{i} \varpi_{\phi(\alpha_{i})}
\in \lP_{Q}^{+}$,
where $\phi$ is the bijection 
$\mathsf{R}^{+} \to \widehat{I}_{Q}$
defined in Subsection \ref{Ssec:quivers}.
We consider the corresponding graded quiver variety
$\M_{0}^{\bullet}(\hlam_{\beta})$.
We identify
$G(\hlam_{\beta})$ with $G_{\beta}$ in the natural way.

We define a map $\mathsf{p} \colon \mathsf{R}^{+} \to \Z$
by $\mathsf{p} (\alpha) := \mathrm{pr}_{2} \circ
\phi (\alpha)$, where
$\mathrm{pr}_{2} \colon I \times \Z \to \Z$ is the second projection.
Using this notation, we define 
a homomorphism
$\rho_{i} \colon \C^{\times} \to GL(\C^{d_i})$ for each $i \in I$
by 
$\rho_{i}(t) := t^{-\mathsf{p}(\alpha_{i})} 
\cdot \mathrm{id}_{\C^{d_{i}}}$.
Then we have a 
group homomorphism
$(\mathrm{id} \times \prod_{i \in I}\rho_{i})
\colon \G(\hlam_{\beta}) = G(\hlam_{\beta}) \times \C^{\times} 
\to G_{\beta}, $
via which $E_{\beta}$ is equipped with a 
$\G(\hlam_{\beta})$-action. 

In \cite{HL15}, Hernandez-Leclerc constructed 
a $\G(\hlam_{\beta})$-equivariant isomorphism of varieties
$\M_{0}^{\bullet}(\hlam_{\beta}) \cong E_{\beta}$.
Their original motivation was to give a geometric interpretation 
to their isomorphism between the quantum Grothendieck ring  of $\Cc_Q$
and the quantized coordinate ring of the unipotent group associated with the positive part of $\g$.
Note that the isomorphism $\M_{0}^{\bullet}(\hlam_{\beta}) \cong E_{\beta}$
is very special for the $\ell$-weight $\hlam_\beta$ because a graded quiver variety $\M_{0}^{\bullet}(\hlam)$
is not an affine space in general. In particular, the restriction to the subcategory $\Cc_Q \subset \Cc_\g$ is crucial here.

Let us recall Hernandez-Leclerc's construction.
By Lemma~\ref{Lem:lrootQ} (\ref{Lem:lrootQ2}),
it is enough to discuss the graded quiver varieties
$\M_{0}^{\bullet}(\hnu, \hlam_{\beta})$ 
with $\hnu \in \lQ^{+}_{Q}$.
We define  a $\C$-algebra $\widetilde{\Lambda}_{Q}$ 
given by the following quiver $\widetilde{\Gamma}_{Q}$ 
with relations.
The quiver $\widetilde{\Gamma}_{Q}$ consists of 
two types of vertices 
$\{ v_{j}(p) \mid (j,p) \in \widehat{J}_{Q}\}
\cup 
\{w_{j}(p) \mid \text{$(j,p) =  \phi(\alpha_{i})$ for some $i \in I$} \}
$
and three types of arrows:
$$ \mathbf{a}_{i}(p) \colon w_{i}(p) \to v_{i}(p-1), \quad
\mathbf{b}_{i}(p) \colon v_{i}(p) \to w_{i}(p-1), $$   
$$\mathbf{B}_{ji}(p) \colon v_{i}(p) \to v_{j}(p-1) \; \text{ for $i \sim j$}.$$
The relations are 
$$\mathbf{a}_{i}(p-1) \mathbf{b}_{i}(p)
+ \sum_{j \sim i}
\varepsilon(i,j) \mathbf{B}_{ij}(p-1) \mathbf{B}_{ji}(p)
= 0 \quad \text{for each $i \in I$}.  
$$
For each $i \in I$,
let $\epsilon_{i} \in \widetilde{\Lambda}_{Q}$ 
denote the idempotent corresponding to 
the vertex $w_{j}(p)$ with $(j,p) = \phi(\alpha_{i})$.
Then \cite[Lemma 9.6]{HL15} proved 
that the algebra 
$\bigoplus_{i,j \in I} \epsilon_{i} \widetilde{\Lambda}_{Q}
\epsilon_{j}$
is identical to the path algebra $\C Q$.
By definition, 
each element $\mathbf{x} = (B, a, b)
\in \mu^{\bullet \, -1}(0) \subset \mathbf{M}^{\bullet}
(V^{\hnu}, W^{\hlam_{\beta}})$
gives a representation of $\widetilde{\Lambda}_{Q}$.
Then 
restricted to 
$
\bigoplus_{i,j \in I} \epsilon_{i} \widetilde{\Lambda}_{Q}
\epsilon_{j},$
it gives a representation of $\C Q$ of dimension vector
$\beta$.
This defines a morphism
$\M_{0}^{\bullet}(\hnu, \hlam_{\beta}) \to E_{\beta}.$

\begin{Thm}[Hernandez-Leclerc~{\cite[Theorem 9.11]{HL15}}]
\label{Thm:HL}
The morphism constructed above gives 
a $\G(\hlam_{\beta})$-equivariant
isomorphism of varieties
$$\Psi_{\beta} \colon 
\M_{0}^{\bullet}(\hlam_{\beta})
\xrightarrow{\cong}
E_{\beta}.$$
\end{Thm}

In the proof of \cite[Theorem 9.11]{HL15}, it was also proved that the stratification \eqref{Eq:stratification_graded} of $\M_{0}^{\bullet}(\hlam_{\beta})$ coincides with the $\G(\hlam_{\beta})$-orbit stratification of $E_\beta$. Here we shall give a precise correspondence between the strata.  
Define a bijection
$f \colon \KP(\beta) \to \lP^{+}_{Q, \beta}$ by 
$$f(\mathbf{m}) := 
\sum_{\alpha \in \mathsf{R}^{+}}
m_{\alpha} \varpi_{\phi(\alpha)}$$
for $\mathbf{m} = (m_{\alpha})_{\alpha \in \mathsf{R}^{+}}
 \in \KP(\beta).$

\begin{Lem}
\label{Lem:strata}
For each $\mathbf{m} \in \KP(\beta)$,
we have $$\Psi_{\beta}(\Mregg(\hlam_{\beta} - f(\mathbf{m}), \hlam_{\beta}))
= \mathbb{O}_{\mathbf{m}}.$$
\end{Lem}
\begin{proof}
For each $\mathbf{m} \in \KP(\beta)$,
there is a unique $\hnu \in \lQ^{+}_{Q}$ (recall Lemma~\ref{Lem:lrootQ})
such that $\Psi_{\beta}(\M_{0}^{\bullet \, \mathrm{reg}}
(\hnu, \hlam_{\beta}) 
) \supset \mathbb{O}_{\mathbf{m}}$ since each stratum $\M_{0}^{\bullet \, \mathrm{reg}}
(\hnu, \hlam_{\beta})$ is stable under the action of $\G(\hlam_{\beta})$.
It suffices to prove that 
$\hnu = \hlam_{\beta} - f(\mathbf{m}).$ 

First we consider the case when $\beta = \alpha \in \mathsf{R}^{+}$
and $\mathbf{m}$ is the Kostant partition 
$\mathbf{m}_{\alpha} := (\delta_{\alpha, \alpha^{\prime}}
)_{\alpha^{\prime} \in \mathsf{R}^{+}}$ 
consisting of the single root $\alpha$.
In this case, the orbit $\mathbb{O}_{\mathbf{m}_{\alpha}}$ is
the unique open dense orbit of $E_{\alpha}$ as 
$\dim \mathbb{O}_{\mathbf{m}_{\alpha}} = \dim E_{\alpha}$ (see \cite[Proposition 4.4.9 (2)]{DW} for example).
Recall that $\Mregg(\hnu, \hlam_{\beta})
\subset \overline{\Mregg(\hnu^{\prime}, \hlam_{\beta})}$  
implies $\hlam_{\beta} - \hnu \geq \hlam_{\beta} - \hnu^{\prime}$.
Since the $\ell$-weight
$\varpi_{\phi(\alpha)}$
is minimum in $\lP^{+}_{Q, \alpha}$,
the corresponding stratum
$\Mregg(\hnu_{\alpha}, \hlam_{\alpha})$
is maximum, 
where we put $\hnu_{\alpha} := \hlam_{\alpha} - \varpi_{\phi(\alpha)}.$
Therefore we have
$\Psi_{\beta}(\Mregg(\hnu_{\alpha}, \hlam_{\alpha}))
\supset \mathbb{O}_{\mathbf{m}_{\alpha}}
$ as desired.

Next we consider general 
$\mathbf{m} = (m_{\alpha})_{\alpha \in \mathsf{R}^{+}} 
\in \KP(\beta).$
For each $\alpha \in \mathsf{R}^{+}$,
we fix an element 
$\mathbf{y}_{\alpha} \in \mu^{\bullet \, -1}(0) \subset
\mathbf{M}^{\bullet}(V^{\hnu_{\alpha}} , W^{\hlam_{\alpha}})
$
such that 
$\mathbf{y}_{\alpha}$ has a closed $G(\hnu_{\alpha})$-orbit
and $\Stab_{G(\hnu_{\alpha})} \mathbf{y}_{\alpha} = \{1\}$
holds. By the previous paragraph, 
the element $\mathbf{y}_{\alpha}$, which is
regarded as a representation of the algebra 
$\widetilde{\Lambda}_{Q}$,
restricts to give an indecomposable 
representation of $\C Q$ isomorphic to $M(\alpha)$.  
We put
\begin{align*}
\mathbf{x}_{\mathbf{m}} := 
\bigoplus_{\alpha \in \mathsf{R}^{+}} 
\mathbf{y}_{\alpha}^{\oplus m_{\alpha}}
& \in \mu^{\bullet \, -1}(0) \\
&\subset
\mathbf{M}^{\bullet} \left( \bigoplus_{\alpha \in \mathsf{R}^{+}}
(V^{\hnu_{\alpha}})^{\oplus m_{\alpha}}, 
\bigoplus_{\alpha \in \mathsf{R}^{+}}
(W^{\hlam_{\alpha}})^{\oplus m_{\alpha}}
\right) 
=
\mathbf{M}^{\bullet}(V^{\hnu_{\mathbf{m}}},
W^{\hlam_{\beta}}),
\end{align*}
where $\hnu_{\mathbf{m}} := \sum_{\alpha \in \mathsf{R}^{+}}
m_{\alpha} \hnu_{\alpha}.
$
Then $\mathbf{x}_{\mathbf{m}}$ defines a closed
$G(\hnu_{\mathbf{m}})$-orbit and 
has a trivial stabilizer.
Hence, the corresponding point $x_{\mathbf{m}}$ 
belongs to $\Mregg(\hnu_{\mathbf{m}}, \hlam_{\beta}).$
On the other hand, the element $\mathbf{x}_{\mathbf{m}}$,
which is regarded as a representation of 
$\widetilde{\Lambda}_{Q}$,
restricts to give a representation of $\C Q$ isomorphic to
$\bigoplus_{\alpha \in \mathsf{R}^{+}} 
M(\alpha)^{\oplus m_{\alpha}}$.
This means that $\Psi_{\beta}(x_{\mathbf{m}}) 
\in \mathbb{O}_{\mathbf{m}}.$
Therefore we have
$\Psi_{\beta}(\Mregg(\hnu_{\mathbf{m}}, \hlam_{\beta}))
\supset \mathbb{O}_{\mathbf{m}}$
as desired.
\end{proof}

We define a partial order $\le$ on the set $\KP(\beta)$
of Kostant partitions of $\beta$ 
by the condition that for $\mathbf{m}, \mathbf{m}^{\prime} \in \KP(\beta)$,
we have $\mathbf{m} \le \mathbf{m}^{\prime}$
if and only if $\overline{\mathbb{O}}_{\mathbf{m}} \supset
\mathbb{O}_{\mathbf{m}^{\prime}}$.
From the Lemma~\ref{Lem:strata} above,
we conclude that
the bijection $f \colon \KP(\beta) \to \lP^{+}_{Q, \beta}$
preserves the partial orders. 

Using the isomorphism $\Psi_\beta$ and the representation theory of $Q$,
we can obtain more precise information on the stabilizer group of a point of
$\M_{0}^{\bullet}(\hlam_{\beta})$ as follows.

\begin{Lem}
\label{Lem:red}
Let $\hnu \in \lQ^{+}_{Q}$ such that 
$\hlam_{\beta} - \hnu \in \lP^{+}_{Q, \beta}$.
Fix an arbitrary point $x_{\hnu} \in \Mregg(\hnu, \hlam_{\beta})$.
Then the maximal reductive quotient
($=$ the quotient by the unipotent radical)
of the group $\Stab_{\G(\hlam_{\beta})} x_{\hnu}$
is isomorphic to $\G(\hlam_{\beta}-\hnu)$.
Moreover the group homomorphism 
$(\hat{\varphi} \times \hat{\rho})
\circ \hat{r}^{-1}
\colon \Stab_{\G(\hlam_{\beta})} x_{\hnu} 
\to \G(\hlam_{\beta}-\hnu)$ 
defined in (\ref{Eq:Stabhom2})
is identical to the canonical quotient map.
\end{Lem}
\begin{proof}
Define $\mathbf{m} = (m_{\alpha})_{\alpha \in \mathsf{R}^{+}}
\in \KP(\beta)$
by $f(\mathbf{m}) = \hlam_{\beta} - \hnu$.
By Lemma~\ref{Lem:strata},
the point $\Psi_{\beta}(x_{\hnu})$ corresponds to a 
$\C Q$-module
$ M(\mathbf{m}) \cong \bigoplus_{\alpha \in \mathsf{R}^{+}} 
M(\alpha)^{\oplus m_{\alpha}}$.
Then we have 
$\Stab_{G(\hlam_{\beta})} x_{\hnu}
= \Stab_{G_{\beta}} \Psi_{\beta}(x_{\hnu})
= \End_{\C Q}( M(\mathbf{m}))^{\times}.$
We consider a subgroup 
$$
G_{1} := 
\prod_{\alpha \in \mathsf{R}^{+}}
\End_{\C Q}(M(\alpha)^{\oplus m_{\alpha}} )^{\times}
\subset
\End_{\C Q}( M(\mathbf{m}))^{\times}. 
$$
Note that 
we have $\End_{\C Q}(M(\alpha)) = \C$ for any root $\alpha \in 
\mathsf{R}^{+}$.
We can easily see that  
this subgroup $G_{1}$ is 
a Levi subgroup of $\End_{\C Q}( M(\mathbf{m}))^{\times}$
and therefore is
isomorphic to
the maximal reductive quotient of 
$\Stab_{G(\hlam_{\beta})} x_{\hnu}$
by the canonical quotient map. 
This shows that 
the maximal reductive quotient of 
$\Stab_{\G(\hlam_{\beta})} x_{\hnu} = \Stab_{G(\hlam_{\beta})}
x_{\hnu}
\times \T$ is isomorphic to 
$G(\hlam_{\beta} - \hnu) \times \C^{\times} 
= \G(\hlam_{\beta} - \hnu).$  

Let us prove the latter assertion.
Corresponding to the decomposition 
$M(\mathbf{m}) \cong \bigoplus_{\alpha \in \mathsf{R}^{+}} 
M(\alpha)^{\oplus m_{\alpha}}$,
we choose an element 
$\mathbf{x}_{\hnu} = \bigoplus_{\alpha \in \mathsf{R}^{+}}
\mathbf{y}_{\alpha}^{\oplus m_{\alpha}}$
as a lift of the point $x_{\hnu}$,
where $\mathbf{y}_{\alpha}$\rq{}s are the same as in 
the proof of Lemma~\ref{Lem:strata} above.
Then we have 
$\Stab_{G(\hnu) \times G(\hlam_{\beta})} \mathbf{x}_{\hnu}
= \End_{\widetilde{\Lambda}_{Q}}
(\mathbf{x}_{\hnu})^{\times}.$
We consider a subgroup
$$
\widetilde{G}_{1} := 
\prod_{\alpha \in \mathsf{R}^{+}} 
GL_{m_{\alpha}}(\C \cdot \mathrm{id}_{\mathbf{y}_{\alpha}}) 
\subset 
\prod_{\alpha \in \mathsf{R}^{+}}
\End_{\widetilde{\Lambda}_{Q}}
(\mathbf{y}_{\alpha}^{\oplus m_{\alpha}})^{\times} 
\subset
\End_{\widetilde{\Lambda}_{Q}}( \mathbf{x}_{\hnu})^{\times}. 
$$
Note that the homomorphism $\hat{r}$
gives an isomorphism $\widetilde{G}_{1} \xrightarrow{\cong}
G_{1}$.
On the other hand, we can easily see that 
the homomorphism
$\hat{\varphi} \colon \Stab_{G(\hnu) \times 
\G(\hlam_{\beta})}\mathbf{\hnu} \to G(\hlam_{\beta} - \hnu)$
induces the isomorphism
$$
\widetilde{G}_{1} \xrightarrow{\cong}
\prod_{\alpha \in \mathsf{R}^{+}}
GL(H^{0}(C_{\phi(\alpha)}(\hnu, \hlam_{\beta}
)_{\mathbf{x}_{\hnu}})) \cong
G(\hlam_{\beta} - \hnu).
$$
As a result, we have a following commutative diagram:
\begin{equation}
\label{Eq:section}
\xy
\xymatrix{
\Stab_{\G(\hlam_{\beta})} x_{\hnu}
&
\Stab_{G(\hnu) \times \G(\hlam_{\beta})} \mathbf{x}_{\hnu}
\ar[l]_-{\hat{r}}^-{\cong}
\ar[r]^-{\hat{\varphi} \times \hat{\rho}}
&
\G(\hlam_{\beta} - \hnu)
\\
G_{1} \times \T
\ar@{^{(}->}[u]
&
\widetilde{G}_{1} \times \widetilde{\T}
\ar[l]_-{\cong}
\ar[r]^{\cong}
\ar@{^{(}->}[u]
&
\G(\hlam_{\beta} - \hnu)
\ar@{=}[u]
}
\endxy
\end{equation} 
where the torus $\widetilde{\T}$ is defined as 
in the proof of Lemma~\ref{Lem:fiberisom2}.
This diagram completes a proof.
\end{proof}

For a linear algebraic group $G$,
we denote its representation ring by $R(G)$.
For $G = \C^{\times}$, we always identify
$R(\C^{\times}) = \mathbb{A} := \Z[q^{\pm 1}]$ so that $q$ corresponds to the natural $1$-dimensional representation of $\C^\times$.
With this notation, we have the standard identifications
$$
R(\G(\hlam_{\beta})) = \bigotimes_{i \in I} 
\left( \mathbb{A}[z_{i}^{\pm 1}]^{\otimes d_i}
\right)^{\SG_{d_{i}}}, \quad
R(\G(\hmu)) = \bigotimes_{\alpha \in \mathsf{R}^{+}} 
\left( \mathbb{A}[z_{\alpha}^{\pm 1}]^{\otimes m_{\alpha}}
\right)^{\SG_{m_{\alpha}}},
$$
where $\hmu = 
\sum_{\alpha \in \mathsf{R}^{+}} 
m_{\alpha} \varpi_{\phi(\alpha)} \in \lP^{+}_{Q, \beta}$.
Here $\otimes$'s are taken over $\mathbb{A}$. 

For a positive root $\alpha = \sum_{i \in I}
c_{i} \alpha_{i}
\in \mathsf{R}^{+},
$
we define the following algebra homomorphism:
$$
\theta_{\alpha} \colon 
\bigotimes_{i \in I}
\mathbb{A}[z_{i}^{\pm1}]^{\otimes {c_{i}}} 
\to
\mathbb{A}[z_{\alpha}^{\pm1}],
\quad
z_{i} \mapsto 
q^{\pt(\alpha_{i}) - \pt(\alpha)}
z_{\alpha}.
$$ 

Now we return to the setting of Lemma~\ref{Lem:red}.
By the commutative diagram (\ref{Eq:section})
in the proof of Lemma~\ref{Lem:red},
we have the following group embedding:
$$ 
\G(\hlam_{\beta} - \hnu) \cong G_{1} \times \T 
\hookrightarrow \Stab_{\G(\hlam_{\beta})} x_{\hnu} 
\hookrightarrow
\G(\hlam_{\beta}),
$$
which induces the following homomorphism:
\begin{equation}
\label{Eq:rest}
\theta_{\hlam_{\beta} - \hnu}
\colon R(\G(\hlam_{\beta})) \to
R(\Stab_{\G(\hlam_{\beta})} x_{\hnu}) 
\xrightarrow{\cong}
R(\G(\hlam_{\beta} - \hnu)).
\end{equation}
From the proof of Lemma~\ref{Lem:red},
we have the following.
\begin{Cor}
\label{Cor:rest}
Write $\hlam_{\beta} - \hnu
= \sum_{\alpha \in \mathsf{R}^{+}}
m_{\alpha} \varpi_{\phi(\alpha)}
\in \lP^{+}_{Q, \beta}.$
Then the above homomorphism 
$\theta_{\hlam_{\beta} - \hnu} \colon
R(\G(\hlam_{\beta}))
\to R(\G(\hlam_{\beta}- \hnu))
$
coincides with the restriction 
of the homomorphism
$$
\bigotimes_{\alpha} 
\theta_{\alpha}^{\otimes m_{\alpha}} \colon
\bigotimes_{i \in I} 
\mathbb{A}[z_{i}^{\pm 1}]^{\otimes d_{i}}
\to
\bigotimes_{\alpha \in \mathsf{R}^{+}} 
\mathbb{A}[z_{\alpha}^{\pm 1}]^{\otimes m_{\alpha}}.
$$
\end{Cor}


\section{Central completion of convolution algebra}
\label{Sec:completion}

In this section, we study the structure of 
the Hernandez-Leclerc category $\Cc_{Q, \beta}$
using the geometry of graded quiver varieties.


\subsection{Notation for equivariant $K$-theory}
\label{Ssec:notationK}

For a quasi-projective variety $X$ over $\C$
equipped with an action of a linear algebraic group $G$,
we denote by $K^{G}(X)$ 
the Grothendieck group of the abelian
category of $G$-equivariant coherent sheaves on $X$.
For a $G$-equivariant coherent sheaf $\mathcal{F}$ on $X$, 
we denote by $[\mathcal{F}]$
the corresponding element in 
$K^{G}(X)$.
The structure sheaf of $X$ is denoted by $\mathcal{O}_{X}$.
For a $G$-equivariant vector bundle $\mathcal{E}$ on $X$,
the map
$K^{G}(X) \ni [\mathcal{F}] \mapsto $
$[\mathcal{E}] \cdot
[\mathcal{F}] := 
[\mathcal{E} \otimes_{\mathcal{O}_{X}}
\mathcal{F}] \in K^{G}(X)$
is well-defined.
For a $G$-equivariant vector bundle $\mathcal{E}$ on $X$,
we also define
$
\bigwedge_{u} [\mathcal{E}] := 
\sum_{i=0}^{\mathrm{rank} \, \mathcal{E}}
u^{i} [\bigwedge^{i}\mathcal{E}]
\in [\mathcal{O}_{X}] + u K^{G}(X)[u].
$
If $0 \to \mathcal{E}_{1} \to \mathcal{E} \to \mathcal{E}_{2} \to  0$
is an exact sequence of $G$-equivariant 
vector bundles on $X$, we have
$\bigwedge_{u} [\mathcal{E}] = 
\bigwedge_{u} [\mathcal{E}_{1}] \cdot
\bigwedge_{u} [\mathcal{E}_{2}]$.
Therefore we set 
$\bigwedge_{u}([\mathcal{E}_{1}] + 
[\mathcal{E}_{2}]) := 
\bigwedge_{u}[\mathcal{E}_{1}]
\cdot \bigwedge_{u}[\mathcal{E}_{2}]
$  
and 
$\bigwedge_{u}(-[\mathcal{E}]) := 
(\bigwedge_{u} [\mathcal{E}])^{-1} \in 
[\mathcal{O}_{X}] + u K^{G}(X)[\![u]\!].
$

When $X = \mathrm{pt}$, we have
$K^{G}(\mathrm{pt}) = R(G)$.
For a general $X$,
the group $K^{G}(X)$ is a module over $R(G)$.
When $G$ is written in the form 
$G = G_{0} \times \C^{\times}$ 
where $G_{0}$ is another linear algebraic group,
we regard $K^{G}(X)$ as a $\Z[q^{\pm 1}]$-module 
through the standard identification
$\Z[q^{\pm 1}] \cong R(\C^\times)$
and the algebra inclusion $R(\C^\times) \hookrightarrow R(G)$
induced from the natural projection $G=G_0 \times \C^\times \to \C^\times$.
Then we write
$\K^{G}(X) := K^{G}(X)\otimes_{\Z[q^{\pm 1}]} \kk$
and $\mathcal{R}(G) := R(G) \otimes_{\Z[q^{\pm 1}]} \kk$,
where $\kk = \overline{\Q(q)}$
as before.
For each $m \in \Z$, let $L_{m}$ denote
the $1$-dimensional $\C^{\times}$-module
of weight $m \in \Z$.
Namely $[L_{m}] = q^{m}$
in $R(\C^{\times})$. 
Although it is an abuse of notation,
we write $q^{m}\mathcal{F} := L_{m} \otimes_{\mathcal{O}_{X}}
\mathcal{F} $ for any 
$G$-equivariant coherent sheaf $\mathcal{F}$
on $X$ so that we have $q^{m}[\mathcal{F}] 
= [q^{m} \mathcal{F}]$ in $\K^{G}(X)$.  

We also use the equivariant 
topological $K$-homologies denoted by 
$K_{i, \mathrm{top}}^{G}(X)$
($i=0,1$).
There is a canonical comparison map
$K^{G}(X) \to K^{G}_{0, \mathrm{top}}(X)$
(see~\cite[Section 5.5.5]{CG97}).

Let $Y$ be a $G$-invariant closed 
subvariety of $X$ and
$U = X \setminus Y$ be the complement of $Y$.
Then the inclusions
$
Y \xrightarrow{i} X \xleftarrow{j} U 
$ 
induce the followings:
\begin{enumerate}
\item
an exact sequence:
\begin{equation}
\label{Eq:MVseq}
\xy
\xymatrix{
K^{G}(Y) \ar[r]^{i_{*}} & K^{G}(X)
\ar[r]^{j^{*}} 
&K^{G}(U) \ar[r] & 0, 
}
\endxy
\end{equation}
\item
an exact hexagon:
\begin{equation}
\label{Eq:hexagon}
\xy
\xymatrix{
K^{G}_{0, \mathrm{top}}(Y) \ar[r]^{i_{*}}
& K^{G}_{0, \mathrm{top}}(X) \ar[r]^{j^{*}}
& K^{G}_{0, \mathrm{top}}(U) \ar[d] \\
K^{G}_{1, \mathrm{top}}(U) \ar[u]
& K^{G}_{1, \mathrm{top}}(X) \ar[l]_{j^{*}} 
& K^{G}_{1, \mathrm{top}}(Y) \ar[l]_{i_{*}}.}
\endxy
\end{equation}
\end{enumerate}


\subsection{The Nakajima homomorphism and its completion}
\label{Ssec:Nakajima}

Fix a dominant weight $\lambda \in \cP^{+}$ and
consider the corresponding quiver variety
$\pi \colon \M(\lambda) \to \M_{0}(\lambda).$
We define the {\em Steinberg type variety} $Z(\lambda)$ as
$$
Z(\lambda) := 
\M(\lambda) 
\times_{\M_{0}(\lambda)} \M(\lambda)
= 
\bigsqcup_{\nu_{1}, \nu_{2} \in \cQ^{+}}
\M(\nu_{1},\lambda)
\times_{\M_{0}(\lambda)}
\M(\nu_{2}, \lambda),
$$
together with the canonical map 
$\pi \colon Z(\lambda) \to \M_{0}(\lambda).$
By the convolution product,
the equivariant $K$-group 
$$
\K^{\G(\lambda)}(Z(\lambda))
= \bigoplus_{\nu_{1}, \nu_{2} \in \mathsf{Q^{+}}}
\K^{\G(\lambda)}
(\M(\nu_{1},\lambda)
\times_{\M_{0}(\lambda)}
\M(\nu_{2}, \lambda)
)
$$
becomes an algebra
over the commutative $\kk$-algebra
$ \R(\lambda) := \R(\G(\lambda)).$
Note that this notation is consistent
with (\ref{Eq:Rlambda}).

We consider the following 
tautological vector bundles on $\M(\nu, \lambda)$. 
The vector bundle $\V^{\nu}_{i}$
is defined by $\V^{\nu}_{i}
:= \mu^{-1}(0)^{\mathrm{st}} \times_{G(\nu)} V^{\nu}_{i}$
for each $i \in I$. 
We regard $\V^{\nu}_{i}$ as a 
$\G(\lambda)$-equivariant vector bundle
with the trivial action.
On the other hand, we consider 
the trivial vector bundle
$\W^{\lambda}_{i}
:= \M(\nu, \lambda) \times W^{\lambda}_{i}$
with fiber $W^{\lambda}_{i}$
for each $i \in I$. 
We regard $\W^{\lambda}_{i}$ as
a $\G(\lambda)$-equivariant vector bundle
with the natural $G(\lambda)$-action    
and the trivial $\C^{\times}$-action.
Recall the complex of vector spaces
$C_{i}(\nu, \lambda)_{\mathbf{x}}$ 
for each $\mathbf{x} \in \mu^{-1}(0) 
\subset \mathbf{M}(V^{\nu}, W^{\lambda})
$
defined in (\ref{Eq:three_pt}).
This complex yields the complex 
$\mathcal{C}_{i}(\nu, \lambda)$ 
of 
$\G(\lambda)$-equivariant vector bundles on $\M(\nu, \lambda)$:
\begin{equation}
\nonumber
\mathcal{C}_{i}(\nu, \lambda) \colon \;
q^{-2} \V^{\nu}_{i} \xrightarrow{\sigma_{i}}
q^{-1} \left(
\W^{\lambda}_{i} \oplus 
\bigoplus_{j \sim i} \V^{\nu}_{j} \right)
\xrightarrow{\tau_{i}}
\V^{\nu}_{i}.
\end{equation}          
Note that the class 
of the complex $\mathcal{C}_{i}(\nu, \lambda)$
in $\K^{\G(\lambda)}(\M(\nu, \lambda))$
is calculated as
$$  
[\mathcal{C}_{i}(\nu, \lambda)]
= q^{-1}
\left( [\W^{\lambda}_{i}] 
-(q+q^{-1})[\V^{\nu}_{i}]
+\sum_{j \sim i}[\V^{\nu}_{j}]
\right).
$$

Then we have the following fundamental result 
due to Nakajima~\cite{Nakajima01}.

\begin{Thm}[Nakajima~{\cite[Theorem 9.4.1]{Nakajima01}}] 
\label{Thm:Nakajima}
There is a $\kk$-algebra homomorphism
$$
\Phi_{\lambda} \colon
\tilU(L \g) \to 
\K^{\G(\lambda)}(Z(\lambda)),
$$
such that 
$$
\Phi_{\lambda} (a_{\mu}) = 
\begin{cases}
\Delta_{*}[\mathcal{O}_{\M(\nu, \lambda)}] &
\text{if $\nu := \lambda - \mu \in \cQ^{+}$;} \\
0 
& \text{otherwise,}
\end{cases}
$$
and it sends 
the series
$\psi_{i}^{\pm}(z)a_{\mu}$
with $\nu := \lambda - \mu \in \cQ^{+}, \, i \in I$
to the series
$$
q^{\mu(h_{i})}
\Delta_{*} \left(
\frac{\bigwedge \nolimits_{-1/qz}[\mathcal{C}_{i}
(\nu, \lambda)]}
{\bigwedge \nolimits_{-q/z}[\mathcal{C}_{i}(\nu, \lambda)]}
\right)^{\pm},
$$
where
$\Delta \colon \M(\nu, \lambda)
\to \M(\nu, \lambda)
\times_{\M_{0}(\lambda)} 
\M(\nu, \lambda)
$ 
is the diagonal embedding
and 
$(-)^{\pm}$ denotes the formal expansion at 
$z= \infty$ and $0$ respectively. 
\end{Thm}               

We refer to the homomorphism $\Phi_{\lambda}$
as the {\em Nakajima homomorphism}.

By construction,
the equivariant $K$-group
$\K^{\G(\lambda)}(\LL(\lambda))$
of the central fiber 
$\LL(\lambda)$
becomes a module over
the convolution algebra
$\K^{\G(\lambda)}(Z(\lambda)).$
Via the Nakajima homomorphism $\Phi_{\lambda}$,
we regard $\K^{\G(\lambda)}(\LL(\lambda))$
as a $(U_{q}, \R(\lambda))$-bimodule.

\begin{Thm}[Nakajima]
\label{Thm:Kcentral}
As a $(U_{q}, \R(\lambda))$-bimodule,
the module $\K^{\G(\lambda)}(\LL(\lambda))$
is isomorphic to
the global Weyl module $\bW(\lambda)$.
The element $[\mathcal{O}_{\LL(0, \lambda)}]
\in \K^{\G(\lambda)}(\LL(\lambda))$
corresponds to the cyclic vector
$w_{\lambda} \in \bW(\lambda).$
(Recall $\LL(0, \lambda) = \M(0, \lambda) = \mathrm{pt}$ from Subsection~\ref{Ssec:usual}.)
\end{Thm}
\begin{proof}
See \cite[Theorem 2]{Nakajima04}.
\end{proof}

For future references, 
we collect some important properties 
of the equivariant $K$-groups of
central fibers.  
 
\begin{Thm}[Nakajima]   
\label{Thm:Kfibers}
Let $G^{\prime}$ a closed subgroup 
of $\G(\lambda)$.
The followings
hold.
\begin{enumerate}
\item \label{Thm:Kfibers:K1}
We have 
$K^{G^{\prime}}_{1, \mathrm{top}}(\LL(\lambda)) = 0;$
\item \label{Thm:Kfibers:K0} 
$K_{0, \mathrm{top}}^{G^{\prime}}(\LL(\lambda))$ 
is a free $R(G^{\prime})$-module and
the comparison map
$K^{G^{\prime}}(\LL(\lambda)) \to 
K^{G^{\prime}}_{0, \mathrm{top}}(\LL(\lambda))$ 
is an isomorphism;
\item \label{Thm:Kfibers:subgroup}
The natural map 
$K^{\G(\lambda)}(\LL(\lambda)) \otimes_{R(\G(\lambda))}
R(G^{\prime}) \to K^{G^{\prime}}(\LL(\lambda))$
is an isomorphism;
\item \label{Thm:Kfibers:Kunneth}
The K\"{u}nneth homomorphisms
\begin{align*}
K^{\G(\lambda)}(\LL(\lambda)) 
\otimes_{R(\G(\lambda))}
K^{\G(\lambda)}(\LL(\lambda)) 
&\to 
K^{\G(\lambda)}(\LL(\lambda) \times \LL(\lambda)),\\
K_{0, \mathrm{top}}^{\G(\lambda)}(\LL(\lambda)) 
\otimes_{R(\G(\lambda))}
K_{i, \mathrm{top}}^{\G(\lambda)}(\LL(\lambda)) 
&\to 
K_{i, \mathrm{top}}^{\G(\lambda)}(\LL(\lambda) 
\times \LL(\lambda))
\end{align*}
are isomorphisms, 
where $i = 0, 1$.
\end{enumerate}
\end{Thm}
\begin{proof}
The properties 
(\ref{Thm:Kfibers:K1}),
(\ref{Thm:Kfibers:K0}),
(\ref{Thm:Kfibers:subgroup})
are the same as the property ($T_{\G(\lambda)}$) in
\cite[Section 7]{Nakajima01}. 
The assertion for $K^{\G(\lambda)}$ in (\ref{Thm:Kfibers:Kunneth})
follows from \cite[Theorem 3.4]{Nakajima01t}.
The assertion for $K^{\G(\lambda)}_{i, \mathrm{top}}$
in (\ref{Thm:Kfibers:Kunneth})
follows from 
the properties (\ref{Thm:Kfibers:K1}), (\ref{Thm:Kfibers:K0})
and the property (n3) in \cite[Section 1.2]{KL87}. 
\end{proof}
Next we fix an $\ell$-dominant $\ell$-weight
$\hlam \in \lP^{+}_{\Z}$
and put $\lambda := \mathsf{cl}(\hlam).$
We consider the corresponding
graded quiver varieties
$\pi^{\bullet} \colon \M^{\bullet}(\hlam)
\to \M_{0}^{\bullet}(\hlam)$ 
and the corresponding Steinberg type variety
$$
Z^{\bullet}(\hlam)
:= \M^{\bullet}(\hlam)
\times_{\M_{0}^{\bullet}(\hlam)}
\M^{\bullet}(\hlam).
$$
The $\G(\hlam)$-equivariant
$K$-group
$\K^{\G(\hlam)}(Z^{\bullet}(\hlam))$
is an algebra over
$\R(\G(\hlam))$
by the convolution product.
We set $\R(\hlam):= \R(\G(\hlam))$,
which is consistent with (\ref{Eq:Rhlam}).

We choose the $1$-dimensional subtorus $\T
\subset \G(\hlam) \subset \G(\lambda)$
as in Lemma~\ref{Lem:fixed}.
Then we have the identification
\begin{equation}
\label{Eq:fixedSteinberg}
Z^{\bullet}(\hlam) \cong Z(\lambda)^{\T}.
\end{equation}
Let $\rr_{\hlam}$ denote the maximal ideal
of the commutative ring
$\R(\hlam)$
corresponding to the subtorus $\T \subset \G(\hlam)$ and 
form the completion 
$\hR(\hlam) := \displaystyle
\varprojlim
\R(\hlam)
/ \rr_{\hlam}^{N}.
$
Note that the notation is consistent with (\ref{Eq:etale}). 
For any $\G(\hlam)$-variety $X$,
we write the corresponding completions of $K$-groups by
$$
\hK^{\G(\hlam)}(X)
: = \K^{\G(\hlam)}
(X) 
\otimes_{\R(\hlam)}
\hR(\hlam),
\quad
\hK_{i, \mathrm{top}}^{\G(\hlam)}(X)
: = \K_{i, \mathrm{top}}^{\G(\hlam)}
(X) 
\otimes_{\R(\hlam)}
\hR(\hlam).
$$
When $X$ is the Steinberg type variety 
$Z^{\bullet}(\hlam)$,
the completed $K$-group
$
\hK^{\G(\hlam)}
(Z^{\bullet}(\hlam))
$
is an algebra over $\hR(\hlam)$
with respect to the convolution product.
Thanks to the localization theorem, 
this algebra can be obtained from the convolution algebra $\K^{\G(\hlam)}(Z(\lambda))$
by the central completion $- \otimes_{\R(\hlam)} \hR(\hlam)$. 

\begin{Def}
\label{Def:completedNakajima}
We define the {\em completed Nakajima homomorphism}
$\widehat{\Phi}_{\hlam}
\colon \tilU \to \hK^{\G(\hlam)}
(Z^{\bullet}(\hlam))$
as the following composition:
$$
\tilU
\xrightarrow{\Phi_{\lambda}}
\K^{\G(\lambda)}(Z(\lambda)) 
\to
\K^{\G(\hlam)}(Z(\lambda)) 
\to
\hK^{\G(\hlam)}(Z(\lambda))
\xrightarrow{\cong} 
\hK^{\G(\hlam)}(Z^{\bullet}(\hlam)), 
$$
where the first homomorphism is the Nakajima
homomorphism $\Phi_{\lambda}$ in Theorem 
\ref{Thm:Nakajima}, the second 
is the restriction to 
the subgroup 
$\G(\hlam) \subset \G(\lambda)$,
the third is canonical and
the last is due to the localization theorem
and (\ref{Eq:fixedSteinberg}). 
\end{Def}      

Let $\hnu \in \lQ^{+}_{\Z}$ be an element such that
$\Mregg(\hnu, \hlam) \neq \varnothing$.
We pick an arbitrary point 
$x_{\hnu} \in \Mregg(\hnu, \hlam)$ and
consider 
the (non-equivariant) $K$-group 
$\K(\M^{\bullet}(\hlam)_{x_{\hnu}})
:= K(\M^{\bullet}(\hlam)_{x_{\hnu}}) \otimes_{\Z}
\kk
$
of the fiber $\M^{\bullet}(\hlam)_{x_{\hnu}}$.
This is a module over the convolution algebra
$\K(Z^{\bullet}(\hlam)) := K(Z^{\bullet}(\hlam)) \otimes_{\Z} \kk$.
We regard $\K(\M^{\bullet}(\hlam)_{x_{\hnu}})$
as a $\tilU$-module
via the following composition:
\begin{align*}
\tilU
\xrightarrow{\widehat{\Phi}_{\hlam}}
\hK^{\G(\hlam)}(Z^{\bullet}(\hlam))
\to
\hK^{\G(\hlam)}(Z^{\bullet}(\hlam))/\rr_{\hlam}
\xrightarrow{\cong}
\K^{\T}(Z^{\bullet}(\hlam))
\cong \K(Z^{\bullet}(\hlam)),
\end{align*}
where the third arrow is an isomorphism by Theorem~\ref{Thm:Kfibers}~(\ref{Thm:Kfibers:subgroup}).

\begin{Prop}
\label{Prop:KlocalWeyl}
The $\tilU$-module
$\K(\M^{\bullet}(\hlam)_{x_{\hnu}})$
is isomorphic to 
the local Weyl module $W(\hlam - \hnu)$.
\end{Prop}  
\begin{proof}
When $\hnu = 0$,
we have
\begin{align*}
\K(\LL^{\bullet}(\hlam)) &\cong
K^{\T}(\LL(\lambda)) \otimes_{R(\T)} \kk
&& \text{by the localization theorem}
\\
&\cong
\K^{\G(\lambda)}(\LL(\lambda))/\rr_{\lambda, \hlam}
&&
\text{by Theorem 
\ref{Thm:Kfibers} (\ref{Thm:Kfibers:subgroup})}
\\
&\cong \bW(\lambda)/\rr_{\lambda, \hlam}
&& \text{by Theorem~\ref{Thm:Kcentral}}\\
&= W(\hlam).
\end{align*}
For a general $\hnu$,
we know that the $\tilU$-module 
$\K(\M^{\bullet}(\hlam)_{x_{\hnu}})$
is a quotient of $W(\hlam - \hnu)$
by \cite[Proposition 13.3.1]{Nakajima01}
and by the universality of the local Weyl module.
Because there is an isomorphism 
$\M^{\bullet}(\hlam)_{x_{\hnu}} \cong \LL^{\bullet}(\hlam - \hnu)$
by Lemma~\ref{Lem:fiberisom2},
we have
$\dim \K(\M^{\bullet}(\hlam)_{x_{\hnu}})
= \dim \K(\LL^{\bullet}(\hlam - \hnu))
= \dim W(\hlam-\hnu)$
and hence the isomorphism
$W(\hlam - \hnu) \xrightarrow{\cong}
\K(\M^{\bullet}(\hlam)_{x_{\hnu}})$.
\end{proof}
 

\subsection{Completed Hernandez-Leclerc category 
$\hCc_{Q, \beta}$}
\label{Ssec:main}

Throughout this subsection,
we fix an element 
$\beta = \sum_{i \in I}d_{i}\alpha_{i} \in \cQ^{+}$
and set
$$\hlam \equiv
\hlam_{\beta}
:= \sum_{i \in I}d_{i}\varpi_{\phi(\alpha_{i})} \in \lP_{Q}^{+}, \quad
\lambda := \mathsf{cl}(\hlam) \in \cP^{+},$$
as in Section \ref{Ssec:identify}. 
For simplicity, we mainly use the notation $\hlam$
rather than $\hlam_{\beta}$, 
suppressing $\beta$. 

Fix an element $\hnu \in \lQ^{+}_{Q}$ 
such that $\hmu := \hlam - \hnu \in \lP^{+}_{Q, \beta}$
and put
$\nu := \mathsf{cl}(\hnu) \in \cQ^{+}$,
$\mu := \mathsf{cl}(\hmu) \in \cP^{+}$.
Recall that we have an algebra homomorphism
$\theta_{\hmu} \colon 
R(\G(\hlam)) \to R(\G(\hmu))$
defined in Subsection \ref{Ssec:identify} (\ref{Eq:rest}).
After the localization $(-) \otimes_{\Z[q^{\pm 1}]} \kk$,
we get a homomorphism 
$$
\theta_{\hmu} \colon
\R(\hlam) \to \R(\hmu), 
$$
for which we use the same symbol $\theta_{\hmu}$.
Through this homomorphism 
$\theta_{\hmu}$, the algebra $\R(\hmu)$
is regarded as an $\R(\hlam)$-algebra.
\begin{Lem}
\label{Lem:rest}
The ideal $\langle \theta_{\hmu}(\rr_{\hlam})
\rangle \subset \R(\hmu)$
generated by the image of $\rr_{\hlam}$
is a primary ideal whose associated prime is the maximal ideal
$\rr_{\hmu}.$
In particular, we have
$$ \R(\hmu) \otimes_{\R(\hlam)} 
\hR(\hlam) \cong \hR(\hmu).$$
\end{Lem} 
\begin{proof}
This is a direct consequence of Corollary \ref{Cor:rest}.
\end{proof}
 
\begin{Prop}
\label{Prop:KMmu}
With the above notation, let $\M_{\hmu}^{\bullet}$ denote the inverse image of the stratum
$\Mregg(\hnu, \hlam)$ 
along the canonical morphism
$\pi^{\bullet} \colon \M^{\bullet}(\hlam) \to 
\M_{0}^{\bullet}(\hlam)$. 
Then we have the following isomorphism
of $(\tilU, \hR(\hlam))$-bimodules:
$$
\hK^{\G(\hlam)}(\M_{\hmu}^\bullet) \cong 
\hW(\hmu),
$$  
where the action of $\hR(\hlam)$ on the RHS is given via 
the homomorphism $\theta_{\hmu}$.
\end{Prop} 
\begin{proof}
Fix a point $x \in \Mregg(\hnu, \hlam)$.
Since $\Mregg(\hnu, \hlam)$ consists of a single
$\G(\hlam)$-orbit by Lemma~\ref{Lem:strata}
and the morphism
$\pi^{\bullet} \colon \M^{\bullet}(\hlam) \to 
\M_{0}^{\bullet}(\hlam)$
is $\G(\hlam)$-equivariant,
we have an isomorphism
$
\M_{\mu}^{\bullet} \cong \G(\hlam) 
\times_{(\Stab_{\G(\hlam)} x)} 
\M^{\bullet}(\hlam)_{x}.  
$
Then we have
\begin{align*}
\hK^{\G(\hlam)}(\M^{\bullet}_{\hmu}) &\cong 
\hK^{\G(\hlam)}\left(\G(\hlam) 
\times_{(\Stab_{\G(\hlam)} x)} 
\M^{\bullet}(\hlam)_{x}\right) \\
& \cong
\K^{(\Stab_{\G(\hlam)}x )}(\M^{\bullet}(\hlam)_{x})
\otimes_{\R(\hlam)}
\hR(\hlam)\\
&\cong
\K^{\G(\hmu)}(\LL^{\bullet}(\hmu))\otimes_{\R(\hmu)}\hR(\hmu)\\
&\cong
\K^{\G(\mu)}(\LL(\mu))\otimes_{\R(\mu)}
\hR(\hmu),
\end{align*}
where
the second isomorphism is by 
the induction (see \cite[5.2.16]{CG97}),
the third is due to 
Lemma~\ref{Lem:fiberisom2},
Lemma~\ref{Lem:red} and
Lemma~\ref{Lem:rest},
the last is due to the localization and
Theorem~\ref{Thm:Kfibers} (\ref{Thm:Kfibers:subgroup}).
Through this isomorphism 
$\hK^{\G(\hlam)}(\M^{\bullet}_{\hmu}) \cong
\K^{\G(\mu)}(\LL(\mu))\otimes_{\R(\mu)}
\hR(\hmu)$, we see that
the action of $\hR(\hlam)$ on 
$\hK^{\G(\hlam)}(\M^{\bullet}_{\hmu})$
extends to an action of $\hR(\hmu)$, 
which commutes with the action of $U_{q}$.
By Proposition~\ref{Prop:KlocalWeyl}, the module
$\hK^{\G(\hlam)}(\M^{\bullet}_{\hmu})/\rr_{\hmu}
\cong \K(\M(\lambda)_{x})
$
is isomorphic to the local Weyl module $W(\hmu)$.
Therefore, by Nakayama\rq{}s lemma, we see that
the vector in $\hK^{\G(\hlam)}(\M^{\bullet}_{\hmu})$
which corresponds to $[\mathcal{O}_{\LL(0, \mu)}]$
in $\K^{\G(\mu)}(\LL(\mu))\otimes_{\R(\mu)}
\hR(\hmu)$
generates $\hK^{\G(\hlam)}(\M^{\bullet}_{\hmu})$
as $(U_{q}, \hR(\hmu))$-bimodule.
Moreover, 
from the construction of isomorphism
$\M(\lambda)_{x} \cong \LL(\mu)$
in Lemma~\ref{Lem:fiberisom},
we see that 
the restriction of
the class $[\mathcal{C}_{i}(\lambda, \nu)]$
in $\hK^{\G(\hlam)}(\M^{\bullet}_{\hmu})$
corresponds to the class
$[\W_{i}^{\lambda - \nu}]|_{\LL(\mu)}$
in $\K^{\G(\mu)}(\LL(\mu))\otimes_{\R(\mu)}
\hR(\hmu)$.
Therefore, by the universal property
of the global Weyl module,
we find a surjection 
$\hW(\hmu) \to \hK^{\G(\hlam)}(\M^{\bullet}_{\hmu})$
of $(U_{q}, \hR(\hmu))$-bimodules.
Since both are free over $\hR(\hmu)$
of the same rank $\dim W(\hmu)$,
this surjection should be an isomorphism
$\hW(\hmu) \cong \hK^{\G(\hlam)}(\M^{\bullet}_{\hmu}).$
\end{proof}

\begin{Def}
We define the {\em completed Hernandez-Leclerc category}
$\hCc_{Q, \beta}$ to be the category 
of finitely generated $\hK^{\G(\hlam)}
(Z^{\bullet}(\hlam))$-modules, i.e., 
$$\hCc_{Q, \beta} := 
\hK^{\G(\hlam)}(Z^{\bullet}(\hlam)) \modfg.$$ 
Let $(\hCc_{Q, \beta})_{f}$ denote
the full subcategory of
finite-dimensional modules in $\hCc_{Q, \beta}$, i.e.,
$$
(\hCc_{Q, \beta})_{f} := 
\hK^{\G(\hlam)}(Z^{\bullet}(\hlam)) \modfd. 
$$
\end{Def} 

We regard the deformed local 
Weyl module $\hW(\hmu)$ corresponding to 
$\hmu \in \lP^{+}_{Q, \beta}$ 
as a $(\tilU, \hR(\hlam))$-bimodule
on which the algebra $\hR(\hlam)$ acts via 
the homomorphism $\theta_{\hmu}$.  
By Proposition~\ref{Prop:KMmu} above,
we can regard $\hW(\hmu) \in \hCc_{Q, \beta}$
for any $\hmu \in \lP^{+}_{Q, \beta}.$
Now we state the main theorem 
of this section.   
   
\begin{Thm}
\label{Thm:main}
With the above notation, the followings hold:
\begin{enumerate}
\item \label{Thm:main:HL}
Via the completed Nakajima homomorphism 
$\widehat{\Phi}_{\hlam}$, we can identify 
the category $(\hCc_{Q, \beta})_{f}$ 
with the Hernandez-Leclerc category $\Cc_{Q, \beta}$.
More precisely, the pullback functor
$(\hCc_{Q, \beta})_{f} \to \Cc_{\g}; \, 
M \mapsto (\widehat{\Phi}_{\hlam})^{*}M$
gives an equivalence of categories 
$(\hCc_{Q, \beta})_{f} \xrightarrow{\simeq} \Cc_{Q, \beta}$;  
\item \label{Thm:main:affhw}
The completed Hernandez-Leclerc category
$\hCc_{Q, \beta}$ is an affine highest weight category
for the poset $(\lP^{+}_{Q, \beta}, \le)$.
The standard module (resp.~proper standard module, 
proper costandard module)
associated with $\hmu \in \lP^{+}_{Q, \beta}$
is given by the deformed local Weyl module $\hW(\hmu)$
(resp.~local Weyl module $W(\hmu)$,
dual local Weyl module $W^{\vee}(\hmu)$). 
\end{enumerate}
\end{Thm} 

\begin{Rem}
Note that there is a small conflict of terminologies.
As we noted in Remark~\ref{Rem:locWeyl}, 
local Weyl modules $W(\hmu)$ are the same as the standard modules 
in the sense of Nakajima and Varagnolo-Vasserot.
But, in our context of affine highest weight category, 
they are not standard modules but proper standard modules. 
\end{Rem}

A proof of Theorem~\ref{Thm:main}
is given after Corollary \ref{Cor:Kernel}.
We need some lemmas.
  
\begin{Lem}
\label{Lem:KZmu}
For $\hnu \in \lQ_{Q}^{+}$ with
$\hmu := \hlam - \hnu \in \lP^{+}_{Q, \beta}$,
let $Z_{\hmu}^{\bullet}$
denote the inverse image of the stratum
$\Mregg(\hnu, \hlam)$ along the canonical morphism
$\pi^{\bullet} \colon Z^{\bullet}(\hlam) 
\to \M_{0}^{\bullet}(\hlam)$. 
\begin{enumerate}
\item
\label{Lem:KZmu:bimod}
As a $(\tilU, \tilU)$-bimodule, we have
$
\hK^{\G(\hlam)}(Z^{\bullet}_{\hmu}) \cong 
\hW(\hmu) \otimes_{\hR(\hmu)}
\hW(\hmu)^{\sharp};
$
\item
\label{Lem:KZmu:1}
We have 
$\hK^{\G(\hlam)}_{1, \mathrm{top}}(Z^{\bullet}_{\hmu}) = 0;$
\item
\label{Lem:KZmu:0}
The comparison map gives an isomorphism
$\hK^{\G(\hlam)}(Z^{\bullet}_{\hmu}) \cong
\hK^{\G(\hlam)}_{0, \mathrm{top}}(Z^{\bullet}_{\hmu}).$
\end{enumerate}
\end{Lem}
\begin{proof}
Fix a point $x \in \Mregg(\hnu, \hlam)$.
Then
we have an isomorphism
\begin{equation*}
Z_{\hmu}^{\bullet} \cong \G(\hlam) 
\times_{(\Stab_{\G(\hlam)} x)} 
\left( \M^{\bullet}(\hlam)_{x} \times
\M^{\bullet}(\hlam)_{x} \right).  
\end{equation*}
A similar computation as in the proof of 
Proposition~\ref{Prop:KMmu} yields:
\begin{align*}
\hK^{\G(\hlam)}(Z^{\bullet}_{\hmu})
&
\cong \hK^{\G(\hlam)}\left(
\G(\hlam) 
\times_{(\Stab_{\G(\hlam)} x)} 
\left( \M^{\bullet}(\hlam)_{x} \times
\M^{\bullet}(\hlam)_{x} \right) \right) \\
&\cong
\K^{\G(\mu)}
\left(
\LL(\mu)
\times \LL(\mu)
\right)
\otimes_{\R(\mu)}\hR(\hmu) \\
&\cong
\K^{\G(\mu)}(\LL(\mu)) \otimes_{\R(\mu)}
\hR(\hmu) \otimes_{\R(\mu)}
\K^{\G(\mu)}(\LL(\mu)),
\end{align*}
where 
the last isomorphism is due to
Theorem~\ref{Thm:Kfibers} (\ref{Thm:Kfibers:Kunneth}). 
Then Proposition~\ref{Prop:KMmu} proves 
the assertion (\ref{Lem:KZmu:bimod}).
Because the same computation is valid for 
equivariant $K$-homologies, we
have
$$
\hK_{i, \mathrm{top}}^{\G(\hlam)}(Z^{\bullet}_{\hmu})
\cong 
\K_{0, \mathrm{top}}^{\G(\mu)}(\LL(\mu)) \otimes_{\R(\mu)}
\hR(\hmu) \otimes_{\R(\mu)}
\K_{i, \mathrm{top}}^{\G(\mu)}(\LL(\mu)),
$$
for $i= 0,1$.
Then Theorem~\ref{Thm:Kfibers} (\ref{Thm:Kfibers:K1}),
(\ref{Thm:Kfibers:K0}) prove the assertions
(\ref{Lem:KZmu:1}), (\ref{Lem:KZmu:0}).
\end{proof}
 
We fix a total ordering
$\{\lambda_{1}, \lambda_{2}, \ldots, \lambda_{l} \}$
of the set 
$\cP^{+}_{\le \lambda} := \{ \mu \in \cP^{+} \mid \mu \le \lambda\}$
such that $\lambda = \lambda_{l}$ and 
$i<j$ whenever $\lambda_{i} < \lambda_{j}$. 
Let $\nu_{i} := \lambda - \lambda_{i} \in \cQ^{+}$
and $\N_{i} := \M_{0}^{\mathrm{reg}}(\nu_{i}, \lambda)$
for $i \in \{ 1, \ldots, l\}.$
Then the stratification (\ref{Eq:stratif1}) is written as:
\begin{equation}
\nonumber
\M_{0}(\lambda) = \N_{1} \sqcup \N_{2} 
\sqcup \cdots \sqcup \N_{l}
\end{equation}
with $\N_{l} = \{ 0 \}$.
For each $i \in \{1, \ldots, l \}$, 
we set $\N_{\le i} := \bigsqcup_{j \le i}\N_{j} \subset 
\M_{0}(\lambda)$.
Note that $\N_{i}$ is a closed subvariety
of $\N_{\le i}$ and its complement
is $\N_{\le i-1}$. 

We set
$\N_{i}^{\bullet} := \N_{i} \cap \M^{\bullet}_{0}(\hlam)$
and 
$\N_{\le i}^{\bullet} :=
\N_{\le i} \cap \M^{\bullet}_{0}(\hlam)$
for each $i \in \{1, \ldots, l \}$.
We fix a total ordering
$\{\hlam_{i,1}, \hlam_{i,2}, \ldots, \hlam_{i,k_{i}}\}$
of the set 
$\mathsf{cl}^{-1}(\lambda_{i}) \cap \lP^{+}_{Q, \beta},$
where we define $k_{i} = 0$
if $\mathsf{cl}^{-1}(\lambda_{i}) \cap \lP^{+}_{Q, \beta} 
= \varnothing$.
We simplify the notation by setting
$\R_{i,s} := \R(\hlam_{i,s}), \,
\hR_{i,s} := \hR(\hlam_{i,s}), \,
\theta_{i,s} := \theta_{\hlam_{i,s}}$
for each $i \in I$ and  $s \in \{1, \ldots, k_{i} \}$.
Set $\hnu_{i,s} := \hlam - \hlam_{i,s} \in \lQ^{+}_{Q}$
and $\N^{\bullet}_{i,s} := \Mregg(\hnu_{i,s}, \hlam)$.
Note that we have $\mathsf{cl}(\hnu_{i,s}) = \nu_{i}$
and that $\N^{\bullet}_{i,s}$
gives a connected component of $\N_{i}^{\bullet}$
for any $s \in \{1, \ldots, k_i \}$. 
Namely, we get the decomposition 
\begin{equation}
\label{Eq:decompNi}
\N_{i}^{\bullet} = \bigsqcup_{s = 1}^{k_{i}} 
\N^{\bullet}_{i,s}
\end{equation}
of $\N_{i}^{\bullet}$ into connected components.
We define a subvariety $Z^{\bullet}_{i}$ 
(resp.~$Z^{\bullet}_{\le i},
Z^{\bullet}_{i,s}$) 
of $Z^{\bullet}(\hlam)$
to be the inverse image of the subvariety
$\N^{\bullet}_{i}$ (resp.~$\N_{\le i}^{\bullet}, \N^{\bullet}_{i,s}$)
along the canonical morphism 
$\pi^{\bullet} \colon Z^{\bullet}(\hlam) \to \M_{0}^{\bullet}(\hlam)$. 
From the decomposition (\ref{Eq:decompNi}), 
we have
\begin{equation}
\label{Eq:decompZi}
Z_{i}^{\bullet} = \bigsqcup_{s = 1}^{k_{i}} 
Z^{\bullet}_{i,s}.
\end{equation}
By construction, $Z_{i}^{\bullet}$ is a closed subvariety
of $Z^{\bullet}_{\le i}$ and its complement is
$Z^{\bullet}_{\le i-1}$ for each $i \in \{2, \ldots, l \}.$
From (\ref{Eq:MVseq}), 
we have an exact sequence:
\begin{equation}
\label{Eq:MVseqZ}
\xy
\xymatrix{
\hK^{\G(\hlam)}(Z^{\bullet}_{i})
\ar[r]^-{\imath_{*}}
&
\hK^{\G(\hlam)}(Z^{\bullet}_{\le i})
\ar[r]^-{\jmath^{*}}
&
\hK^{\G(\hlam)}(Z^{\bullet}_{\le i-1}) 
\ar[r]
&
0,
}
\endxy
\end{equation}
where 
$\imath \colon Z^{\bullet}_{i} \hookrightarrow Z^{\bullet}_{\le i}$
and 
$\jmath \colon Z^{\bullet}_{\le i-1} 
\hookrightarrow Z^{\bullet}_{\le i}$
are the inclusions.

\begin{Lem}
\label{Lem:MVseqZ}
The map $\imath_{*}$ in the sequence
(\ref{Eq:MVseqZ}) is injective.
Therefore we have the following short exact sequence:
\begin{equation*}
\xy
\xymatrix{
0
\ar[r]
&
\hK^{\G(\hlam)}(Z^{\bullet}_{i})
\ar[r]^-{\imath_{*}}
&
\hK^{\G(\hlam)}(Z^{\bullet}_{\le i})
\ar[r]^-{\jmath^{*}}
&
\hK^{\G(\hlam)}(Z^{\bullet}_{\le i-1}) 
\ar[r]
&
0,
}
\endxy
\end{equation*}
for each $i \in \{2, \ldots, l\}$.
\end{Lem}
\begin{proof}
We shall prove that
$\hK_{1, \mathrm{top}}^{\G(\hlam)}(Z_{\le i}^{\bullet})=0$
and the comparison map
$
\hK^{\G(\hlam)}(Z_{\le i}^{\bullet}) \to
\hK_{0, \mathrm{top}}^{\G(\hlam)}(Z_{\le i}^{\bullet})
$
is an isomorphism for each $i \in \{1, \ldots, l\}$.
If we prove them, the exact hexagon
(\ref{Eq:hexagon}) yields the assertion.
We proceed by induction on $i$.

When $i=1$, from (\ref{Eq:decompZi}), we have
$$
\hK_{j, \mathrm{top}}^{\G(\hlam)}(Z_{\le 1}^{\bullet})=
\hK_{j, \mathrm{top}}^{\G(\hlam)}(Z_{1}^{\bullet})=
\bigoplus_{s=1}^{k_{1}}
\hK_{j, \mathrm{top}}^{\G(\hlam)}(Z_{1,s}^{\bullet}),
$$   
where $j=0,1.$
From Lemma~\ref{Lem:KZmu}, we know that 
$\hK_{1, \mathrm{top}}^{\G(\hlam)}(Z_{1,s}^{\bullet})=0$
and   
$
\hK^{\G(\hlam)}(Z_{1,s}^{\bullet}) \cong
\hK_{0, \mathrm{top}}^{\G(\hlam)}(Z_{1,s}^{\bullet})
$
for each $s \in \{1, \ldots, k_1 \}$. 
Therefore we are done in this case.

Let $i > 1$ and assume that we know that
$\hK_{1, \mathrm{top}}^{\G(\hlam)}(Z_{\le i-1}^{\bullet})=0$
and   
$
\hK^{\G(\hlam)}(Z_{\le i-1}^{\bullet}) \cong
\hK_{0, \mathrm{top}}^{\G(\hlam)}(Z_{\le i-1}^{\bullet}).
$
By the same reason for the case $i=1$ above,  
we have 
$\hK_{1, \mathrm{top}}^{\G(\hlam)}(Z_{i}^{\bullet})=0$
and   
$
\hK^{\G(\hlam)}(Z_{i}^{\bullet}) \cong
\hK_{0, \mathrm{top}}^{\G(\hlam)}(Z_{i}^{\bullet}).
$
Then 
we see that 
$\hK_{1, \mathrm{top}}^{\G(\hlam)}(Z^{\bullet}_{\le i}) = 0$
from the exact hexagon (\ref{Eq:hexagon}). 
Moreover 
we have the following commutative 
diagram:
$$
\xy
\xymatrix{
&
\hK^{\G(\hlam)}(Z^{\bullet}_{i})
\ar[r]
\ar[d]^{\cong}
&
\hK^{\G(\hlam)}(Z^{\bullet}_{\le i})
\ar[r]
\ar[d]
&
\hK^{\G(\hlam)}(Z^{\bullet}_{\le i-1}) 
\ar[r]
\ar[d]^{\cong}
&
0
\\
0
\ar[r]
&
\hK_{0, \mathrm{top}}^{\G(\hlam)}(Z^{\bullet}_{i})
\ar[r]
&
\hK_{0, \mathrm{top}}^{\G(\hlam)}(Z^{\bullet}_{\le i})
\ar[r]
&
\hK_{0, \mathrm{top}}^{\G(\hlam)}(Z^{\bullet}_{\le i-1}) 
\ar[r]
&
0,
}
\endxy
$$
where the upper row is the exact sequence (\ref{Eq:MVseqZ}),
the lower row is the exact sequence coming from 
the exact hexagon (\ref{Eq:hexagon}).
All vertical arrows are the comparison maps.
Applying the five lemma, we see that 
the middle comparison map 
$\hK^{\G(\hlam)}(Z_{\le i}^{\bullet}) \to
\hK_{0, \mathrm{top}}^{\G(\hlam)}(Z_{\le i}^{\bullet})$
is also an isomorphism.   
\end{proof}

\begin{Cor} \label{Cor:fg}
The completed convolution algebra $\hK^{\G(\hlam)}(Z^{\bullet}(\hlam))$
is finitely generated as an $\hR(\hlam)$-module.
\end{Cor}
\begin{proof}
The assertion follows from Lemmas \ref{Lem:KZmu} and \ref{Lem:MVseqZ}.
\end{proof}

Recall we have defined a quotient
algebra $U_{\le \lambda}$ of 
the modified quantum loop algebra $\tilU$
in Subsection \ref{Ssec:affcell} by (\ref{Eq:Ule}).  

\begin{Lem}
\label{Lem:factor}
The completed Nakajima homomorphism
$\widehat{\Phi}_{\hlam}\colon
\tilU \to \hK^{\G(\hlam)}(Z^{\bullet}(\hlam)) 
$
factors through the quotient
$\tilU \to U_{\le \lambda}.$
\end{Lem}
\begin{proof}
Since we have 
$
\hK^{\G(\hlam)}(Z^{\bullet}(\hlam)) \cong
\varprojlim
\hK^{\G(\hlam)}(Z^{\bullet}(\hlam)) / \rr_{\hlam}^{N},
$
it is enough to prove that 
the composition 
\begin{equation}
\label{Eq:maptoquot}
\tilU \xrightarrow{\widehat{\Phi}_{\hlam}}
\hK^{\G(\hlam)}(Z^{\bullet}(\hlam))
\to \hK^{\G(\hlam)}(Z^{\bullet}(\hlam)) / \rr_{\hlam}^{N}
\end{equation}
factors through the quotient $U_{\le \lambda}$
for every $N \in \Z_{>0}$.
We can discuss the composition factors of 
$\hK^{\G(\hlam)}(Z^{\bullet}(\hlam)) / \rr_{\hlam}^{N}$
as a left $\tilU$-module because it is finite-dimensional by Corollary~\ref{Cor:fg}.  
By Lemma~\ref{Lem:KZmu} (\ref{Lem:KZmu:bimod})
and the exact sequence (\ref{Eq:MVseqZ}),
we see that every composition factor is a subquotient of
the global Weyl modules $\bW(\mu)$ 
for some $\mu \le \lambda.$
Therefore the ideal $\bigcap_{\mu \le \lambda}
\Ann_{\tilU}\bW(\mu)$ is included in the kernel
of the map (\ref{Eq:maptoquot}).
\end{proof}

By Lemma~\ref{Lem:factor} above, 
we have the induced homomorphism
$\widehat{\Phi}_{\hlam}\colon U_{\le \lambda} \to 
\hK^{\G(\hlam)}(Z^{\bullet}(\hlam))$,
which we denote by the same symbol 
$\widehat{\Phi}_{\hlam}$.

\begin{Prop}
\label{Prop:compare}
For each $N \in \Z_{>0}$, the homomorphism 
$\widehat{\Phi}_{\hlam}^{N}$
given by the composition
$$
\widehat{\Phi}_{\hlam}^{N} \colon
U_{\le \lambda} \xrightarrow{\widehat{\Phi}_{\hlam}}
\hK^{\G(\hlam)}(Z^{\bullet}(\hlam)) 
\to \hK^{\G(\hlam)}(Z^{\bullet}(\hlam)) / \rr_{\hlam}^{N}
$$
is surjective.
In particular, the forgetful functor from
$\hCc_{Q, \beta}$ to the category
of $(U_{\le \lambda}, \hR(\hlam))$-bimodules
is fully faithful.
\end{Prop} 
\begin{proof}
Fix $N \in \Z_{>0}$.
Using the homomorphism $\widehat{\Phi}_{\hlam}$,
we compare
the affine cellular structure of the algebra $U_{\le \lambda}$
with the filtration of the algebra 
$\hK^{\G(\hlam)}(Z^{\bullet}(\hlam))$ 
coming from the geometric stratification of 
$\M_{0}^{\bullet}(\hlam)$
as in Lemma~\ref{Lem:MVseqZ}.
First, for each $i$,
we observe that there is the following isomorphism of 
$(U_{\le \lambda}, U_{\le \lambda})$-bimodules
by Lemma~\ref{Lem:KZmu} (\ref{Lem:KZmu:bimod}):
\begin{align}
\hK^{\G(\hlam)}(Z^{\bullet}_{i})/\rr_{\hlam}^{N}
&\cong 
\bigoplus_{s=1}^{k_i}
\hK^{\G(\hlam)}(Z^{\bullet}_{i,s})/\rr_{\hlam}^{N}
\nonumber
\\
& \cong
\bigoplus_{s=1}^{k_i}
\hW(\hlam_{i,s}) \otimes_{\hR_{i,s}} 
\left(
\frac{\hR_{i,s}}{\langle \theta_{i,s}(\rr_{\hlam})^{N} \rangle}
\right)
\otimes_{\hR_{i,s}} \hW(\hlam_{i,s})^{\sharp}.
\nonumber \\
& \cong
\bigoplus_{s=1}^{k_i}
\bW(\lambda_{i}) \otimes_{\R(\lambda_{i})} 
\left(
\frac{\R(\lambda_{i})}{\R(\lambda_{i}) \cap
\langle \theta_{i,s}(\rr_{\hlam})^{N} \rangle}
\right)
\otimes_{\R(\lambda_{i})} \bW(\lambda_{i})^{\sharp}
\nonumber \\
& \cong
\bW(\lambda_{i}) \otimes_{\R(\lambda_{i})} 
\left( \frac{\R(\lambda_{i})}{
\prod_{s=1}^{k_{i}}
\R(\lambda_{i}) \cap
\langle \theta_{i,s}(\rr_{\hlam})^{N} \rangle} \right)
\otimes_{\R(\lambda_{i})} \bW(\lambda_{i})^{\sharp},
\label{Eq:CRT}
\end{align} 
where  
we apply the Chinese remainder theorem
for the last isomorphism.
This is possible 
because the maximal ideals associated with
the primary ideals $\R(\lambda_{i}) \cap
\langle \theta_{i,s}(\rr_{\hlam})^{N} \rangle$
are distinct by Lemma~\ref{Lem:rest}.
By (\ref{Eq:CRT}), we see that the $K$-group
$\hK^{\G(\hlam)}(Z^{\bullet}_{i})$
is cyclic as $(U_{\le \lambda}, U_{\le \lambda})$-bimodule.
By construction, we can easily see that
the class in $\hK^{\G(\hlam)}(Z^{\bullet}_{i})$
obtained as the restriction of 
the class $\Delta_{*}[\mathcal{O}_{\M(\nu_{i}, \lambda)}]$ 
corresponds to the cyclic vector 
$w_{\lambda_{i}} \otimes 1 \otimes w_{\lambda_{i}}$
of the RHS of (\ref{Eq:CRT}).  

Recall the ideals $I_{i}$ of $U_{\le \lambda}$
defined by (\ref{Eq:cellularideals}). 
By downward induction on $i \in \{1, \ldots, l \}$,
we shall construct algebra
homomorphisms
$f_{i}^{N} \colon U_{\le \lambda}/I_{i} \to
\hK^{\G(\hlam)}(Z^{\bullet}_{\le i})/ \rr_{\hlam}^{N}$
and 
$(U_{\le \lambda}, U_{\le \lambda})$-bimodule 
homomorphisms
$g_{i}^{N} \colon I_{i-1}/I_{i} \to 
\hK^{\G(\hlam)}(Z^{\bullet}_{i})/ \rr_{\hlam}^{N}$,
which make the
following diagrams commute:
\begin{equation}
\label{Eq:compare}
\xy
\xymatrix{
0
\ar[r]
&
\bW(\lambda_{i}) \otimes_{\R(\lambda_{i})} 
\bW(\lambda_{i})^{\sharp}
\ar[r]^-{a_{i}}
\ar[d]^{g_{i}^{N}}
&
U_{\le \lambda}/I_{i}
\ar[r]^-{b_{i}}
\ar[d]^{f_{i}^{N}}
&
U_{\le \lambda}/I_{i-1}
\ar[r]
\ar[d]^{f_{i-1}^{N}}
&
0
\\
&
\hK^{\G(\hlam)}(Z^{\bullet}_{i})/ \rr_{\hlam}^{N}
\ar[r]
&
\hK^{\G(\hlam)}(Z^{\bullet}_{\le i})/ \rr_{\hlam}^{N}
\ar[r]
&
\hK^{\G(\hlam)}(Z^{\bullet}_{\le i-1})/ \rr_{\hlam}^{N} 
\ar[r]
&
0,
}
\endxy
\end{equation} 
where the upper row is the exact sequence coming from
the ideal chain and the lower row 
is the exact sequence coming from (\ref{Eq:MVseqZ}).  

We start from $i=l$. Define $f_{l}^{N}$ to be the 
homomorphism $\widehat{\Phi}_{\hlam}^{N}$.
Recall that the Nakajima homomorphism
$\Phi_{\lambda}$ sends the element 
$a_{\lambda}$ to the class
$\Delta_{*}[\mathcal{O}_{\M(0, \lambda)}]$
(see Theorem~\ref{Thm:Nakajima}).
Therefore,
by our observation in the previous paragraph,
if we define the homomorphism $g_{l}^{N}$
to give the quotient map from
$\bW(\lambda) \otimes_{\R(\lambda)} 
\bW(\lambda)^{\sharp}$
to the RHS of (\ref{Eq:CRT})
via the isomorphism (\ref{Eq:CRT}),
the left square in (\ref{Eq:compare})
commutes.
Then we have the induced homomorphism
$f_{l-1}^{N}$ between the cokernels.

The induction step is similar.
Assume that we have defined $f_{i}^{N}$.
By construction, $f_{i}^{N}$ sends the 
image of the element $a_{\lambda_{i}}$
to the restriction of the class
$\Delta_{*}[\mathcal{O}_{\M(\nu_{i}, \lambda)}]$.
Then if we define $g_{i}^{N}$
to be the quotient map via the
isomorphism (\ref{Eq:CRT}),
the left square in (\ref{Eq:compare})
commutes.
We get $f_{i-1}^{N}$ as the induced homomorphism 
between the cokernels.

Note that
we have $f_{1}^{N} = g_{1}^{N}$
and all the homomorphisms $g_{i}^{N}$ are surjective
by construction.
Therefore we can apply the five lemma to
the diagram ($\ref{Eq:compare})$
inductively, starting from the case $i=2$,
to prove that every $f_{i}^{N}$ is also surjective.
Eventually we see that the homomorphism
$f_{l}^{N} = \widehat{\Phi}_{\hlam}^{N}$ is surjective.
\end{proof}

Using the notation in 
the above proof of 
Proposition~\ref{Prop:compare},
we define
$K^{N}_{\le i} := \Ker 
f_{i}^{N},
K_{i}^{N}
:= \Ker g_{i}^{N}
$
for each $i \in \{1, \ldots, l \}$ and 
$N \in \Z_{>0}$.

\begin{Prop}
\label{Prop:Kernel}
For each $i \in \{1, \ldots, l\}$ and 
for any $N_{1}, N_{2} \in \Z_{>0}$, there 
exists a positive integer $N > N_{1} + N_{2}$ 
satisfying
\begin{enumerate}
\item \label{Prop:Kernel:g}
$K_{i}^{N} \subset K_{i}^{N_{1}} \cdot K_{i}^{N_{2}}$;
\item \label{Prop:Kernel:f}
$K_{\le i}^{N} \subset K_{\le i}^{N_{1}} \cdot K_{\le i}^{N_{2}}$. 
\end{enumerate}
\end{Prop}

\begin{proof}
We first prove the assertion (\ref{Prop:Kernel:g}).
Assume that $k_{i} = 0$.
Thus we have $g_{i}^{N}=0$ and hence
$K^{N}_{i} = I_{i-1}/I_{i}$ for any $N \in \Z_{>0}.$
In this case, the assertion (\ref{Prop:Kernel:g})
is equivalent to the assertion 
$
(I_{i-1}/I_{i})^{2} = I_{i-1}/I_{i},
$
which follows from Theorem~\ref{Thm:cellular}.
Next we consider the case $k_{i} \neq 0$.
By (\ref{Eq:CRT}), we have
$$
K_{i}^{N} =
\bW(\lambda_{i}) \otimes_{\R(\lambda_{i})} 
\left(
\prod_{s=1}^{k_{i}}
\left(\R(\lambda_{i}) \cap
\langle \theta_{i,s}(\rr_{\hlam})^{N} \rangle
\right) \right)
\otimes_{\R(\lambda_{i})} \bW(\lambda_{i})^{\sharp},
$$
for each $N \in \Z_{>0}$.
By Lemma~\ref{Lem:rest}, 
the ideal $\R(\lambda_{i}) \cap
\langle \theta_{i,a}(\rr_{\hlam})^{N} \rangle$
is a primary ideal whose associated prime is 
the maximal ideal $\rr_{\lambda_{i}, \hlam_{i,a}}$.
Thus for a sufficiently large $N > 0$,
we have 
$\R(\lambda_{i}) \cap
\langle \theta_{i,s}(\rr_{\hlam})^{N} \rangle
\subset
(\R(\lambda_{i}) \cap
\langle \theta_{i,s}(\rr_{\hlam})^{N_{1}} \rangle)
(\R(\lambda_{i}) \cap
\langle \theta_{i,s}(\rr_{\hlam})^{N_{2}} \rangle)
$.  
Then we obtain
the assertion 
$K_{i}^{N} \subset K_{i}^{N_{1}} \cdot K_{i}^{N_{2}}$.

We prove the assertion (\ref{Prop:Kernel:f}) 
by induction on $i$.
The case $i=1$ follows from (\ref{Prop:Kernel:g})
since $f_{1}^{N} = g^{N}_{1}$.
We assume that $i > 1$ and the assertion (\ref{Prop:Kernel:f}) is
true for $i-1$. 
For given $N_{1}, N_{2} \in \Z_{>0}$, we can find
an integer $N' > N_{1} + N_{2}$
such that $K_{i}^{N'} \subset 
K_{i}^{N_{1}} \cdot K_{i}^{N_{2}}$ by (\ref{Prop:Kernel:g}).
By the Artin-Rees lemma, we find an integer $N'' > N'$
such that we have 
$$
\hK^{\G(\hlam)}(Z^{\bullet}_{i}) \cap \left( \rr_{\hlam}^{N''} \hK^{\G(\hlam)}(Z_{\le i}^{\bullet})
\right) 
\subset \rr_{\hlam}^{N'} \hK^{\G(\hlam)}(Z^{\bullet}_{i}),
$$
which implies  
\begin{equation}
\label{Eq:Kernel_incl}
\Ker (f_{i}^{N''} \circ a_{i})
\subset 
K^{N'}_{i}
\subset 
K_{i}^{N_{1}} \cdot K_{i}^{N_{2}},
\end{equation}
where $a_{i}$ is the inclusion
$I_{i-1}/I_{i} \hookrightarrow U_{\le \lambda}/I_{i}$
as in the diagram (\ref{Eq:compare}).
Applying the snake lemma 
to the following diagram:
$$
\xy
\xymatrix{
0 
\ar[r]
&
\bW(\lambda_{i})
\otimes_{\R(\lambda_{i})}
\bW(\lambda_{i})^{\sharp}
\ar[r]^-{a_{i}}
\ar[d]^-{f_{i}^{N''} \circ a_{i}}
&
U_{\le \lambda} / I_{i}
\ar[d]^{f^{N''}_{i}}
\ar[r]^{b_{i}}
&
U_{\le \lambda}/ I_{i-1}
\ar[r]
\ar[d]^{f_{i-1}^{N''}}
&
0
\\
0
\ar[r]
&
\Ima (f_{i}^{N''} \circ a_{i}) 
\ar[r]
&
\hK^{\G(\hlam)}(Z_{\le i}^{\bullet})/ \rr^{N''}_{\hlam}
\ar[r]
&
\hK^{\G(\hlam)}(Z^{\bullet}_{\le i-1})/\rr^{N''}_{\hlam}
\ar[r]
& 
0,
}
\endxy
$$
we get
an exact sequence
\begin{equation}
\label{Eq:Kernel}
\xy
\xymatrix{
0
\ar[r]
&
\Ker (f^{N''}_{i} \circ a_{i})
\ar[r]
&
K_{\le i}^{N''}
\ar[r]
&
K_{\le i-1}^{N''}
\ar[r]
&
0.}
\endxy
\end{equation}
Let $N_{1}^{\prime}, N_{2}^{\prime}$ 
be any two integers larger than $N''$. 
By induction hypothesis, there is an integer 
$N > N_{1}^{\prime} + N_{2}^{\prime}$
such that $K_{\le i-1}^{N} \subset 
K_{\le i-1}^{N_{1}^{\prime}} \cdot K_{\le i-1}^{N_{2}^{\prime}}$.
We shall prove that the assertion (\ref{Prop:Kernel:f})
holds for this $N$.
Let $x \in K_{\le i}^{N}$ be an arbitrary element.
Note that we have
$b_{i}(x) \in K_{\le i-1}^{N}
\subset K_{\le i-1}^{N_{1}^{\prime}} \cdot 
K_{\le i-1}^{N_{2}^{\prime}}.$
Since the quotient map 
$b_{i} \colon U_{\le \lambda}/I_{i} \to U_{\le \lambda}/I_{i-1}$
induces the surjection
$K_{\le i}^{N_{1}^{\prime}} \cdot 
K_{\le i}^{N_{2}^{\prime}}
\twoheadrightarrow
K_{\le i-1}^{N_{1}^{\prime}} \cdot 
K_{\le i-1}^{N_{2}^{\prime}}$,
we can choose an element 
$y \in K_{\le i}^{N_{1}^{\prime}} \cdot 
K_{\le i}^{N_{2}^{\prime}}$ 
so that $b_{i}(x-y)=0$.
By the exact sequence (\ref{Eq:Kernel}),
there is an element $y^{\prime} \in 
\Ker (f_{i}^{N''} \circ a_{i})$ 
such that $x = y + a_{i}(y^{\prime})$.
By (\ref{Eq:Kernel_incl}), we see that
$a_{i}(y^{\prime}) \in 
K_{\le i}^{N_{1}} \cdot 
K_{\le i}^{N_{2}}$.
Therefore we have 
$x \in K_{\le i}^{N_{1}} \cdot 
K_{\le i}^{N_{2}}$
as desired. 
\end{proof}

As the special case $i=l$ of 
Proposition~\ref{Prop:Kernel}
(\ref{Prop:Kernel:f}), we obtain the following.

\begin{Cor}
\label{Cor:Kernel}
For any $N_{1}, N_{2} \in \Z_{>0}$, 
there is a positive integer $N$ such that
$\Ker \widehat{\Phi}_{\hlam}^{N}
\subset (\Ker \widehat{\Phi}_{\hlam}^{N_{1}}) \cdot
(\Ker \widehat{\Phi}_{\hlam}^{N_{2}}).$
\end{Cor}

\begin{proof}[Proof of Theorem~\ref{Thm:main}
(\ref{Thm:main:HL})]
Note that we have
$(\hCc_{Q, \beta})_{f} 
= \bigcup_{N}
\hCc_{Q, \beta}^{N}$, where
$\hCc_{Q, \beta}^{N} : = \left(
\hK^{\G(\hlam)}(Z^{\bullet}(\hlam))/ \rr^{N}_{\hlam}
\right) \modfg$ is the full subcategory of $\hCc_{Q, \beta}$
consisting of modules $M$ with $\rr^{N}_{\hlam} M = 0$.
Thus, by Proposition~\ref{Prop:compare},
we see that the pullback functor 
$(\hCc_{Q, \beta})_{f} \to \Cc_{\g}; \, 
M \mapsto (\widehat{\Phi}_{\hlam})^{*}M$
is fully faithful.
To prove that it is essentially surjectivite onto $\Cc_{Q, \beta}$,
it is enough to show that
for each module $M \in \Cc_{Q, \beta}$
there is a positive integer $N \in \Z_{>0}$
such that $(\Ker \widehat{\Phi}_{\hlam}^{N}) M = 0$.
We proceed by induction on the length of $M$.
When $M=L(\hmu)$ is a simple module of $\Cc_{Q, \beta}$
of $\ell$-highest weight $\hmu \in \lP^{+}_{Q, \beta}$,
we have
$(\Ker \widehat{\Phi}_{\hlam}^{1})L(\hmu)=0$ because we know
$(\Ker \widehat{\Phi}_{\hlam}^{1}) W(\hmu) = 0$.
For induction step, we write the given module $M$
as an extension of two non-zero modules 
$M_{1}, M_{2} \in \Cc_{Q, \beta}$.
By induction hypothesis, there are integers 
$N_{1}, N_{2} \in \Z_{>0}$ such that
$(\Ker \widehat{\Phi}_{\hlam}^{N_{1}}) M_{1} = 
(\Ker \widehat{\Phi}_{\hlam}^{N_{2}}) M_{2} = 0$.
We can find an integer $N \in \Z_{>0}$ such that
$\Ker \widehat{\Phi}_{\hlam}^{N} \subset 
(\Ker \widehat{\Phi}_{\hlam}^{N_{1}}) \cdot
(\Ker \widehat{\Phi}_{\hlam}^{N_{2}})$
by Corollary \ref{Cor:Kernel}.
Then we have 
$(\Ker \widehat{\Phi}_{\hlam}^{N}) M = 0$.
\end{proof}

\begin{proof}[Proof of Theorem~\ref{Thm:main}
(\ref{Thm:main:affhw})]
We construct an affine quasi-heredity chain
of the algebra $\hK^{\G(\hlam)}(Z^{\bullet}(\hlam))$.
Let $\{\hlam_{1}, \ldots, \hlam_{m} \}$ be a total ordering
of the set $\lP_{Q, \beta}^{+}$
satisfying the condition that 
we have $j < k$ whenever $\hlam_{j} < \hlam_{k}$.
Note that $\hlam_{m} = \hlam.$ 
Put $\hnu_{i} := \hlam - \hlam_{i}$ and set
$\mathcal{N}_{i} :=
\Mregg(\hnu_{i}, \hlam), \,
\mathcal{N}_{\ge i} := \bigsqcup_{j \ge i} \mathcal{N}_{j}, \,
\mathcal{N}_{\le i} := \bigsqcup_{j \le i} \mathcal{N}_{j}.$
We denote the inverse image of
$\mathcal{N}_{i}$ (resp.~$\mathcal{N}_{\ge i}, 
\mathcal{N}_{\le i}$) 
along the canonical morphism 
$\pi \colon Z^{\bullet}(\hlam) \to \M^{\bullet}_{0}(\hlam)$
by $\mathcal{Z}_{i}$ 
(resp.~$\mathcal{Z}_{\ge i}, \mathcal{Z}_{\le i}$).
By construction, 
$\mathcal{Z}_{\ge i+1}$ (resp.~$\mathcal{Z}_{i}$)
is a closed subvariety of $\mathcal{Z}_{\ge i}$
(resp.~$\mathcal{Z}_{\le i}$) and 
its complement is 
$\mathcal{Z}_{i}$ (resp.~$\mathcal{Z}_{\le i-1}$).
We also note that 
$\mathcal{Z}_{\ge i}$ is a closed subvariety of
$Z^{\bullet}(\hlam)$ whose complement
is $\mathcal{Z}_{\le i-1}$.
Then we consider the following commutative diagram
arising from (\ref{Eq:MVseq}):
\begin{equation}\label{Eq:nine}
\xy
\xymatrix{
&
0
\ar[d]
&
0
\ar[d]
&
&
\\
&
\hK^{\G(\hlam)}(\mathcal{Z}_{\ge i+1})
\ar[d]
\ar@{=}[r]
&
\hK^{\G(\hlam)}(\mathcal{Z}_{\ge i+1})
\ar[d]
&
&
\\
0
\ar[r]
&
\hK^{\G(\hlam)}(\mathcal{Z}_{\ge i})
\ar[r]
\ar[d]
&
\hK^{\G(\hlam)}(Z^{\bullet}(\hlam))
\ar[r]
\ar[d]
&
\hK^{\G(\hlam)}(\mathcal{Z}_{\le i-1})
\ar[r]
\ar@{=}[d]
&
0
\\
0
\ar[r]
&
\hK^{\G(\hlam)}(\mathcal{Z}_{i})
\ar[r]
\ar[d]
&
\hK^{\G(\hlam)}(\mathcal{Z}_{\le i})
\ar[r]
\ar[d]
&
\hK^{\G(\hlam)}(\mathcal{Z}_{\le i-1})
\ar[r]
&
0
\\
&
0
&
0
&
&
}
\endxy
\end{equation}
Arguing as in Lemma~\ref{Lem:MVseqZ}, 
we see that the left column and the lower row in (\ref{Eq:nine})
are exact. By downward induction on $i$ and 
diagram chases, we see that
the middle row (and hence the middle column) in 
the diagram (\ref{Eq:nine}) is also exact.

Therefore we can regard
$\mathcal{I}_{i} := \hK^{\G(\hlam)}(\mathcal{Z}_{\ge i+1})$
for each $i \in \{0, \ldots, m \}$
as a two-sided ideal of the algebra
$\Aa := \hK^{\G(\hlam)}(Z^{\bullet}(\hlam))$.
What we have to do is to prove that 
the chain of ideals
\begin{equation}
\label{Eq:affheredchainK}
0 = \mathcal{I}_{m} \subsetneq \mathcal{I}_{m-1} \subsetneq \cdots
\subsetneq \mathcal{I}_{1} 
\subsetneq \mathcal{I}_{0} = \Aa
\end{equation}
gives an affine quasi-heredity chain.
Observe that
$$\mathcal{I}_{i-1}/ \mathcal{I}_{i}
\cong \hK^{\G(\hlam)}(\mathcal{Z}_{i})
\cong \hW(\hlam_{i}) \otimes_{\hR(\hlam_{i})} \hW(\hlam_{i})^{\sharp}
\cong \hW(\hlam_{i})^{\oplus s_{i}}$$ 
as a left 
$\Aa$-module,
where
$s_{i} := \dim W(\hlam_{i})$.
By Theorem~\ref{Thm:main} (\ref{Thm:main:HL}),
the category of finite-dimensional modules over
$\Aa/\mathcal{I}_{i}$
is identified with the full subcategory of $\Cc_{Q, \beta}$
consisting of the modules $M$ whose $\ell$-weights belong
to the set $\bigcup_{j \le i} 
\{ \hmu \in \lP_{\Z} \mid \hmu \le \hlam_{j} \}.$ 
For such a module $M$,  we have
$\Hom_{\hCc_{Q, \beta}}(\hW(\hlam_{i}), M) \cong 
M_{\hlam_i}$ 
by Proposition~\ref{Prop:defWeyl} (\ref{Prop:defWeyl:univ}).
In particular, the functor $\Hom_{\hCc_{Q, \beta}} (\hW(\hlam_{i}), -)$
is exact on the subcategory 
$\Aa/\mathcal{I}_{i}
\modfd$.
Since any module $M \in
\Aa / \mathcal{I}_{i}
\modfg$ can be written as a projective limit of 
finite-dimensional modules 
(i.e.,~$M \cong \varprojlim M / \rr_{\hlam}^{N}$), 
we see that
the deformed local Weyl module
$\hW(\hlam_{i})$ is a projective module 
in the category 
$\Aa/\mathcal{I}_{i}\modfg$ with its simple head $L(\hlam_{i})$.
Moreover, we have
$\Hom_{\hCc_{Q, \beta}}(\mathcal{I}_{i-1}/ \mathcal{I}_{i}, 
\Aa/ \mathcal{I}_{i-1})
=0.$
From these observations and Proposition
\ref{Prop:defWeyl}, we conclude that 
the chain (\ref{Eq:affheredchainK}) is 
an affine quasi-heredity chain.

By Theorem~\ref{Thm:CPS} and Remark~\ref{Rem:ordering},
we see that the category $\hCc_{Q,\beta}$
is an affine highest weight category for
the poset $(\lP^{+}, \le)$,
whose standard module (resp.~proper standard module)
associated with $\hmu \in \lP^{+}_{Q,\beta}$
is the deformed local Weyl module $\hW(\hmu)$
(resp.~local Weyl module $W(\hmu)$).
To prove the assertion for proper costandard modules,
we have to show the $\Ext$-orthogonality as in 
Theorem~\ref{Thm:affhw}.
This will be done in Proposition~\ref{Prop:extWeyl} below.
\end{proof}

In the remainder of this section,
we show the $\Ext$-orthogonality
between deformed local Weyl modules
and dual local Weyl modules in the category $\hCc_{Q, \beta}$. 
Note that 
the dual local Weyl module 
$W^{\vee}(\hmu)$ 
associated with $\hmu \in \lP^{+}_{Q, \beta}$ 
actually belongs to $\hCc_{Q, \beta}$
by Proposition~\ref{Prop:dualWeyl} and Theorem~\ref{Thm:main} 
(\ref{Thm:main:HL}).

We need to prepare 
some duality functors.
We temporarily use the ambient category $\mathcal{B}$ of all
$(\tilU, \hR(\hlam))$-bimodules.
Note that each $M \in \mathcal{B}$ has the weight space
decomposition $M = \bigoplus_{\mu \in \cP} M_{\mu}$,
where $M_{\mu} = a_{\mu}M$ 
and each weight space $M_{\mu}$
is preserved by the action of $\hR(\hlam)$.
We define the full dual and 
the topological dual of $M \in \mathcal{B}$ by 
$M^{*} := \Hom_{\kk}(M, \kk)$ and
$D(M) := \bigcup_{N} \Hom_{\kk}(M/ \rr_{\hlam}^{N}, \kk)$
respectively. We equip a left $\tilU$-module structure 
on $M^{*}$ (resp.~on $D(M)$) by twisting   
the natural right $\tilU$-module structure 
with the antipode $S$ (resp.~$S^{-1}$).
Thus we obtain an exact functor
$(-)^{*} \colon \mathcal{B} \to \mathcal{B}^{op}$ and
a right exact functor 
$D \colon \mathcal{B} \to \mathcal{B}^{op}$.
Moreover, if $M \in \mathcal{B}$ is finitely generated 
over $\hR(\hlam)$, we have
$(D(M))^{*} \cong M$. 
If $M \in \Cc_{\g}$ 
(with the trivial $\hR(\hlam)$-action),
the dual $M^{*}$ (resp.~$D(M)$) coincides
with the left dual $M^{*}$ (resp.~the right 
dual ${}^{*}M$) of $M$.

\begin{Prop}
\label{Prop:extWeyl}
For any $\hlam_{1}, \hlam_{2} \in \lP^{+}_{Q, \beta}$ and $i \in \Z_{\ge 0}$,
we have
$$
\Ext_{\hCc_{Q, \beta}}^{i}
(\hW(\hlam_{1}), W^{\vee}(\hlam_{2})) =
\begin{cases} 
\kk & i=0, \hlam_{1}=\hlam_{2}; \\
0 & \text{otherwise}.
\end{cases}
$$
\end{Prop}
\begin{proof}
The case $i=0$ follows from Proposition~\ref{Prop:defWeyl} (\ref{Prop:defWeyl:univ}) and 
the fact that the module $W^{\vee}(\hlam_{2})$ has
a simple socle $L(\hlam_{2})$ and 
$\dim W^{\vee}(\hlam_{2})_{\hlam_{2}} = 1$.

For $i=1$, we consider an extension in $\hCc_{Q, \beta}$:
\begin{equation}
\label{Eq:extension}
0 \to W^{\vee}(\hlam_{2})
\to E \to \hW(\hlam_{1}) \to 0.
\end{equation}
Put $\lambda_{j} := \mathsf{cl}(\hlam_{j})$ for $j=1,2$.
If $\lambda_{1} \not < \lambda_{2}$, then the $\ell$-weight
$\hlam_{1}$ is maximal in $E/ \rr_{\hlam}^{N}$ for 
any $N \in \Z_{>0}$.
By Proposition~\ref{Prop:defWeyl} (\ref{Prop:defWeyl:univ}), we see that the 
sequence (\ref{Eq:extension}) splits.
If $\lambda_{1} < \lambda_{2}$, we apply 
the topological duality functor $D$ to the 
sequence (\ref{Eq:extension}) to get
the following exact sequence:
\begin{equation}
\label{Eq:extension2}
0 \to D(\hW(\hlam_{1}))
\to D(E)
\to W({}^{*}\hlam_{2}). 
\end{equation} 
Since $\lambda_{1} < \lambda_{2} $, 
we have $\hW(\hlam_1)_{\lambda_2} = 0 = D(\hW(\hlam_1))_{-w_0 \lambda_2}$ and hence
$$
\dim D(E)_{-w_{0} \lambda_{2}} = 
\dim E_{w_{0} \lambda_{2}} = 
\dim E_{\lambda_{2}} 
= \dim W^{\vee}(\hlam_{2})_{\lambda_2}
=
1
=
\dim W({}^{*}\hlam_{2})_{-w_{0} \lambda_{2}},$$
where the second equality is due to Weyl group symmetry
coming from integrability.
In particular, the image of 
the weight space $D(E)_{-w_{0} \lambda_{2}}$
coincides with $W({}^{*}\hlam_{2})_{- w_{0} \lambda_{2}}$, 
which generates $W({}^{*}\hlam_{2})$.
Therefore the most right arrow in 
the sequence (\ref{Eq:extension2}) is surjective.
Moreover, 
by the universal property of the local Weyl module
$W({}^{*}\hlam_{2})$, we see that the 
sequence (\ref{Eq:extension2}) is split.
By applying the full duality functor $(-)^{*}$,
we find that the sequence (\ref{Eq:extension})
is also split.
Therefore we have $\Ext^{1}_{\hCc_{Q, \beta}}
(\hW(\hlam_{1}), W^{\vee}(\hlam_{2})) = 0.$

The cases $i>1$ follow from the case $i=1$
by a standard argument in
(affine) highest weight categories.
\end{proof} 

\begin{Rem}
\label{Rem:ext}
There is an alternative proof of
Theorem~\ref{Thm:main} (\ref{Thm:main:affhw})
by
the general theory of geometric extension 
algebras due to Kato~\cite{Kato17} and 
McNamara~\cite{McNamara17}. 
In fact, we have the following 
$\hR(\hlam)$-algebra isomorphisms:
\begin{equation}
\label{Eq:ext}
\hK^{\G(\hlam)}(Z^{\bullet}(\hlam))
\cong 
H_{\bullet}^{G(\hlam)}(Z^{\bullet}(\hlam), \kk)^{\wedge}
\cong 
\Ext^{\bullet}_{G_{\beta}}
(\mathcal{L}, \mathcal{L})^{\wedge},
\end{equation}
where the middle term 
$H_{\bullet}^{G(\hlam)}(Z^{\bullet}(\hlam), \kk)$
denotes the convolution algebra of 
the $G(\hlam)$-equivariant Borel-Moore homology 
of $Z^{\bullet}(\hlam)$, 
the right term
$\Ext^{\bullet}_{G_{\beta}}
(\mathcal{L}, \mathcal{L})$ denotes the Yoneda algebra 
in the $G_{\beta}$-equivariant derived category $D_{G_{\beta}}(E_{\beta}; \kk)$
of constructible complexes 
on $E_{\beta}$ and  
$\mathcal{L} := \pi_{*}\underline{\kk}$
is the derived push-forward
of the constant sheaf along the proper morphism
$\pi \colon \M^{\bullet}(\hlam) \to \M_{0}^{\bullet}(\hlam) = E_{\beta}$. 
The symbol $(-)^{\wedge}$ stands for the completion 
with respect to the degree.
Here, we make a suitable 
identification $\hR(\hlam) \cong H^{G(\hlam)}_{\bullet}
(\mathrm{pt}, \kk)^{\wedge}$. 
The first isomorphism in (\ref{Eq:ext})
is given by the equivariant Chern character map
and the second one is as in \cite[Section 8.6]{CG97}.
By \cite[Theorem 14.3.2]{Nakajima01}, we have
$$
\mathcal{L} = \bigoplus_{\mathbf{m} \in \KP(\beta)}
L_{\mathbf{m}} \otimes \mathsf{IC}(\mathbb{O}_{\mathbf{m}}),
$$ 
where $L_{\mathbf{m}}$ is a non-zero
finite-dimensional graded vector space and 
$ \mathsf{IC}(\mathbb{O}_{\mathbf{m}})$ 
is the intersection cohomology complex associated with
the constant sheaf on the orbit 
$\mathbb{O}_{\mathbf{m}}$ for each $\mathbf{m} \in \KP(\beta)$.
Therefore, the Yoneda algebra 
$\Ext^{\bullet}_{G_{\beta}}(\mathcal{L}, \mathcal{L})$
is Morita equivalent to 
the quiver Hecke algebra $H_{Q}(\beta)$
(cf.~Definition \ref{Def:KLR})
by \cite{VV11} 
and the completed algebra 
$\hK^{\G(\hlam)}(Z^{\bullet}(\hlam))$ is 
affine quasi-hereditary (cf.~Theorem~\ref{Thm:KBKM}).  

\end{Rem}


\section{Application to Kang-Kashiwara-Kim's functor}
\label{Sec:KKK}

In this section, we apply the results obtained so far to 
prove that
Kang-Kashiwara-Kim\rq{}s generalized quantum affine 
Schur-Weyl duality functor gives an equivalence
of categories assuming the simpleness of some poles of normalized $R$-matrices 
between $\ell$-fundamental modules.


\subsection{Quiver Hecke algebra}
\label{Ssec:KLR}

In this subsection, we recall the definition of the quiver Hecke algebra
and the affine highest weight structure
of its module category.

We keep the notation in the previous sections.
In particular, we fix a Dynkin quiver
$Q = (I, \Omega)$.  
Let $\beta = \sum_{i \in I}
d_{i} \alpha_{i}
\in \cQ^{+}$ with $\mathsf{ht}\, \beta 
:= \sum_{i \in I} d_{i}
= d$.
Set 
$$
I^{\beta} := \{ 
\mathbf{i} = (i_{1}, \ldots, i_{d}) \in I^{d} \mid
\alpha_{i_{1}} + \cdots + \alpha_{i_{d}} = \beta 
\}.
$$
Note that the symmetric group $\SG_{d}$
of degree $d$ acts on the set $I^{\beta}$ by
$\sigma \cdot \mathbf{i} := (i_{\sigma^{-1}(1)}, \ldots
, i_{\sigma^{-1}(d)})$
for $\mathbf{i} = (i_{1}, \ldots, i_{d}) \in I^{\beta}$ and
$\sigma \in \SG_{d}.$ For each $k \in \{1, \ldots, d-1 \}$, 
we denote the
transposition of $k$ and $k+1$ by $\sigma_k \in \SG_d$.

\begin{Def}[Khovanov-Lauda~\cite{KL09}, Rouquier~\cite{Rouquier08}]
\label{Def:KLR}
The quiver Hecke algebra $H_{Q}(\beta)$
is defined to be a $\kk$-algebra with the generators:
$$
\{ e( \mathbf{i} ) \mid \mathbf{i} \in I^{\beta}\}
\cup 
\{x_{1}, \ldots, x_{d} \}
\cup
\{\tau_{1}, \ldots, \tau_{d-1} \},
$$  
satisfying the following relations:
$$
e(\mathbf{i}) e(\mathbf{i}^{\prime}) = 
\delta_{\mathbf{i}, \mathbf{i}^{\prime}} e(\mathbf{i}), \quad
\sum_{\mathbf{i} \in I^{\beta}} e(\mathbf{i}) = 1,
\quad
x_{k} x_{l} = x_{l} x_{k}, \quad
x_k e(\mathbf{i}) = e(\mathbf{i}) x_{k},  
$$
$$
\tau_{k} e(\mathbf{i}) = e(\sigma_k \cdot \mathbf{i}) \tau_k,
\quad
\tau_k \tau_l = \tau_l \tau_k \quad \text{if $|k-l| > 1$},
$$
$$
\tau_{k}^{2} e(\mathbf{i}) = 
\begin{cases}
(x_{k} - x_{k+1})e(\mathbf{i}), 
& \text{if $i_{k} \leftarrow i_{k+1} \in \Omega$},\\
(x_{k+1} - x_{k})e(\mathbf{i}), 
& \text{if $i_{k} \to i_{k+1} \in \Omega$},\\ 
e(\mathbf{i}) & \text{if $a_{i_{k}, i_{k+1}} = 0$},\\
0
& \text{if $i_{k}=i_{k+1}$},
\end{cases}
$$
$$
(\tau_k x_l - x_{s_{k}(l)} \tau_k)e(\mathbf{i})
= \begin{cases}
- e(\mathbf{i}) & \text{if $l=k, i_{k} = i_{k+1}$}, \\
e(\mathbf{i}) & \text{if $l=k+1, i_{k} = i_{k+1}$}, \\
0 & \text{otherwise},
\end{cases}  
$$
$$
(\tau_{k+1} \tau_{k} \tau_{k+1} - \tau_{k} \tau_{k+1} \tau_{k}) 
e(\mathbf{i})
= 
\begin{cases}
e(\mathbf{i})
& \text{if $i_{k} = i_{k+2}, (i_{k} \leftarrow i_{k+1}) \in \Omega$}, \\
-e(\mathbf{i}) 
& \text{if $i_{k} = i_{k+2}, (i_{k} \to i_{k+1}) \in \Omega$}, \\
0 &
\text{otherwise}.
\end{cases}  
$$
\end{Def}

Let
$\mathbb{P}_{\beta} := \bigoplus_{\mathbf{i} \in I^{\beta}}
\kk[x_{1}, \ldots , x_{d}] e(\mathbf{i})
$
with a commutative $\kk[x_{1}, \ldots, x_{d}]$-algebra structure 
$e(\mathbf{i}) e(\mathbf{i}^{\prime}) = 
\delta_{\mathbf{i}, \mathbf{i}^{\prime}} e(\mathbf{i})$. 
The symmetric group $\SG_{d}$ acts on 
$\mathbb{P}_{\beta}$ by
$$
\sigma(f(x_{1}, \ldots, x_{d})e(\mathbf{i}))
:= f(x_{\sigma(1)}, \ldots, x_{\sigma(d)} ) 
e(\sigma \cdot \mathbf{i})$$
for $f(x_{1}, \ldots, x_{d}) \in \kk[x_{1}, \ldots, x_{d}]$
and $\sigma \in \SG_{d}$.
The $\SG_{d}$-invariant part 
$(\mathbb{P}_{\beta})^{\SG_{d}}$
is identified with the algebra
$$
\mathbb{S}_{\beta} := \bigotimes_{i \in I} 
(\kk[x_{i}]^{\otimes d_{i}})^{\SG_{d_{i}}}
= \bigotimes_{i \in I} \kk[x_{i, 1}, \ldots, x_{i, d_{i}}]^{\SG_{d_{i}}},
$$
via the isomorphism 
$\mathbb{S}_{\beta} \xrightarrow{\cong} 
(\mathbb{P}_{\beta})^{\SG_{d}}$ given by the formula
\begin{multline}
\label{Eq:Hecke_center}
f_{1}(x_{1,1}, \ldots, x_{1,d_{i}}) \otimes
\cdots \otimes f_{n}(x_{n,1}, \ldots, x_{n, d_{n}}) \\
\mapsto 
\frac{1}{d_{1}! \cdots d_{n}!} 
\sum_{\sigma \in \SG_{d}}
\sigma \left(
f_{1}(x_{1,1}, \ldots, x_{1, d_{1}}) \cdots 
f_{n}(x_{n, 1}, \ldots, x_{n, d_{n}}) e(1^{d_{1}}, \ldots, n^{d_{n}})
\right),
\end{multline}
where we put 
$x_{i,j} := x_{d_{1} + \cdots + d_{i-1} + j}$ 
on the RHS 
for $i \in \{1, \ldots ,n \}, j \in \{1, \ldots, d_{i}\}$. 

For each $\sigma \in \SG_{d}$, we 
fix a reduced expression 
$\sigma = \sigma_{k_1} \cdots \sigma_{k_p}$.
Then we define 
$\tau_{\sigma} := \tau_{k_1} \cdots \tau_{k_p} \in H_{Q}(\beta)$.
Note that this element $\tau_{\sigma}$ depends on 
the choice of a reduced expression of $\sigma$ in general
because $\tau_k$'s do not satisfy the braid relations.

\begin{Thm}[Khovanov-Lauda~\cite{KL09}]
\label{Thm:KL}
The followings hold.
\begin{enumerate}
\item \label{Thm:KL:PBW} 
The quiver Hecke algebra $H_{Q}(\beta)$ is a
left (or right) free module over the commutative 
subalgebra $\mathbb{P}_{\beta}$ with
a free basis $\{ \tau_{\sigma} \mid \sigma \in \SG_{d} \}$;
\item \label{Thm:KL:center}
The center of $H_{Q}(\beta)$ coincides with the subalgebra
$(\mathbb{P}_{\beta})^{\SG_{d}}$
of $\SG_{d}$-invariant polynomials,
which is identified with $\mathbb{S}_{\beta}$. 
In particular, $H_{Q}(\beta)$ is a free module over its center $\mathbb{S}_{\beta}$
of rank $(d!)^{2}$.
\end{enumerate}
\end{Thm}
\begin{proof}
See \cite[Proposition 2.7]{KL09} for (\ref{Thm:KL:PBW})
and \cite[Theorem 2.9]{KL09} for (\ref{Thm:KL:center}).
\end{proof}

The quiver Hecke algebra $H_{Q}(\beta)$
is equipped with a $\Z$-grading given by
$
\mathrm{deg} \, e(\mathbf{i}) = 0, \;
\mathrm{deg} \, x_{k} = 2, \;
\mathrm{deg} \, \tau_{k}e(\mathbf{i}) = -a_{i_k, i_{k+1}}.
$
Let $\hH_{Q}(\beta)$ (resp.~$\widehat{\mathbb{P}}_{\beta}$)
be the completion 
of the algebra $H_{Q}(\beta)$ (resp.~$\mathbb{P}_{\beta}$)
with respect to the grading.
More explicitly, by Theorem~\ref{Thm:KL} (\ref{Thm:KL:PBW}), we have 
$$
\hH_{Q}(\beta) = 
\bigoplus_{\sigma \in \SG_{d}} \widehat{\mathbb{P}}_{\beta}
\tau_{\sigma}
= 
\bigoplus_{\sigma \in \SG_{d}} 
\tau_{\sigma} \widehat{\mathbb{P}}_{\beta},
$$
where
$\widehat{\mathbb{P}}_{\beta} 
:=  \bigoplus_{\mathbf{i}} 
\kk[\![ x_{1}, \ldots, x_{d} ]\!]e(\mathbf{i}).$
The center of $\hH_{Q}(\beta)$ is 
the $\SG_{d}$-invariant part  
$(\widehat{\mathbb{P}}_{\beta})^{\SG_{d}}$
of $\widehat{\mathbb{P}}_{\beta}$,
which is identified with
the completion $\hS_{\beta}$
of $\mathbb{S}_{\beta}$
at the maximal ideal $0$ 
via the isomorphism 
given by the formula 
(\ref{Eq:Hecke_center}). 

\begin{Def}
We define the category
$\hMm_{Q, \beta}$ to be the category 
of finitely generated $\hH_{Q}(\beta)$-modules, i.e., 
$$\hMm_{Q, \beta} := \hH_{Q}(\beta) \modfg.$$ 
Let $\Mm_{Q, \beta}$ denote
the full subcategory of
finite-dimensional modules in $\hMm_{Q, \beta}$, i.e.,
$$
\Mm_{Q, \beta} := \hH_{Q}(\beta) \modfd. 
$$
\end{Def}

\begin{Rem}
By Theorem~\ref{Thm:KL} (\ref{Thm:KL:center}), 
we see that the completion $\hH_{Q}(\beta)$  
is naturally isomorphic to 
the central completion 
$H_{Q}(\beta) \otimes_{\mathbb{S}_{\beta}}
\hS_{\beta}$ along
the trivial central character.
Therefore the category $\Mm_{Q, \beta}$
is the same as the category $H_{Q}(\beta) \modfd^{0}$
of finite-dimensional modules on which the elements
$x_{k}$\rq{}s act nilpotently.  
Also note that we have the forgetful functor
$H_{Q}(\beta)\text{-gmod}_{\mathrm{fd}} \to \Mm_{Q, \beta}$,
where $H_{Q}(\beta)\text{-gmod}_{\mathrm{fd}}$ denotes
the category of finite-dimensional graded 
$H_{Q}(\beta)$-modules with homogeneous morphisms. 
\end{Rem}

For $\beta, \beta^{\prime} \in \cQ^{+}$ with
$\mathsf{ht} \, \beta = d, \; 
\mathsf{ht} \, \beta^{\prime} = d^{\prime}$, 
we have an embedding
$$
H_{Q}(\beta) \otimes H_{Q}(\beta^{\prime})
\hookrightarrow H_{Q}(\beta + \beta^{\prime})$$ 
given by
$e(\mathbf{i}) \otimes e(\mathbf{i}^{\prime})
\mapsto e(\mathbf{i} \circ \mathbf{i}^{\prime}), \;
x_{k}\otimes 1 \mapsto x_{k}, \;
\tau_{k} \otimes 1 \mapsto \tau_{k}, \;
1 \otimes x_{k} \mapsto x_{d + k}, \;
1 \otimes \tau_{k} \mapsto \tau_{d + k},$
where $\mathbf{i} \circ \mathbf{i}^{\prime}
:= (i_{1}, \ldots, i_{d}, 
i^{\prime}_{1}, \ldots, i^{\prime}_{d^{\prime}})
\in I^{\beta_{1} + \beta_{2}}.$ 
This embedding naturally 
induces its completed version
$$\hH_{Q}(\beta) \otimes 
\hH_{Q}(\beta^{\prime}) \hookrightarrow 
\hH_{Q}(\beta + \beta^{\prime}),$$
which equips the direct sum category
$\hMm_{Q} := \bigoplus_{\beta \in \cQ^{+}} \hMm_{Q, \beta}$
with a structure of monoidal category whose
tensor product 
$(-) \circ (-) \colon \hMm_{Q, \beta} \times \hMm_{Q, \beta^{\prime}}
\to \hMm_{Q, \beta + \beta^{\prime}}$ is given by 
$$M \circ M^{\prime} := \hH_{Q}(\beta + \beta^{\prime})
\otimes_{\hH_{Q}(\beta) \otimes 
\hH_{Q}(\beta^{\prime})} (M \otimes M^{\prime}).$$
The subcategory
$\Mm_{Q} := \bigoplus_{\beta \in \cQ^{+}} \Mm_{Q, \beta}$
of finite-dimensional modules is closed under 
this monoidal structure.

Write $\beta = \sum_{i \in I} d_{i} \alpha_{i},
\beta^{\prime} = \sum_{i \in I} d^{\prime}_{i} \alpha_{i}$.
Then we have the following injective homomorphism 
$$
\mathbb{S}_{\beta + \beta^{\prime}}
= \bigotimes_{i \in I} 
(\kk[x_{i}]^{\otimes d_{i} + d_{i}^{\prime}})^{\SG_{
d_{i} + d_{i}^{\prime}}}
\hookrightarrow 
\bigotimes_{i \in I} 
(\kk[x_{i}]^{\otimes d_{i}} \otimes 
\kk[x_{i}]^{\otimes d_{i}^{\prime}})^{\SG_{
d_{i}} \times \SG_{d_{i}^{\prime}}}
\cong
\mathbb{S}_{\beta} \otimes \mathbb{S}_{\beta^{\prime}},
$$
which induces an injective homomorphism
$
\hS_{\beta + \beta^{\prime}}
\hookrightarrow
\hS_{\beta} \hat{\otimes} \hS_{\beta^{\prime}}
$ for the completions.
We refer to this kind of injective homomorphism
as the standard inclusion.
Note that for any $M \in \hMm_{Q, \beta}$
and $M^{\prime} \in \hMm_{Q, \beta^{\prime}}$,
actions of the centers
$\hS_{\beta}$ and $
\hS_{\beta^{\prime}}$
induce the natural homomorphism
$\hS_{\beta} \hat{\otimes} \hS_{\beta^{\prime}}
\to \End_{\hH_{Q}(\beta + \beta^{\prime})}
(M \circ M^{\prime}).$
The following lemma is easily deduced from the definition.
\begin{Lem}
\label{Lem:center_incl}
For any $M \in \hMm_{Q, \beta}$
and $M^{\prime} \in \hMm_{Q, \beta^{\prime}}$,
the action of the center 
$\hS_{\beta + \beta^{\prime}}$ on the module
$M \circ M^{\prime}$
is given via the standard inclusion 
$\hS_{\beta + \beta^{\prime}}
\hookrightarrow
\hS_{\beta} \hat{\otimes} \hS_{\beta^{\prime}}.$
\end{Lem}

We define
the affinization $M_{\mathrm{af}}$ of 
an $H_{Q}(\beta)$-module $M$ 
following \cite[Section 1.3]{KKK18}.
Let $y$ be an indeterminate.
We set $M_{\mathrm{af}} := \kk[\![ y ]\!] \otimes M$
and define an action of $H_{Q}(\beta)$ by
$
e(\mathbf{i})(a \otimes m) := a \otimes (e(\mathbf{i})m), 
x_{k}(a \otimes m) := (y a) \otimes m + a \otimes (x_{k} m),
\tau_{k}(a \otimes m) := a \otimes (\tau_{k} m), 
$
where $a \in \kk[\![ y ]\!], m \in M$.
Note that if $M \in \Mm_{Q, \beta}$, 
the completion $\hH_{Q}(\beta)$ naturally acts on
the affinization $M_{\mathrm{af}}$ and 
we have $M_{\mathrm{af}} \in \hMm_{Q, \beta}$.

Recall from Subsection \ref{Ssec:identify} that
the set $\KP(\beta)$ of Kostant partitions
of $\beta$ is equipped with a partial order $\le$,
which is the opposite of the orbit closure ordering
in the $G_\beta$-space $E_{\beta}$.
Then we have the following important result. 

\begin{Thm}
[Kato~\cite{Kato14}, Brundan-Kleshchev-McNamara~\cite{BKM14}]
\label{Thm:KBKM}
The category $\hMm_{Q, \beta}$ has a structure of 
affine highest weight category for $(\KP(\beta), \le)$.
\end{Thm}
\begin{proof}
Actually, \cite{Kato14} and \cite{BKM14}
proved that the finitely generated graded module category
of $H_{Q}(\beta)$ satisfies the
conditions of affine highest weight category
in the setting of graded algebras, 
which Kleshchev originally concerned in
his paper \cite{Kleshchev15}.
We can translate the results to the completed version 
$\hH_{Q}(\beta)$ by applying \cite[Theorem 4.6]{Fujita18}, which is available now
thanks to Theorem~\ref{Thm:KL} (\ref{Thm:KL:center}).

The orderings of the set $\KP(\beta)$
used in \cite{Kato14} and \cite{BKM14} 
are stronger than ours. However, our ordering
is enough to control the affine highest weight structure
of $\hMm_{Q, \beta}$.
See \cite[Remark 4.1]{Kato14} and \cite{Kato17}, 
or \cite[Theorem 6.12]{McNamara17}.    
\end{proof}

In particular, Theorem~\ref{Thm:KBKM} tells us
that simple modules of the category 
$\hMm_{Q, \beta}$
are labelled by Kostant partitions of $\beta$.
Let $S(\mathbf{m})$ denote the simple module corresponding
to $\mathbf{m} \in \KP(\beta)$. 
To simplify the notation,
we identify a positive root $\alpha$
with the Kostant partition 
$\mathbf{m}_{\alpha} := 
(\delta_{\alpha, \alpha^{\prime}})_{\alpha^{\prime} 
\in \mathsf{R}^{+}}
\in \KP(\alpha)$
consisting of the single root $\alpha$.
For instance, we write $S(\alpha)$
instead of $S(\mathbf{m}_{\alpha})$. 
By convention, our $S(\alpha)$ is the same as 
``the dual root module of weight $\alpha$'' in 
\cite[Section 4.3]{KKK15}
(see also \cite[Remark 4.3.3(b)]{KKK15}
for a comparison with \cite{BKM14}).

According to \cite{BKM14, Kato14},
we can construct 
the standard module $\std(\mathbf{m})$,
the proper standard module $\pstd(\mathbf{m})$ and
the proper costandard module $\pcstd(\mathbf{m})$
in the affine highest weight category  $\hMm_{Q, \beta}$ for $(\KP(\beta), \le)$
from the dual root modules
$\{ S(\alpha) \mid \alpha \in \mathsf{R}^{+}\}$ as in the next paragraph.
Although the 
construction a priori depends on the choice
of a reduced expression
$w_{0} = s_{i_{1}} \cdots s_{i_{r}}$ 
of the longest element $w_{0} \in W$
adapted to the quiver $Q$,
the resulting modules 
$\std(\mathbf{m}), \pstd(\mathbf{m}), \pcstd(\mathbf{m})$
do not depend on the choice
up to isomorphism. See \cite[Remark 4.1]{Kato14}. 

Fix a reduced expression 
$w_{0} = s_{i_{1}} \cdots s_{i_{r}}$ 
adapted to $Q$.
It yields a total ordering
$\mathsf{R}^{+} = \{ \gamma_{1}, \ldots, \gamma_{r}\}$
among the positive roots, where
$\gamma_{k} := s_{i_1} \cdots s_{i_{k-1}} (\alpha_{i_k})$.
Using this total ordering, we write
$
\KP(\beta) = \{ \mathbf{m} = 
(m_{1}, \ldots, m_{r}) \in \Z_{\ge 0}^{r} \mid 
\sum_{i=1}^{r} m_{i} \gamma_{i} = \beta \}.$ 
Then for each $\mathbf{m} 
= (m_{1}, \ldots, m_{r}) \in \KP(\beta)$, 
we have
\begin{align*}
\pstd(\mathbf{m}) &= S(\gamma_{1})^{\circ m_{1}} 
\circ S(\gamma_{2})^{\circ m_{2}} \circ \cdots
\circ S(\gamma_{r})^{\circ m_{r}}, \\
\pcstd(\mathbf{m}) &= S(\gamma_{r})^{\circ m_{r}} 
\circ S(\gamma_{r-1})^{\circ m_{r-1}} \circ \cdots
\circ S(\gamma_{1})^{\circ m_{1}}.
\end{align*}

The standard module $\std(\mathbf{m})$ is obtained as
an indecomposable direct summand 
of 
the product 
$S(\gamma_{1})_{\mathrm{af}}^{\circ m_{1}} 
\circ S(\gamma_{2})_{\mathrm{af}}^{\circ m_{2}} \circ \cdots
\circ S(\gamma_{r})_{\mathrm{af}}^{\circ m_{r}}.$
Here, all the summands are mutually isomorphic.
Namely, we have
\begin{equation}
\label{Eq:stdH}
\std(\mathbf{m})^{\oplus m_{1}! \cdots m_{r}!}
\cong S(\gamma_{1})_{\mathrm{af}}^{\circ m_{1}} 
\circ S(\gamma_{2})_{\mathrm{af}}^{\circ m_{2}} \circ \cdots
\circ S(\gamma_{r})_{\mathrm{af}}^{\circ m_{r}}.
\end{equation}
In particular, 
we have $\pstd(\alpha) = \pcstd(\alpha) = S(\alpha)$
and $\std(\alpha) = S(\alpha)_{\mathrm{af}}$ for
each $\alpha \in \mathsf{R}^{+}$.

\subsection{Kang-Kashiwara-Kim's functor}
\label{Ssec:KKK}
Recall the bijection $\phi \colon \mathsf{R}^+ \to \widehat{I}_Q$ from Subsection~\ref{Ssec:quivers}.
For each $i \in I$, we set $\mathbb{V}_{i} 
:= \bW(\mathsf{cl}(\varpi_{\phi(\alpha_{i})}) )$.

Fix $\beta = \sum_{i \in I} d_i \alpha_i \in \cQ^+$.  
For each $\mathbf{i} = (i_1, \ldots, i_d) \in I^{\beta}$, 
we write
$$\mathcal{O}_{\mathbf{i}} := \bigotimes_{k=1}^{d}
\End_{\tilU}(\mathbb{V}_{i_k}) \cong 
\kk[z_{1}^{\pm 1}, \ldots, z_{d}^{\pm 1}],$$
where
$\End_{\tilU}(\mathbb{V}_{i_k}) = \kk[z_{k}^{\pm 1}]$ as in Theorem~\ref{Thm:globalWeyl} 
(\ref{Thm:globalWeyl:End}).
Using $\mathsf{p} = \mathrm{pr}_2 \circ \phi$, we set
$$
\widehat{V}_{\mathbf{i}}
 :=
\kk[\![ z_{1} - q^{\mathsf{p}(\alpha_{i_1})}, \ldots, 
z_{d} - q^{\mathsf{p}(\alpha_{i_d})}]\!]
\otimes_{\mathcal{O}_{\mathbf{i}}}
\left(\mathbb{V}_{i_{1}} \otimes \cdots 
\otimes \mathbb{V}_{i_{d}}
\right).
$$
Take the direct sum over $\mathbf{i} \in I^\beta$, we define
$$
\widehat{V}^{\otimes \beta} := \bigoplus_{\mathbf{i} \in I^{\beta}}
\widehat{V}_{\mathbf{i}}.
$$
Let $e(\mathbf{i}) \colon \widehat{V}^{\otimes \beta} \to \widehat{V}_{\mathbf{i}}$
denote the projection to the $\mathbf{i}$-th component.
By construction, the algebra
$$
\widehat{\mathcal{O}}_{\beta}
:= \bigoplus_{\mathbf{i} \in I^{\beta}}
\kk[\![ z_{1} - q^{\mathsf{p}(\alpha_{i_1})}, \ldots, 
z_{d} - q^{\mathsf{p}(\alpha_{i_d})}]\!]
e(\mathbf{i})
$$
naturally acts on $\widehat{V}^{\otimes \beta}$.

\begin{Thm}[Kang-Kashiwara-Kim~\cite{KKK18, KKK15}]
\label{Thm:KKK}
Assume the following condition:
\begin{itemize}
\item[$(\star)$]
For any $(i, p) , (j, r) \in \{ \phi(\alpha_{i}) \mid i \in I \}
\subset \widehat{I}_{Q}$,
the zero order of the denominator $d_{i,j}(u)$
at $u=q^{r-p}$ is at most one.
\end{itemize}
Then there exists a $\kk$-algebra homomorphism
$
\hH_{Q}(\beta)   \to 
\End_{\tilU}(\widehat{V}^{\otimes \beta}
)^{\mathrm{op}} $
such that
$$e(\mathbf{i})  \mapsto e(\mathbf{i}), \quad
x_{k}e(\mathbf{i})  \mapsto (q^{- \mathsf{p}(\alpha_{i_k})} 
z_k - 1) 
e({\mathbf{i}}),$$
and the images of the elements $\tau_k$\rq{}s are
defined  in a suitable way using the normalized $R$-matrices.
In particular, $\widehat{V}^{\otimes \beta}$
has a structure of 
$(\tilU, 
\hH_{Q}(\beta))$-bimodule.
\end{Thm}   

\begin{Rem}
For a quiver $Q$ of type $\mathrm{AD}$,
the condition $(\star)$ is  
always satisfied thanks to \cite[Lemma 3.2.4]{KKK15}. 
Henceforth, we keep assuming the condition
$(\star)$ in Theorem~\ref{Thm:KKK}
until the end of this paper
if our quiver $Q$ is of type $\mathrm{E}$.\footnote{
After the initial submission of this paper, it was also proved that the condition $(\star)$ is always satisfied for 
a quiver $Q$ of type $\mathrm{E}$ by Oh-Scrimshaw~\cite{OS19} and by the author~\cite{Fujita20} independently.}  
\end{Rem}

In the remaining part of this subsection, 
we shall extend the $(\tilU, \hH_{Q}(\beta))$-bimodule $\widehat{V}^{\otimes \beta}$
to a $(\hK^{\G(\hlam)}(Z^{\bullet}(\hlam)), 
\hH_{Q}(\beta))$-bimodule.
As before, we set 
$$\text{$\hlam \equiv \hlam_{\beta} := \sum_{i \in I} d_{i} 
\varpi_{\phi(\alpha_{i})} \in \lP^{+}_{Q, \beta},$ \quad and \quad $\lambda := \mathsf{cl}(\hlam) \in \cP^{+}.$}$$
Recall from Section \ref{Sec:quivvar}
that we have the quiver varieties
$\pi \colon \M(\lambda) \to \M_{0}(\lambda)$
with $\G(\lambda)$-action.
We realize the graded quiver varieties
$\pi^{\bullet} \colon \M^{\bullet}(\hlam) \to \M_{0}^{\bullet}(\hlam)$
with $\G(\hlam)$-action
as the $\T$-fixed point subvarieties,
where $\T \subset \G(\lambda)$ is a certain
$1$-dimensional subtorus and 
$\G(\hlam)$ is regarded 
as the centralizer of $\T$ in $\G(\lambda)$.

We fix a maximal torus $H(\hlam)$ of $G(\hlam)$:
$$H(\hlam) \cong 
(\C^{\times})^{d} = (\C^{\times})^{d_1} \times
\cdots \times (\C^{\times})^{d_{n}}
\subset
GL_{d_{1}}(\C) \times \cdots \times GL_{d_{n}}(\C) 
\cong G(\hlam),
$$
and set $\HH(\hlam) := H(\hlam) \times \C^{\times}
\subset G(\hlam) \times \C^{\times} = \G(\hlam).$
We have the standard identification
$\R(\HH(\hlam)) = \kk[z_{1}^{\pm 1}, \ldots, z_{d}^{\pm 1}]$.
Then, the natural homomorphism 
$\R(\hlam) = \R(\G(\hlam)) \hookrightarrow \R(\HH(\hlam))$
is given by the following composition
$$
\bigotimes_{i \in I} 
\kk[z_{i,1}^{\pm 1}, \ldots, z_{i, d_{i}}^{\pm}]^{\SG_{d_{i}}}
\hookrightarrow
\bigotimes_{i \in I} 
\kk[z_{i,1}^{\pm 1}, \ldots, z_{i, d_{i}}^{\pm}]
\cong
\kk[z_{1}^{\pm 1}, \ldots, z_{d}^{\pm 1}],
$$
where the second isomorphism sends
$z_{i, j}$ to $z_{d_1 + \cdots + d_{i-1} + j}$ for 
each $i \in \{1, \ldots, n\}$ and $j \in \{1, \ldots, d_{i} \}$.

For each $\mathbf{i} = (i_1, \ldots, i_d) \in I^{\beta}$, we 
choose an element $\sigma \in \SG_{d}$
such that 
$\mathbf{i} = \sigma \cdot (1^{d_{1}}, 2^{d_{2}}, 
\ldots, n^{d_{n}})$.
Then we define a group homomorphism
$f_{\sigma} \colon \C^{\times} \to H(\hlam)$
by 
$
f_{\sigma}(t) := (t^{\sigma(1)}, \ldots, t^{\sigma(d)}).
$ 
Following Nakajima~\cite{Nakajima01t}, we consider the following subvariety of $\M(\lambda)$:
\begin{align*}
\tZ(\lambda; \sigma) &:= 
\left\{ x \in \M(\lambda) \; 
\middle| \; \lim_{t \to 0} f_{\sigma}(t) \pi(x) = 0 
\in \M_{0}(\lambda)
\right\}.
\end{align*}
By construction, 
the group $\HH(\hlam)$ acts on  
$\tZ(\lambda ; \sigma)$
and the $\HH(\hlam)$-equivariant $K$-group
$\mathcal{K}^{\HH(\hlam)}(\tZ(\lambda ; \sigma))$ 
becomes a left module over the
convolution algebra $\mathcal{K}^{\HH(\hlam)}(Z(\lambda))$.
Via the homomorphism
$$
\tilU \xrightarrow{\Phi_{\lambda}} 
\mathcal{K}^{\G(\lambda)}(Z(\lambda))
\xrightarrow{\text{forget}}
\mathcal{K}^{\HH(\hlam)}(Z(\lambda)), 
$$
we regard $\mathcal{K}^{\HH(\hlam)}(\tZ(\lambda ; \sigma))$
as a left $\tilU$-module.
  
\begin{Thm}[Nakajima~\cite{Nakajima01t}]
\label{Thm:Nak_tensor}
With the above notation, the followings hold.
\begin{enumerate}
\item
\label{Thm:Nak_tensor:tensor}
As a $\tilU$-module, we have
$
\mathcal{K}^{\HH(\hlam)}(\tZ(\lambda ; \sigma))
\cong \mathbb{V}_{i_{1}} \otimes \cdots 
\otimes \mathbb{V}_{i_{d}};
$
\item
\label{Thm:Nak_tensor:center}
The isomorphism in (\ref{Thm:Nak_tensor:tensor})
induces the following commutative diagram
$$
\xy
\xymatrix{
\End(\mathcal{K}^{\HH(\hlam)}(\tZ(\lambda ; \sigma)))
\ar[r]^-{\cong}
&
\End(\mathbb{V}_{i_{1}} \otimes \cdots 
\otimes \mathbb{V}_{i_{d}})
\\
\R(\HH(\hlam))
\ar[r]^-{\sigma}
\ar[u]
&
\mathcal{O}_{\mathbf{i}} ,
\ar[u]
}
\endxy
$$
where the vertical arrows denote
the actions on the modules and
the bottom arrow
$\sigma \colon 
\R(\HH(\hlam)) = \kk[z_{1}^{\pm 1}, \ldots, z_{d}^{\pm 1}]
\to \kk[z_{1}^{\pm 1}, \ldots, z_{d}^{\pm 1}]
= \mathcal{O}_{\mathbf{i}}$
is given by
$\sigma (f(z_{1}, \ldots, z_{d})) = f(z_{\sigma(1)},
\ldots, z_{\sigma(d)}).$ 
\end{enumerate}
\end{Thm} 
\begin{proof}
See \cite[Theorem 6.12]{Nakajima01t}.
\end{proof}
 
Let $\hR(\HH(\hlam))$ denote
the completion of $\R(\HH(\hlam))$
with respect to the maximal ideal
corresponding to $\T \subset \HH(\hlam)$.
We apply 
the corresponding completion $(-) \otimes_{\R(\HH(\hlam))} \hR(\HH(\hlam))$
to the isomorphism in
Theorem~\ref{Thm:Nak_tensor}
(\ref{Thm:Nak_tensor:tensor}) 
to get 
the following isomorphism:
\begin{equation}
\label{Eq:Vi}
\mathcal{K}^{\HH(\hlam)}
(\tZ(\lambda ; \sigma)) \otimes_{\R(\HH(\hlam))} 
\hR(\HH(\hlam))
\cong 
\widehat{V}_{\mathbf{i}}.
\end{equation}
Restricting to $\hR(\hlam) \subset
\hR(\HH(\hlam))$,
we regard the isomorphism (\ref{Eq:Vi}) 
as an isomorphism of $(\tilU, \hR(\hlam))$-bimodules,
which does not depend on the choice of $\sigma$
such that $\mathbf{i} = \sigma \cdot
(1^{d_{1}}, \ldots, n^{d_{n}}).$
By the construction, this $(\tilU, \hR(\hlam))$-bimodule structure
on the LHS of $(\ref{Eq:Vi})$ comes 
from the convolution action of 
$\hK^{\G(\hlam)}(Z^{\bullet}(\hlam))$,
and hence
$\widehat{V}_{\mathbf{i}} \in \hCc_{Q, \beta}$
via the isomorphism (\ref{Eq:Vi}). 
As a result, we can regard the $(\tilU, 
\hH_{Q}(\beta))$-bimodule $\widehat{V}^{\otimes \beta}$ as a $(\hK^{\G(\hlam)}(Z^{\bullet}(\hlam)), \hH_{Q}(\beta))$-bimodule. 

By Theorem~\ref{Thm:Nak_tensor}
(\ref{Thm:Nak_tensor:center}),
the action of $\hR(\hlam)$
on $\widehat{V}^{\otimes \beta}$ factors through the
homomorphism
$$
\hR(\hlam) = \bigotimes_{i \in I}
\kk [\![ z_{i,1} - q^{\mathsf{p}(\alpha_{i})} , \ldots,
z_{i, d_{i}} - q^{\mathsf{p}(\alpha_{i})} ]\!]^{\SG_{d_{i}}}
\to 
\widehat{\mathcal{O}}_{\beta}, 
$$
which is given by the following formula:
\begin{multline}
\label{Eq:Wbeta_center}
f_{1}(z_{1,1}, \ldots, z_{1,d_{i}}) \otimes
\cdots \otimes f_{n}(z_{n,1}, \ldots, z_{n, d_{n}}) \\
\mapsto 
\frac{1}{d_{1}! \cdots d_{n}!} 
\sum_{\sigma \in \SG_{d}}
\sigma \left(
f_{1}(z_{1,1}, \ldots, z_{1, d_{1}}) \cdots 
f_{n}(z_{n, 1}, \ldots, z_{n, d_{n}}) e(1^{d_{1}}, \ldots, n^{d_{n}})
\right),
\end{multline}
where we put 
$z_{i,j} := z_{d_{1} + \cdots + d_{i-1} + j}$ 
for $i \in \{1, \ldots, n\},$ $j \in \{1, \ldots, d_{i} \}$
on the RHS. 

\begin{Lem}
\label{Lem:center_isom}
We have the following commutative diagram:
$$
\xy
\xymatrix{
\End(\widehat{V}^{\otimes \beta})
\ar@{=}[r]
&
\End(\widehat{V}^{\otimes \beta})
\\
\hS_{\beta} 
\ar[r]^-{\cong}
\ar[u]
&
\hR(\hlam),
\ar[u]
}
\endxy
$$
where
the vertical arrows denote the actions of
$\hS_{\beta} \subset \hH_{Q}(\beta)$
and $\hR(\hlam) \subset 
\hK^{\G(\hlam)}(Z^{\bullet}(\hlam))$
respectively.
The bottom isomorphism is 
given by $x_{i,j}  \mapsto 
q^{-\mathsf{p}(\alpha_{i})} z_{i,j} -1$
for each $i \in I, j \in \{1, \ldots, d_{i}\}$.
\end{Lem}
\begin{proof}The assertion
follows from the construction in
Theorem~\ref{Thm:KKK} and 
the formulas 
(\ref{Eq:Hecke_center}) and (\ref{Eq:Wbeta_center}). 
\end{proof}

\begin{Def}
\label{Def:functor}
Under the assumption ($\star$),
we define the (completed version of) Kang-Kashiwara-Kim functor by
$$
\mathcal{F}_{Q, \beta} 
\colon \hMm_{Q, \beta} \to \hCc_{Q, \beta}, \quad M \mapsto 
\widehat{V}^{\otimes \beta} \otimes_{\hH_{Q}(\beta)} M.
$$
\end{Def}

\begin{Thm}[Kang-Kashiwara-Kim~\cite{KKK18, KKK15}]
\label{Thm:KKK2}
The Kang-Kashiwara-Kim functor 
$\mathcal{F}_{Q, \beta}$ 
satisfies the following properties.
\begin{enumerate}
\item \label{Thm:KKK2:exact}
The bimodule $\widehat{V}^{\otimes \beta}$ is flat over $\hH_Q(\beta)$ and hence the functor $\mathcal{F}_{Q, \beta}$ is exact;
\item \label{Thm:KKK2:monoidal}
For any 
$M \in \hMm_{Q, \beta}$ and 
$M^{\prime} \in \hMm_{Q, \beta^{\prime}}$,
there is a natural isomorphism 
$$
\mathcal{F}_{Q, \beta+\beta^{\prime}}(M \circ M^{\prime})
\cong \mathcal{F}_{Q, \beta}(M) \hat{\otimes} 
\mathcal{F}_{Q, \beta^{\prime}}(M^{\prime})
$$
of $(\tilU, 
\hR(\hlam_{\beta+\beta^{\prime}}))$-bimodules,
where the action of 
$\hR(\hlam_{\beta+\beta^{\prime}})$
on the RHS is given via the standard inclusion
$\hR(\hlam_{\beta+\beta^{\prime}})=
\hR(\hlam_{\beta} + \hlam_{\beta^{\prime}})
\hookrightarrow
\hR(\hlam_{\beta}) \hat{\otimes}\hR(\hlam_{\beta^{\prime}});$
\item \label{Thm:KKK2:center}
For any $M \in \hMm_{Q, \beta}$, 
we have the following commutative diagram: 
$$
\xy
\xymatrix{
\End_{\hMm_{Q, \beta}}(M)
\ar[r]^-{\mathcal{F}_{Q, \beta}}
&
\End_{\hCc_{Q, \beta}}( \mathcal{F}_{Q, \beta}(M))
\\
\hS_{\beta} 
\ar[r]^-{\cong}
\ar[u]
&
\hR(\hlam),
\ar[u]
}
\endxy
$$
where the vertical arrows denote the actions
of the centers.
The bottom arrow is the isomorphism in Lemma~\ref{Lem:center_isom};
\item \label{Thm:KKK2:simple}
For any $\alpha \in \mathsf{R}^{+}$, we have
$\mathcal{F}_{Q, \alpha} (S(\alpha)) \cong L(\varpi_{\phi(\alpha)})$.
\end{enumerate}
\end{Thm}
\begin{proof}
(\ref{Thm:KKK2:exact}) is \cite[Theorem 3.3.3]{KKK18}.
(\ref{Thm:KKK2:monoidal}) follows from
\cite[Theorem 3.2.1]{KKK18} together with Lemma 
\ref{Lem:center_incl} and Lemma~\ref{Lem:center_isom}.
(\ref{Thm:KKK2:center}) is a direct consequence 
of Lemma~\ref{Lem:center_isom}.
(\ref{Thm:KKK2:simple}) is \cite[Theorem 4.3.4]{KKK15}.
\end{proof}

\subsection{Comparison of affine highest weight structures}
\label{Ssec:final}

In this subsection, we prove the following.

\begin{Thm}
\label{Thm:equiv}
Let $\beta \in \cQ^{+}$ be an element.
Under the assumption $(\star)$,
the Kang-Kashiwara-Kim functor gives an equivalence of 
affine highest weight categories
$$\mathcal{F}_{Q, \beta} \colon \hMm_{Q, \beta} \xrightarrow{\simeq}
\hCc_{Q, \beta}.$$
\end{Thm}

We keep the assumption $(\star)$ until the end of this paper.

\begin{Lem}
\label{Lem:root_std}
For any $\alpha \in \mathsf{R}^{+}$, we have
$\mathcal{F}_{Q, \alpha} (\std(\alpha)) \cong 
\hW(\varpi_{\phi(\alpha)}).$
\end{Lem}

\begin{proof}
Let $\mathfrak{m}_{\alpha} \subset \hS_{\alpha}$
be the maximal ideal of the center 
$\hS_{\alpha} \subset \hH_{Q}(\alpha)$.
By the definition of the affinitzation 
$\std(\alpha) = S(\alpha)_{\mathrm{af}}$,
we have
$\std(\alpha) / \mathfrak{m}_{\alpha} \cong S(\alpha).$
From Theorem~\ref{Thm:KKK2}
(\ref{Thm:KKK2:center}) and (\ref{Thm:KKK2:simple}),
we see 
$$
\mathcal{F}_{Q, \alpha}(\std(\alpha)) / \rr_{\hlam_{\alpha}}
\cong \mathcal{F}_{Q, \alpha}(\std(\alpha)/ \mathfrak{m}_{\alpha})
\cong \mathcal{F}_{Q, \alpha}(S(\alpha))
\cong L(\varpi_{\phi(\alpha)}).
$$
In particular, the module 
$\mathcal{F}_{Q, \alpha}(\std(\alpha))$ 
has a simple head isomorphic to $L(\varpi_{\phi(\alpha)})$.
Let us consider the Serre subcategory
$\hCc(\alpha)$ of $\hCc_{Q, \alpha}$
generated by the standard module $\hW(\varpi_{\phi(\alpha)})$.
Then $\mathcal{F}_{Q, \alpha}(\std(\alpha)) \in
\hCc(\alpha).$
By a standard theory of affine highest weight category
(see \cite[Proposition 5.16]{Kleshchev15}),
this category $\hCc(\alpha)$ is equivalent to
the category of finitely generated 
$\hR(\varpi_{\phi(\alpha)})$-modules
under the functor 
$\Hom_{\hCc_{Q, \alpha}}(\hW(\varpi_{\phi(\alpha)}), -).$
Since $\hW(\varpi_{\phi(\alpha)})$ corresponds to the rank $1$ free
$\hR(\varpi_{\phi(\alpha)})$-module under this equivalence,
it is a projective cover
of $L(\varpi_{\phi(\alpha)})$ in $\hCc(\alpha)$.
Therefore we have a surjective homomorphism 
$\hW(\varpi_{\phi(\alpha)}) \twoheadrightarrow 
\mathcal{F}_{Q, \alpha}(\std(\alpha))$. 
Note that any non-trivial quotient of 
$\hW(\varpi_{\phi(\alpha)})$ in $\hCc(\alpha)$ is finite-dimensional
because $\hR(\varpi_{\phi(\alpha)}) \cong 
\kk[\![ x ]\!]$.
On the other hand, the module 
$\mathcal{F}_{Q, \alpha}(\std(\alpha))$ 
is infinite-dimensional thanks to Theorem~\ref{Thm:KKK2}
 (\ref{Thm:KKK2:exact}) and (\ref{Thm:KKK2:simple}).
Thus the surjective homomorphism 
$\hW(\varpi_{\phi(\alpha)}) \twoheadrightarrow 
\mathcal{F}_{Q, \alpha}(\std(\alpha))$
should be an isomorphism.   
\end{proof}

Recall that we have defined  
the order-preserving bijection $f \colon \KP(\beta) \to \lP^{+}_{Q, \beta}$
in Subsection \ref{Ssec:identify}
by $f(\mathbf{m}) := \sum_{\alpha \in \mathsf{R}^{+}}
m_{\alpha} \varpi_{\phi(\alpha)}$
for $\mathbf{m} = (m_{\alpha})_{\alpha \in \mathsf{R}^{+}}
\in \KP(\beta)$.

\begin{Prop}
\label{Prop:check}
For each $\mathbf{m} \in \KP(\beta)$, 
we have 
$$
\mathcal{F}_{Q, \beta}(\std(\mathbf{m}))
 \cong \hW(f(\mathbf{m})),
\quad
\mathcal{F}_{Q, \beta} 
( \pcstd(\mathbf{m}))\cong 
W^{\vee}(f(\mathbf{m})).
$$
\end{Prop}

\begin{proof}
Fix a reduced expression $w_{0} = s_{j_{1}} \cdots s_{j_{r}}$
of the longest element $w_{0}$ adapted to $Q$.
We set $\gamma_{k} := s_{j_1} \cdots s_{j_{k-1}} (\alpha_{j_k})$
and write 
$(i_k, p_k) := \phi(\gamma_{k}) \in \widehat{I}_{Q}$
for each $k \in \{1, \ldots, r\}$.
Then we can prove that
$d_{i_j, i_k}(q^{p_k - p_j}) \neq 0$ holds
for any $1 \le j < k \le r.$
Indeed, if we suppose 
$d_{i_j, i_k}(q^{p_k - p_j}) = 0$
for some $1 \le j < k \le r$,
there is an oriented path
from $(i_j, p_j)$ to $(i_k, p_k)$ 
in the Auslander-Reiten quiver $\Gamma_{Q}$
by Lemma~\ref{Lem:path} (\ref{Lem:path:AR}).
However, then \cite[Theorem 2.17]{Bedard99} implies
that $j > k$, which is a contradiction.

Therefore, we have
\begin{align*}
\hW(f(\mathbf{m}))^{\oplus m_{1}! \cdots m_{r}!} 
&\cong 
\hW(\varpi_{i_1, p_1})^{\hat{\otimes} m_{1}}
\hat{\otimes} \cdots \hat{\otimes}
\hW(\varpi_{i_r, p_r})^{\hat{\otimes}m_r} \\
&\cong
\mathcal{F}_{Q, \beta}(
\std(\gamma_{1})^{\circ m_1} \circ \cdots \circ
\std(\gamma_{r})^{\circ m_r}) \\
& \cong
\mathcal{F}_{Q, \beta}(\std(\mathbf{m}))^{\oplus m_{1} ! \cdots 
m_{r}!},
\end{align*}
where the first isomorphism is Corollary \ref{Cor:hW_tensor},
the second one is Theorem~\ref{Thm:KKK2} 
(\ref{Thm:KKK2:monoidal}) and 
Lemma~\ref{Lem:root_std}, 
the last one is due to (\ref{Eq:stdH}).
This isomorphism is 
an isomorphism of $(U_{q}, \hR(\hlam))$-bimodules.
In fact, it is easy to check 
from Theorem~\ref{Thm:KKK2} (\ref{Thm:KKK2:center})
that 
the action of $\hR(\hlam)$ on the RHS is also 
given via the homomorphism $\theta_{f(\mathbf{m})}$,
as well as on the LHS.
Thus we obtain 
$\mathcal{F}_{Q, \beta}(\std(\mathbf{m}))
 \cong \hW(f(\mathbf{m}))$.
A proof of 
$\mathcal{F}_{Q, \beta} 
( \pcstd(\mathbf{m}))\cong 
W^{\vee}(f(\mathbf{m}))$ is similar and easier.
\end{proof}

\begin{proof}[Proof of Theorem~\ref{Thm:equiv}]
Both of the affine quasi-hereditary algebras 
$\hK^{\G(\hlam)}(Z^{\bullet}(\hlam))$ and $\hH_{Q}(\beta)$
are finite over their centers by Corollary~\ref{Cor:fg} and Theorem~\ref{Thm:KL} respectively.
Thanks to Proposition~\ref{Prop:check},
we can apply Theorem~\ref{Thm:criterion}
to the functor $\mathcal{F}_{Q, \beta} \colon 
\hMm_{Q, \beta} \to \hCc_{Q, \beta}$, which proves the assertion.
\end{proof}

Finally, we consider the direct sum of the functors 
$\mathcal{F}_{Q, \beta}$
over $\beta \in \cQ^{+}$ and restrict 
it to the full subcategory of finite-dimensional modules to obtain
$$\mathcal{F}_{Q} := 
\bigoplus_{\beta \in \cQ^{+}} \mathcal{F}_{Q, \beta}
\colon \Mm_{Q} = \bigoplus_{\beta \in \cQ^{+}} \Mm_{Q, \beta} 
\to \Cc_{Q} = \bigoplus_{\beta \in \cQ^{+}}\Cc_{Q, \beta},$$
where we recall Lemma~\ref{Lem:block}.
By Theorem~\ref{Thm:KKK2} (\ref{Thm:KKK2:monoidal}),
this functor $\mathcal{F}_{Q}$ is monoidal. 
As a corollary of Theorem~\ref{Thm:equiv},
we have the following. 

\begin{Cor}
\label{Cor:equiv}
Under the assumption $(\star)$,
the Kang-Kashiwara-Kim functor gives an equivalence of 
monoidal categories
$
\mathcal{F}_{Q} \colon \Mm_{Q} \simeq
\Cc_{Q}.  
$
\end{Cor}

\end{document}